\documentclass[noamsfonts,a4paper]{amsart}
\usepackage[margin=2cm]{geometry}
\usepackage[bitstream-charter,cal,sfscaled=false]{mathdesign}
\DeclareFontShape{T1}{mdbch}{m}{up}{<->ssub*mdbch/m/n}{}
\usepackage{hyperref}
\usepackage{tikz-cd}
\tikzcdset{arrow style=tikz,diagrams={>={Straight Barb[scale=0.625]}}}

\theoremstyle{plain}
\newtheorem{theorem}{Theorem}[section]
\newtheorem{proposition}[theorem]{Proposition}
\newtheorem{lemma}[theorem]{Lemma}
\newtheorem{corollary}[theorem]{Corollary}
\theoremstyle{definition}
\newtheorem{definition}[theorem]{Definition}

\title[Non-separable Kasparov product]{Analytic Kasparov product for non-separable C$^*$-algebras}
\author{Ikhan Choi}
\address{Graduate School of Mathematical Sciences\\The University of Tokyo\\3-8-1 Komaba Meguro-ku Tokyo 153-8914, Japan}
\email{choi@ms.u-tokyo.ac.jp}

\begin{document}

\begin{abstract}
We develop an analytic framework for the explicit construction of Kasparov products on uncountably generated Hilbert modules over non-separable C$^*$-algebras by extending the classical Kasparov bimodule picture.
In particular, this yields a well-defined functor from the category of general C$^*$-algebras whose restriction to separable C$^*$-algebras coincides with classical KK-theory and which satisfies several properties expected of non-separable KK-theory.
A key idea of the extension is to replace a single Fredholm operator on a Kasparov bimodule with a net and to encode the directed system of separable C$^*$-subalgebras.
We also introduce the notion of freely complementable Hilbert modules and develop cropping techniques to extend Kasparov's technical theorem beyond the countably generated setting.
It naturally gives rise to a sufficiently broad class of C$^*$-algebras that provides a non-commutative analogue of paracompactness and includes all $\sigma$-unital algebras.
\end{abstract}

\maketitle
\tableofcontents

\section{Introduction}

\subsection{Background}
Kasparov's bivariant K-theory, usually called KK-theory and originally constructed by Kasparov in \cite{MR582160}, is a fundamental framework for the classification theory of C$^*$-algebras and Connes' non-commutative geometry.
An abstract and concise definition of the classical KK-theory was obtained by Higson \cite{MR869779}, which characterizes KK-theory as the initial additive functor $\mathrm{C^*Alg}_\mathrm{sep}\to\mathrm{KK}_\mathrm{sep}$ that is homotopy invariant, stable, and split exact, from the category of separable C$^*$-algebras and $*$-homomorphisms to an additive category.
The target category is often called the Kasparov category, and the morphism groups in it generalize both K-theory and K-homology by letting either its domain or codomain be the complex scalar field.
Its universal property supports KK-theory as the most satisfactory realization of a non-commutative additive motive for C$^*$-algebras, and the existence of triangulated structure on the Kasparov category which is now well-known after Meyer-Nest \cite{MR2193334} has been gradually driving a homotopy theoretic approach to bivariant K-theory of C$^*$-algebras.
While many surprising results on those category-theoretic descriptions of KK-theory are being discovered currently by means of homotopy theory, the analytic pictures of KK-theory remain widely used by operator algebraists and non-commutative geometers due to its computability in applications.

The classical Kasparov picture for KK-theory, defined successfully only with separability conditions, has Kasparov bimodules as its analytic cycles to define the morphism groups.
A Kasparov bimodule is defined as a Hilbert bimodule and a Fredholm operator on it, which we call a Kasparov operator in this paper.
The composition in the Kasparov category is concretely given by an analytic construction called the Kasparov product of two Kasparov bimodules, motivated by the index theory.
Recently, many important sources of non-separable C$^*$-algebras in the KK-theory contexts has arisen, including Paschke duality and coarse geometric applications.
The separability assumption of C$^*$-algebras in the classical KK-theory has been regarded as a golden rule to avoid technical issues that are deeply analytic, and the direct construction of general theory for non-separable C$^*$-algebras has hardly been challenged.
One of the main obstructions is that the separability and $\sigma$-unitality are inevitable when we apply the Kasparov technical theorem in the construction of Kasparov product.
There are also many analytic problems other than the technical theorem regarding Kasparov's stabilization and the existence of Grassmann connections for example, and this makes removing separability from KK-theory a more difficult question.

In fact, there have already been two major extensions of KK-theory to non-separable C$^*$-algebras, both of which use categorical constructions rather than finding analytic cycles, based on the elementary fact that every C$^*$-algebra is an inductive limit of separable C$^*$-subalgebras.
According to the shape theory, a standard way to study homotopy theory of topological spaces with uncountably large or locally ill-behaved structures, whose homotopy types cannot be detected by standard simplices, is to interpret a general topological space as a projective system of amenable topological spaces.
In this viewpoint, it is natural to generalize KK-theory to non-separable C$^*$-algebras by inductive limits via Gelfand duality, and the initial definition was by Skandalis in \cite{MR953916}, which is essentially same as the ordinary categorical ind-completion of the Kasparov category.
The other approach is using $\infty$-categorical methods, first introduced by Land-Nikolaus \cite{MR3788857} and investigated more systematically by Bunke and Bunke-D\"unzinger in \cite{MR4895460}, \cite{MR5059731}, and \cite{bunke2026productskketheory}.
They exhibited a stable $\infty$-categorical enhancement of the separable Kasparov category, and defined KK-theory by its $\infty$-categorical ind-completions.
These two approaches are not equivalent at all outside the category of separable C$^*$-algebras.
However, both categorical approaches have a drawback that computations of KK-groups are generally difficult for applications.

\subsection{Contributions}
In this paper, we address an analytic question for non-separable KK-theory.
One central idea of designing cycles for non-separable KK-theory is to attach a net $(F_i)_i$ of Kasparov operators on a Hilbert bimodule, instead of a single Kasparov operator.
This net is required to satisfy a condition of the form ``for every separable C$^*$-subalgebra $A_0\subset A$ there is $i_0$ such that whenever $i\ge i_0$ we have...''.

We also need to deal with Hilbert modules not countably generated, so in order to overcome the technical theorem and the connection existence, we introduce a property for Hilbert modules called the free complementability.
This property lets Hilbert modules behave more likely projective modules, which is automatic for countably generated Hilbert modules by the Kasparov stabilization theorem.
Then, we develop a cropping technique, which crops out a countably generated complemented Hilbert submodule, containing the data of interactions between certain separable C$^*$-subalgebras of the left action and Kasparov operators, from the original large Hilbert module to apply the classical technical theorem.
It allows us in the trivial complementary large part to construct a separating operator manually by a projection to finally state the extended technical theorem.

\begin{definition}[Freely complementable augmented Hilbert bimodules]
Let $A$ and $B$ be C$^*$-algebras.
An \emph{augmented Hilbert bimodule} over $(A,B)$ is a Hilbert module $E$ over the unitization $\widetilde B$ together with a $*$-homomorphism $A\to B(E)$ such that $AE\subset EB$.
It is called \emph{freely complementable} if there is a complemented embedding $E\to H\otimes\widetilde B$ as Hilbert modules over $\widetilde B$ for a Hilbert space $H$.
\end{definition}

\begin{lemma}[Cropping lemma]
Let $E$ be a graded freely complementable Hilbert module over a unital C$^*$-algebra.
For a $\sigma$-unital C$^*$-subalgebra $J_0\subset K(E)$ and a separable C$^*$-subalgebra $A_0\subset B(E)$, there exists a countably generated complemented graded Hilbert submodule $S\subset E$ over $\widetilde B$ such that $J_0\subset K(S)$ and $A_0$ is reduced by $S$.
\end{lemma}

\begin{theorem}[Technical theorem]
Let $B$ and $C$ be C$^*$-algebras.
Let $E_1$ be a graded freely complementable Hilbert module over the standard unitization $\widetilde B$ of $B$, $E_2$ be a graded freely complementable augmented Hilbert bimodule over $(B,C)$, and $E_{12}:=E_1\otimes_BE_2$.
Consider
\begin{enumerate}
\item
a norm compact subset $\Delta_1\subset B(E_1)\otimes1$,
\item
a norm compact subset $\Delta_2\subset B(E_{12})$,
\item
a separable C$^*$-subalgebra $J_1\subset B(E_1)\otimes 1$ which is graded and $J_1E_{12}\subset E_{12}$ is complemented,
\item
a separable C$^*$-subalgebra $J_2\subset B(E_{12})$,
\end{enumerate}
such that
\[[\Delta_1,J_1]\subset J_1,\qquad[\Delta_2,J_1]\subset K(E_{12}),\qquad J_1J_2\subset K(E_{12}).\]
Then, there exists an even positive contraction $M\in B(E_{12})$ such that
\[[\Delta_1\cup\Delta_2,M],\ MJ_1,\ (1-M)J_2\subset K(E_{12}).\]
\end{theorem}

Another direction to make a further generalization of countably generated Hilbert modules is to leave the ``fiber at infinity'' freely.
We define our underlying bimodule over a pair of C$^*$-algebras $(A,B)$ for the cycle as a Hilbert module $E$ over the unitization $\widetilde B$ of $B$, satisfying $AE\subset EB$.
We will call this an augmented Hilbert bimodule, and we can prove the functoriality of the new theory thanks to this.
The ``fiber at infinity'' $E\otimes_{\widetilde B}\mathbb{C}=E/EB$ need not vanish, but it causes subtle obstructions in several situations, as we will discuss in detail.
After defining these new notions, we define an abelian group $KK^\mathrm{fc}(A,B)$, which we suggest to call it a freely complementable KK-group, as the set of equivalence classes of approximate Kasparov bimodules, similarly to the classical KK-theory.

\begin{theorem}
Let $A_0$, $A$, $B$, and $C$ be C$^*$-algebras with $A_0$ separable.
The following statements hold.
\begin{enumerate}
\item
The collection of equivalence classes of approximate Kasparov bimodules over $(A,B)$, denoted by $KK^\mathrm{fc}(A,B)$, is a set and carries a natural abelian group structure.
\item
There is a group homomorphism
\[KK^\mathrm{fc}(A,B)\otimes KK^\mathrm{fc}(B,C)\to KK^\mathrm{fc}(A,C)\]
satisfying the associativity so that it defines an additive category $\mathrm{KK}^\mathrm{fc}$.
\item
There is a canonical functor $\mathrm{C^*Alg}\to\mathrm{KK}^\mathrm{fc}$ from the category of C$^*$-algebras that is homotopy invariant, completely stable, and additive.
\item
If we denote by $KK(A_0,B)$ the set of equivalence classes of classical Kasparov bimodules over $(A_0,B)$, then we have group isomorphisms
\[KK^\mathrm{fc}(A_0,B)\cong KK(A_0,B)\cong\operatorname{colim}_{B_0\subset_\mathrm{sep}B}KK(A_0,B_0).\]
In particular, it extends the classical KK-theory of separable C$^*$-algebras.
\end{enumerate}
\end{theorem}

Other than the universal property, KK-theory has many other interesting and convenient computational devices such as the six-term exact sequences and the external tensor products, and it is natural to ask if freely complementable KK-theory also has these.
Unfortunately it seems to be impossible to prove in a full generality due to the free complementability condition, nevertheless we can obtain partial results by introducing a suitable class of C$^*$-algebras called para-unital C$^*$-algebras, which is sufficiently broad in the sense that it includes all $\sigma$-unital C$^*$-algebras and the algebra $K(H)$ of compact operators on any non-separable Hilbert space $H$.

\begin{definition}[Para-unital C$^*$-algebras]
Let $A$ be a C$^*$-algebra.
A projection $p\in M(A)$ in the multiplier algebra of $A$ is called \emph{$\sigma$-unital} if the corner $pAp$ is $\sigma$-unital.
We say $A$ is \emph{para-unital} if it admits a mutually orthogonal family $p_i\in M(A)$ of $\sigma$-unital projections whose sum $\sum_ip_i$ strictly converges to the unit of $M(A)$.
\end{definition}

The name comes from the fact that a commutative C$^*$-algebra is para-unital if and only if its spectrum is paracompact, and it follows an analogy of the duality that a commutative C$^*$-algebra is unital or $\sigma$-unital if and only if its spectrum is compact or $\sigma$-compact, respectively.
This property is in fact equivalent to being a freely complementable Hilbert module over its one-point unitization.
The free complementability and para-unitality provide an operator algebraic explanation for the reason why paracompact spaces are useful in bundle theory, and reveal several ill-behaved structures of generally uncountably generated Hilbert modules and non-separable C$^*$-algebras.
Now then, we can summarize our next main results as follows.
The rigorous and precise statements of the two main theorems will be given in the paper.

\begin{theorem}
Let $\mathrm{C^*Alg}_\mathrm{para}\to\mathrm{KK}_\mathrm{para}^\mathrm{fc}$ be the restriction of the functor in Theorem 1.4 (3) to para-unital C$^*$-algebras.
\begin{enumerate}
\item
It is strict symmetric monoidal onto its image, with respect to either the minimal or the maximal tensor product.
\item
It is split exact, and completely additive in the first argument.
\end{enumerate}
\end{theorem}

\subsection{Future directions}

The first natural question is on how to remove the para-unital conditions.
The split exactness of the classical KK-theory is strongly related to the fact that the natural $*$-homomorphism $qA\to A$ is a KK-equivalence for separable $A$.
However, as we will see in the non-para-unitality result of $qA$ in Proposition \ref{Proposition: Non-para-unital Cuntz construction}, it seems impossible to obtain the full split exactness including non-para-unital ideals if we want to keep fixing a single underlying Hilbert module for a cycle.
Rather than considering a net of Hilbert modules, it might be more realistic to develop a theory of para-unital or $\sigma$-unital replacement of split or semi-split exact sequences.

The most important problem on non-separable KK-theory would be the construction of a comparison functor from $\mathrm{KK}^\mathrm{fc}$ to an $\infty$-categorically enhanced KK-theory, at least for the restriction on the category of para-unital C$^*$-algebras.
Our theory only has the functors $\mathrm{KK}^\mathrm{fc}\to\mathrm{KK}^\mathrm{Ska}\leftarrow\mathrm{KK}$, where $\mathrm{KK}^\mathrm{Ska}$ is just the ordinary categorical ind-completion of the separable Kasparov category, and $\mathrm{KK}$ is the homotopy category of $\infty$-categorical ind-completion in the sense of either Land-Nikolaus \cite{MR3788857} or Bunke-D\"unzinger \cite{bunke2026productskketheory}.
This question basically asks a universal property problem for $\mathrm{KK}^\mathrm{fc}$.
The current definition of approximate Kasparov bimodule mimics indeed a topological model of the homotopy limit for the space of Kasparov operators in special cases, but it essentially lacks the homotopy coherence data.

We also expect several potential applications of the research in the sequel, including equivariant theory for non-$\sigma$-compact locally compact groupoids, and categorical explanation of KK-theory with real coefficients using the colimits along II$_1$ factors of Antonini-Azzali-Skandalis \cite{MR3419768}.
For example in the equivariant KK-theory, because the $\sigma$-compactness of structure groupoid is necessary due to the same reason as the separability of the first variable in the technical theorem, one of the further questions we may ask would be whether the idea of approximate cycles can be applied to construct the equivariant KK-theory for general locally compact group or groupoids.

This paper was personally and initially motivated by the successful extension of Tomita-Takesaki theory to non-separable von Neumann algebras using weights and Hilbert algebras, with no direct relations with KK-theory.
The objectives together with guiding principles of this paper are in line with Wulff's work on E-theory \cite{wulff2022generalizedasymptoticalgebrasmathrmetheory}.

As the roadmap of this paper, after the preliminaries on Hilbert modules and the classical KK-theory in Section 2, we develop augmented Hilbert bimodules and freely complementable Hilbert modules in Section 3 and 4.
In Section 5, we define approximate Kasparov bimodules and prove reduction theorems with respect to the homotopy relation.
Section 6 constructs the analytic Kasparov product, and Section 7 is devoted to prove desirable properties of KK-theory.

\subsection*{Notations}
Throughout this paper, we impose no separability, $\sigma$-unitality, or countable generation assumption unless explicitly stated.
The unitization of a C$^*$-algebra $B$ is a unital C$^*$-algebra, denoted by $\widetilde B$ or $B^\sim$, together with a split exact sequence $0\to B\to\widetilde B\to\mathbb{C}\to0$.
This definition is same even if $B$ is unital.
We only consider the complex scalars in any linear structures.

For a Hilbert module $E$ over a C$^*$-algebra $B$, the C$^*$-algebra of compact operators and adjointable operators will be denoted by $K(E)$ and $B(E)$, and we will write $Q(E):=B(E)/K(E)$ for its Calkin algebra.
We often omit the C$^*$-algebra $B$ in this notation because it does not affect the structure of these algebras.

For a $*$-homomorphism $\varphi:A\to B$ between C$^*$-algebras, we will use the same notation $\varphi:=(\widetilde B,0)$ to denote the canonically associated Kasparov bimodule, but it might mean $\varphi:=(B,0)$ only when $B$ is para-unital, which will be precisely given depending on its context.
We will also use greek letters except $\xi,\eta,\zeta$ for approximate Kasparov bimodules although they are not associated with $*$-homomorphisms, for example as $\varphi:=(E_\varphi,F_\varphi)$ over $(A,B)$.
We also sometimes use the same notation to denote the left action $\varphi:A\to B(E_\varphi)$ if it has no ambiguity, but this notation will be basically omitted such as $F_\varphi a:=F_\varphi\varphi(a)\in B(E_\varphi)$ or $a\xi:=\varphi(a)\xi$ for $a\in A$ and $\xi\in E_\varphi$.
Their homotopy classes will be denoted by brackets as $[\varphi]$.

For a Hilbert bimodules $E_1$ and $E_2$ over pairs of C$^*$-algebras $(A,B)$ and $(B,C)$ respectively, $T_{\xi_1}\in B(E_2,E_{12})$ for $\xi_1\in E_1$ is defined by $T_{\xi_1}\zeta_2:=\xi_1\otimes\zeta_2$, whose adjoint $T_{\xi_1}^*\in B(E_{12},E_2)$ is given by $T_{\xi_1}^*(\zeta_1\otimes\zeta_2)=\langle\xi_1,\zeta_1\rangle\zeta_2$, and we also write $\Theta_{\xi_1,\eta_1}\in K(E_1)$ for $\eta_1\in E_1$, where $\Theta_{\xi_1,\eta_1}\zeta_1:=\xi_1\langle\eta_1,\zeta_1\rangle$, for all $\zeta_1\in E_1$ and $\zeta_2\in E_2$, with $\Theta_{\xi_1}:=\Theta_{\xi_1,\xi_1}$.

Every grading will be a $\mathbb{Z}/2\mathbb{Z}$-grading, and every C$^*$-algebra will be assumed ungraded, except the C$^*$-algebras associated with graded Hilbert modules.
The grading operator that transforms odd part to its additive inverse will be denoted by $\gamma$ or the overline.
When we want to explicitly write down the even and odd parts of a graded module $E$, we will use a boxed direct sum $E_+\boxplus E_-$ to mean that its left side $E_+$ is even and its right side $E_-$ is odd.
The parity $\pm$ will be at subscript.
The matrix notation with parentheses means ordinary categorical direct sums, but bracketed matrix will be always for direct sum decompositions by the grading.
The commutator $[a,b]$ for $a,b$ in a graded algebra will always mean the graded commutator.

\subsection*{AI disclosure statement}

The research direction, the definitions of augmented Hilbert bimodules, freely complementable Hilbert modules, approximate Kasparov bimodules, and the initial strategy for constructing the Kasparov product were developed by the author.
GPT-5.6 and GPT-6.0 were used for language editing and mathematical assistance, including identifying gaps and suggesting repairs, simplifying the original arguments especially in Section 6.3, and investigating the necessity of para-unitality assumptions in Section 7.
In particular, other than the construction of Kasparov products, substantial parts of the proofs of Lemma \ref{Lemma: Non-zero-divisor property}, Proposition \ref{Proposition: Non-para-unital Cuntz construction}, Lemma \ref{Lemma: Compact operators on external tensor products}, and Theorem \ref{Theorem: General additive inverses} originated in suggestions generated by GPT-5.6, and Lemma \ref{Lemma: Free complementability of gluing}, Lemma \ref{Lemma: Universally Lipschitz Kasparov family} by GPT-6.0.
AI assistance also contributed to the revision of arguments elsewhere in the paper.
The author has checked all arguments presented in this paper and takes full responsibility for its mathematical correctness.

\section*{Acknowledgement}
The author is deeply grateful to his advisor Yasuyuki Kawahigashi for his continual guidance and generous support throughout the author's doctoral studies.
The author thanks Shintaro Nishikawa, Yuezhao Li, and Yosuke Kubota for fruitful discussions with kind encouragements.
This work was supported by JSPS KAKENHI Grant Number JP26KJ0751.

\section{Preliminaries}

\subsection{Hilbert modules}

A Hilbert module over a possibly non-unital C$^*$-algebra $B$ is a right module $E$ over $B$ together with a sesqui-linear map $\langle\cdot,\cdot\rangle:E\times E\to B$ satisfying $\langle\xi,\xi\rangle\ge0$ but $|\xi|^2:=\langle\xi,\xi\rangle=0$ implies $\xi=0$, $\langle\eta,\xi b\rangle=\langle\eta,\xi\rangle b$, $\langle\eta,\xi\rangle^*=\langle\xi,\eta\rangle$, for all $\xi,\eta\in E$ and $b\in\widetilde B$, and finally the norm given by $\|\xi\|:=\|\langle\xi,\xi\rangle\|^{\frac12}$ for $\xi\in E$ is complete.
A prototypical example is a Hermitian bundle over a compact Hausdorff space, but a Hilbert module can have a fiber of infinite rank and model a vector bundle that may not be locally trivial in general.
Whenever $B$ is included in a larger C$^*$-algebra $B'$ as a closed ideal, every Hilbert module over $B$ can be naturally considered as a Hilbert module over $B'$.
In this sense we can equivalently describe a Hilbert module $E$ over $B$ just as a Hilbert module $E$ over $\widetilde B$ such that $E\subset\overline{EB}$, which means ``vanishing at infinity'' when we regard the unitization as adjoining an infinity.

Let $E$ be a Hilbert module over a C$^*$-algebra $B$.
Then, we can associate a C$^*$-algebra $K(E)$, generated by the operators of the form $\Theta_{\xi,\eta}:E\to E$ such that $\Theta_{\xi,\eta}\zeta:=\xi\langle\eta,\zeta\rangle$ for $\xi,\eta,\zeta\in E$, in which elements are called compact operators.
Its multiplier algebra $B(E):=M(K(E))$ can be identified with the algebra of linear operators $T:E\to E$ called adjointable operators, meaning that it admits another operator $T^*:E\to E$ such that $\langle\eta,T\xi\rangle=\langle T^*\eta,\xi\rangle$ for all $\xi,\eta\in E$.
Since the structures of $K(E)$ and $B(E)$ do not depend on the C$^*$-algebra of the right action, we will not write as $K_B(E)$ nor $B_B(E)$.

For a C$^*$-algebra $K$, there is a natural topology on the multiplier algebra $M(K)$ called the strict topology, which generated by semi-norms of the form $M(K)\to\mathbb{R}_{\ge0}:x\mapsto\|xa\|+\|ax\|$ with $a\in K$.
The norm closed unit ball of $M(K)$ and $M(K)$ itself can be seen as the completion of the closed unit ball of $K$ and $K$ itself with respect to the strict topology, and any linear operator from $M(K)$ strictly continuous on the unit ball is strictly continuous by \cite{MR290117}.
If $K$ is a closed ideal of another C$^*$-algebra $L$, then the multiplier algebras naturally reverse the direction to have a $*$-homomorphism $M(L)\to M(K)$.
However, the inclusion $K\to M(L)$ is rarely extended to $M(K)\to M(L)$ strictly continuously.

A Hilbert bimodule or a correspondence over a pair of C$^*$-algebras $(A,B)$ is a Hilbert module $E$ over $B$ equipped with a $*$-homomorphism $\varphi:A\to B(E)$ called the left action.
Even if $\varphi$ is not faithful, we will just write $a\xi:=\varphi(a)\xi\in E$ for $a\in A$ and $\xi\in E$.
By the Cohen factorization theorem, we have $AE=\overline{AE}$ and $EB=\overline{EB}$.
We always have $K(AE)=AK(E)A$, and there is a one-to-one correspondence between Hilbert submodules of $E$ and hereditary C$^*$-subalgebras of $K(E)$.
A Hilbert submodule $S\subset E$ is called complemented if $E=S\oplus\{\eta\in E:\langle\eta,\xi\rangle=0\text{ for }\xi\in S\}$.
In this case we can find a projection $P\in B(E)$ such that $S=PE$.
The left action $A\to B(E)$ is strictly continuous in the sense that it is extended to $M(A)\to B(E)$ strictly continuously if and only if $AE$ is complemented in $E$.
In particular, the Hilbert bimodule $E$ is non-degenerate, i.e.~$E=AE$, then we naturally have a unital strictly continuous $*$-homomorphism $M(A)\to B(E)$.
In general, we have neither canonical choice of $B(AE)\to B(E)$ nor $B(E)\to B(AE)$ unless $AE\subset E$ is complemented, so the complementedness of submodules becomes important when we induce adjointable operators on Hilbert submodules.

A Hilbert module $E$ over a C$^*$-algebra $B$ is called countably generated if there is a countable subset $E_0\subset E$ such that $E=\overline{E_0B}$.
We may replace $E_0$ to be a separable closed linear subspace in its definition.
It is well-known that $E$ is countably generated if and only $K(E)$ is $\sigma$-unital, and this implies that countable generatedness does not depend on the C$^*$-algebra of right action, namely $E$ is countably generated over $B$ if and only if over $\widetilde B$.
The Kasparov stabilization theorem states that if $E$ is countably generated, then we have an unitary equivalence $E\oplus\ell^2(\mathbb{N})\otimes B\cong\ell^2(\mathbb{N})\otimes B$, which is frequently called the standard Hilbert module over $B$.
In the classical KK-theory, it is important to see that a $\sigma$-unital C$^*$-algebra $B$ can be identified with a complemented Hilbert submodule of $\ell^2(\mathbb{N})\otimes\widetilde B$ over $\widetilde B$.

If we have Hilbert bimodules $E_1$ and $E_2$ over pairs of C$^*$-algebras $(A,B)$ and $(B,C)$ respectively, the internal tensor product $E_1\otimes_BE_2$ is defined by the ``separation and completion'' of the algebraic tensor product $E_1\odot E_2$ with respect to a $C$-valued sesqui-linear form defined by $\langle\eta_1\otimes\eta_2,\xi_1\otimes\xi_2\rangle:=\langle\eta_2,\langle\eta_1,\xi_1\rangle\xi_2\rangle$ for $\xi_1,\eta_1\in E_1$ and $\xi_2,\eta_2\in E_2$.
We have $\xi_1b\otimes\xi_2=\xi_1\otimes b\xi_2$ in $E_1\otimes_BE_2$ for all $\xi_1\in E_1$, $\xi_2\in E_2$, and $b\in B$.
We always have a strictly continuous unital $*$-homomorphism $B(E_1)\to B(E_1\otimes_BE_2)$ so that $E_1\otimes_BE_2$ is a Hilbert bimodule over $(A,C)$.
However, there is not canonical choice of an operator on $E_1\otimes_BE_2$ for a given adjointable operator on $E_2$ because the left action of $B$ does not commute with the operator.
Observing the left action $B\to B(E_2)$ always has a unital extension $\widetilde B\to B(E_2)$, we can see $E_1\otimes_BE_2=E_1\otimes_{\widetilde B}E_2$.
We can consider a locally large category of C$^*$-algebras with non-degenerate Hilbert bimodules as its morphisms, where the composition is the internal tensor product.

Fix $*\in\{\mathrm{min},\mathrm{max}\}$.
The external tensor product of Hilbert bimodules $E_1$ and $E_2$ over pairs of C$^*$-algebras $(A_1,B_1)$ and $(A_2,B_2)$ is the Hilbert bimodule $E_1\otimes_*E_2$ over $(A_1\otimes_*A_2,B_1\otimes_*B_2)$ defined by the completion of the algebraic tensor product $E_1\odot E_2$ with respect to a $B_1\otimes_*B_2$-valued sesqui-linear form given by $\langle\eta_1\otimes\eta_2,\xi_1\otimes\xi_2\rangle:=\langle\eta_1,\xi_1\rangle\otimes\langle\eta_2,\xi_2\rangle$, which is in fact non-degenerate.

Let $X$ be a locally compact Hausdorff space, and let $B$ be a C$^*$-algebra.
Since $C_0(X)$ is a nuclear C$^*$-algebra, there is only one C$^*$-tensor norm on the algebraic tensor product $C_0(X)\odot B$ so that we can identify $C_0(X)\otimes B=C_0(X,B)$ as the completion.
For a Hilbert module $E$ over $C_0(X,B)$ and $x\in X$, the internal tensor product by the evaluation $C_0(X,B)\to B:b\mapsto b(x)$ from right gives a Hilbert module $E_x:=E\otimes_{C_0(X,B)}B=E/C_0(X\setminus\{x\},B)E$ over $B$, which is called the fiber of $E$ at $x$.

Let $E_0$ be a Hilbert module over $B$.
By the similar argument as above using the realizations as the completions of algebraic tensor products, we can identify $C_0(X)\otimes E_0=C_0(X,E_0)$ the exterior tensor product over $C_0(X,B)$, and $K(C_0(X,E_0))=C_0(X,K(E_0))$ with $B(C_0(X,E_0))=C_b(X,B(E_0))$, where $B(E_0)$ is endowed with the strict topology.
This is an extremely useful picture when we want to construct an explicit adjointable operator on a Hilbert module of the form $C_0(X,E_0)$ by a continuous function of adjointable operators.
For $E_1\subset E_0$ a Hilbert submodule over $B$, we can construct $E:=\{\xi\in C([0,1],E_0):\xi(1)\in E_1\}$ a Hilbert module over $C([0,1],B)$, and it is hard to describe $B(E)$ in a closed form.
If a strictly continuous path $[0,1]\to B(E_0):t\mapsto F_t$ satisfies $C^*(F_t)E_1\subset E_1$ for all $t\in[0,1]$, then it defines $F\in B(E)$ in a natural way.

\subsection{Classical KK-theory}

We briefly review the Kasparov picture of the classical KK-theory by Kasparov bimodules.

\begin{definition}[Classical Kasparov bimodules]
Let $A_0$ and $B$ be C$^*$-algebras.
A \emph{classical Kasparov bimodule} over $(A_0,B)$ is a graded countably generated Hilbert module $E$ over $B$ together with a $*$-homomorphism $A_0\to B(E)$ and an odd adjointable operator $F\in B(E)$ such that $[F,a],(F-F^*)a,(F^2-1)a\in K(E)$ for all $a\in A_0$, called a \emph{Kasparov operator} or a \emph{Fredholm operator}.
\end{definition}

The operator $F$ of a classical Kasparov bimodule $(E,F)$ is often called a Fredholm operator, but in this paper we call it a Kasparov operator.

\begin{definition}[Homotopy relation]
Let $A_0$ and $B$ be C$^*$-algebras.
Two classical Kasparov bimodules $(E_0,F_0)$ and $(E_1,F_1)$ over $(A_0,B)$ are called \emph{homotopic} if there is a classical Kasparov bimodule $(E,F)$ over $(A_0,C([0,1],B))$ such that $(E(0),F(0))=(E_0,F_0)$ and $(E(1),F(1))=(E_1,F_1)$, and in this case we say $(E,F)$ is a \emph{homotopy} that \emph{connects} $(E_0,F_0)$ and $(E_1,F_1)$.
An \emph{operator homotopy} connecting $(E_0,F_0)$ and $(E_1,F_1)$ under $A_0$, only defined when $E_0=E_1$, is a norm continuous path $F\in C([0,1],B(E))$ of Kasparov operators connecting $F_0$ and $F_1$ in the sense that $F(0)=F_0$ and $F(1)=F_1$.
An operator homotopy obviously define a classical Kasparov bimodule over $(A_0,C([0,1],B))$.
\end{definition}

To define Kasparov products, we need to define the following conditions for a Kasparov operator on the internal tensor product.
These two notions are repeatedly used through this paper.

\begin{definition}[Positivity and connection conditions]
Let $A$, $B$, and $C$ be C$^*$-algebras.
Let $E_1$ and $E_2$ be graded Hilbert bimodules over $(A,B)$ and $(B,C)$ respectively, and let $E_{12}:=E_1\otimes_BE_2$.
Let $A_0\subset A$ be a C$^*$-subalgebra, and let $E_{10}\subset E_1$ be a closed linear subspace.
Let $F_1\in B(E_1)$ and $F_2\in B(E_2)$ be odd adjointable operators.
For an odd $F_{12}\in B(E_{12})$, we say it is \emph{$F_1$-positive} under $A_0$ if $a^*[F_{12},F_1\otimes1]a\ge0$ in $Q(E_{12})$ for all $a\in A_0$, and is an \emph{$F_2$-connection} on $E_{10}$ if $(F_{12}T_{\xi_1}-T_{\overline\xi_1}F_2)T_{\eta_1}^*,\ (F_{12}^*T_{\xi_1}-T_{\overline\xi_1}F_2^*)T_{\eta_1}^*=0$ in $Q(E_{12})$ for all $\xi_1,\eta_1\in E_{10}$.
\end{definition}

\begin{definition}[Classical Kasparov product]
Let $A_0$, $B_0$, and $C$ be C$^*$-algebras.
Let $(E_1,F_1)$ and $(E_2,F_2)$ be classical Kasparov bimodules over $(A_0,B_0)$ and $(B_0,C)$ respectively.
A \emph{classical Kasparov product} of them is a classical Kasparov bimodule $(E_{12},F_{12})$ over $(A_0,C)$ such that $E_{12}=E_1\otimes_{B_0}E_2$ and $F_{12}$ is an $F_1$-positive $F_2$-connection under $A_0$ and on $E_1$.
\end{definition}

Connections are used to resolve the problem that $1\otimes F_2$ is not well-defined in $B(E_{12})$ for $F_2\in B(E_2)$, which is due to the fact $F_2$ generally does not commute with the left action of $B_0$, and they have original motivations from the connections of differential geometry in the unbounded KK-theory context.
It is not difficult to show the connection condition $(F_{12}T_{\xi_1}-T_{\overline\xi_1}F_2)T_{\eta_1}^*\in K(E_{12})$ is equivalent to the classical definition $F_{12}T_{\xi_1}-T_{\overline\xi_1}F_2\in K(E_2,E_{12})$ if $[F_2,b]\in K(E_2)$ for $b\in\langle E_{10},E_{10}\rangle=C^*(T_{\eta_1}^*T_{\zeta_1}:\eta_1,\zeta_1\in E_{10})$ by taking an approximate unit of it.
On the other hand, the technical condition on positivity for the Kasparov product is in a strong relation with the following theorem.

\begin{theorem}[Classical Connes-Skandalis operator homotopy]
Let $A_0$ and $B$ be C$^*$-algebras.
Let $(E,F_0)$ and $(E,F_1)$ be classical Kasparov bimodules over $(A_0,B)$.
If $a^*[F_0,F_1]a\ge0$ in $Q(E)$ for all $a\in A_0$, then $F_0$ and $F_1$ are connected by a norm continuous path $[0,1]\to B(E):t\mapsto F_t$ such that $(E,F_t)$ is a classical Kasparov bimodule over $(A_0,B)$ for all $t\in[0,1]$.
\end{theorem}

A separation theorem of Kasparov, usually called the technical theorem, is the most technical part of the construction of Kasparov product as the name tells it.

\begin{theorem}[Classical technical theorem]
Let $K$ be a graded $\sigma$-unital C$^*$algebra.
Consider
\begin{enumerate}
\item a norm compact subset $\Delta\subset M(K)$,
\item a $\sigma$-unital C$^*$-subalgebra $J_1\subset M(K)$ that is graded,
\item a $\sigma$-unital C$^*$-subalgebra $J_2\subset M(K)$,
\end{enumerate}
such that $[\Delta,J_1]\subset J_1$ and $J_1J_2\subset K$.
Then, there is an even positive contraction $M\in M(K)$ such that 
\[[\Delta,M],\ MJ_1,\ (1-M)J_2\subset K.\]
\end{theorem}

More precisely speaking with the statement of the technical theorem for $K=K(E_{12})$, one can say that there are two points where the separability of C$^*$-algebras in classical KK-theory becomes necessary.
One reason for the separability is that the closed linear span of $\Delta$ must include $A$, and the other reason is that $J_2$ must include $[\widehat F_2,a]$ for all $a\in A$, where $\widehat F_2\in B(E_{12})$ is an $F_2$-connection for $F_2\in B(E_2)$.
Note that the other algebras $B$ and $C$ are not technically required to be separable.

\section{Augmented Hilbert bimodules}

In this section, we define augmented Hilbert bimodules, and investigate their basic categorical properties.
The most important technical issue arising in the analytic construction of KK-theory for general C$^*$-algebras is the separability condition required in the technical theorem of Kasparov.
However, there is another critical and elementary issue not very being talked about, the functoriality, and augmented Hilbert bimodules are introduced in order to resolve this.
In the classical picture of KK-theory, for $\sigma$-unital C$^*$-algebras $A$ and $B$, the canonical Kasparov bimodule induced by a $*$-homomorphism $A\to B$ has the underlying Hilbert module $B$ or $AB$ over $B$ without odd part, where the $*$-homomorphism determines the left action on it.
Without $\sigma$-unitality, this Hilbert module need not be countably generated in general.
Our solution is to set the underlying Hilbert module as the unitization $\widetilde B$, not $B$ itself.
This essentially enables us to study Hilbert modules by setting the C$^*$-algebra of right action to be always unital, but the ``fiber at infinity'' $E\otimes_{\widetilde B}\mathbb{C}$ of a Hilbert module $E$ over $\widetilde B$ has no need to collapse to be zero.
It seems natural to focus on Hilbert modules over unitizations because the module theory over non-unital rings is also known to be notoriously sophisticated even in algebras without topological structures.

\subsection{Bimodule categories of C\texorpdfstring{$^*$}{*}-algebras}
There are also categorical and geometric justifications for the introduction of augmented Hilbert bimodules.
For example, augmented Hilbert bimodules provide the most natural setting for a commutative monoid enriched category of C$^*$-algebras which allows Hilbert bimodules with degenerate left actions as morphisms.
The problem is that degenerate left actions conflict the identity axiom for a category, as the internal tensor product of classical Hilbert bimodules $A$ over $(A,A)$ and $E$ over $(A,B)$ is $A\otimes_AE=AE\ne E$ in general, for C$^*$-algebras $A$ and $B$.
Hilbert bimodules without non-degeneracy of left actions do not even form a category.

Here is one remark that the augmentation of Hilbert bimodules is not technically needed if one only wants to care about para-unital or $\sigma$-unital C$^*$-algebras.
This is because a KK-cycle will be defined only on freely complementable Hilbert modules, and as we will see later in Theorem \ref{Theorem: Non-degenerate reduction}, one can always assume the underlying Hilbert bimodules are non-degenerate over para-unital C$^*$-algebras, which resolves the functoriality issue without augmentation.
Nevertheless, degenerate left actions are extremely useful when we utilize the calculus of degenerate Kasparov bimodules in constructions of homotopies.
Furthermore, augmented Hilbert bimodules not only motivates the definition of freely complementable Hilbert modules in the next section, but also makes possible to construct the Kasparov product including general non-para-unital C$^*$-algebras.

From the geometric side, via the Gelfand duality, recall that a C$^*$-algebra can be regarded as a non-commutative generalization of either a pointed compact Hausdorff space or a locally compact Hausdorff space.
Denote by $\mathrm{C^*Alg}$ and $\mathrm{C^*Alg}^\mathrm{ndg}$ the categories defined such that objects are C$^*$-algebras and morphisms are given by
\begin{align*}
\mathrm{Hom}_\mathrm{C^*Alg}(A,B)&:=\{\text{$*$-homomorphisms }\varphi:A\to B\},\\
\mathrm{Hom}_{\mathrm{C^*Alg}^\mathrm{ndg}}(A,B)&:=\{\text{non-degenerate $*$-homomorphisms }\varphi:A\to M(B)\}
\end{align*}
for C$^*$-algebras $A$ and $B$.
Then, the full subcategories of commutative C$^*$-algebras of the first category $\mathrm{C^*Alg}$ and the second category $\mathrm{C^*Alg}^\mathrm{ndg}$ are equivalent after taking opposites to the category of pointed compact Hausdorff spaces and the category of locally compact Hausdorff spaces, respectively.
The situation becomes more clear if we see that the unitization and forgetful functors provide an adjunction pair $\mathrm{C^*Alg}\leftrightarrows\mathrm{C^*Alg}^\mathrm{unit}\subset\mathrm{C^*Alg}^\mathrm{ndg}$, where $\mathrm{C^*Alg}^\mathrm{unit}$ denotes the category of unital C$^*$-algebras and unital $*$-homomorphisms.
The category $\mathrm{C^*Alg}$ usually has more benefits than $\mathrm{C^*Alg}^\mathrm{ndg}$ in operator K-theory, but the latter is a more natural setting for the study of dynamical systems.
It is well-known that the two categories share the same isomorphisms.
The following definitions of non-degenerate and augmented Hilbert bimodules provide the bimodule enhancements of the categories $\mathrm{C^*Alg}^\mathrm{ndg}$ and $\mathrm{C^*Alg}$ respectively.

\begin{definition}[Non-degenerate and augmented Hilbert bimodules]
Let $A$ and $B$ be C$^*$-algebras.
A \emph{non-degenerate Hilbert bimodule} or a \emph{non-degenerate correspondence} over $(A,B)$ is a Hilbert module $E$ over $B$ together with a non-degenerate $*$-homomorphism $A\to B(E)$.
An \emph{augmented Hilbert bimodule} or an \emph{augmented correspondence} over $(A,B)$ is a Hilbert module $E$ over $\widetilde B$ together with a $*$-homomorphism $A\to B(E)$ such that $A$ vanishes on the fiber of $E$ at infinity in the sense that $AE\subset EB$.
\end{definition}

Let $E_1$ and $E_2$ be augmented Hilbert bimodules over pairs of C$^*$-algebras $(A,B)$ and $(B,C)$ respectively.
To define an internal tensor product, since $E_1$ is a Hilbert module over $\widetilde B$, we use the unital extension $\widetilde B\to B(E_2)$ of the left action $B\to B(E_2)$.
If $E_1\subset E_1B$, then $E_1\otimes_BE_2=E_1\otimes_{\widetilde B}E_2$, but even for the case $E_1\not\subset E_1B$, we will use the notation
\[E_1\otimes_BE_2:=E_1\otimes_{\widetilde B}E_2.\]
Then, by letting the identity augmented Hilbert bimodule over $(A,A)$ be $E_{\mathrm{id}_A}:=\widetilde A$ with the left action $A\to B(E_{\mathrm{id}_A})\cong\widetilde A$ given by the inclusion, we can verify $E_{\mathrm{id}_A}\otimes_AE_1=E_1$ as an augmented Hilbert bimodule.

Now it is immediate to check that the internal tensor product satisfies the axioms of categories, so that we can define $\mathrm{Corr}^\mathrm{ndg}$ and $\mathrm{Corr}^\mathrm{aug}$ be categories of C$^*$-algebras, with morphisms as unitary equivalence classes of non-degenerate and augmented Hilbert bimodules respectively.
Because we are not interested in bicategories but only in ordinary categorical structures, it is important to observe that a morphism class consists of a unitary equivalence class of certain Hilbert bimodules instead of a single Hilbert bimodule.
The direct sum of Hilbert bimodules gives the two categories the enrichment over (large) commutative monoids.
The category $\mathrm{Corr}^\mathrm{ndg}$ has a zero object so that it is semi-additive, but $\mathrm{Corr}^\mathrm{aug}$ does not.
Note that they are not locally small.
One of the most interesting and mysterious part of KK-theory that does not appear in K-theory is the role of $\mathrm{Corr}^\mathrm{ndg}$, for example, Morita equivalent separable C$^*$-algebras are KK-equivalent.
This indicates that KK-theory and Hilbert bimodule theory actually mix the two intuitive pictures of C$^*$-algebras: as non-commutative counterparts of pointed compact spaces and locally compact spaces.
It would be helpful to draw in one's mind functorial relations between the four categories of C$^*$-algebras that can be depicted as the following diagram.
\[\begin{tikzcd}
\mathrm{C^*Alg}\dar&\mathrm{C^*Alg}^\mathrm{ndg}\dar\\
\mathrm{Corr}^\mathrm{aug}\rar&\mathrm{Corr}^\mathrm{ndg}
\end{tikzcd}\]
The horizontal arrow $\mathrm{Corr}^\mathrm{aug}\to\mathrm{Corr}^\mathrm{ndg}$ is the ``non-degenerate reduction'' functor $E\mapsto AE$ for a Hilbert bimodule $E$ over a pair of C$^*$-algebras $(A,B)$.
This can be considered as a left inverse but not a right inverse of the inclusion $\mathrm{Corr}^\mathrm{ndg}\subset\mathrm{Corr}^\mathrm{aug}$, which is not a functor because we set the different identity morphisms in each category.

\subsection{Failure of external tensor products}
External tensor products of Hilbert bimodules define a monoidal category structure on $\mathrm{Corr}^\mathrm{ndg}$ with respect to the either minimal or maximal tensor product of C$^*$-algebras.
Unfortunately, in the category $\mathrm{Corr}^\mathrm{aug}$, on the contrary to $\mathrm{Corr}^\mathrm{ndg}$, further from the lack of zero object, there exists a truly major technical defect that $\mathrm{Corr}^\mathrm{aug}$ is not compatible with external tensor products.
Trying to define an external tensor product of augmented Hilbert bimodules is similar to construct an external tensor product of vector bundles over pointed compact Hausdorff spaces on their smash product, which is absurd.
If we let $E_1$ and $E_2$ be augmented Hilbert bimodules over $(A_1,B_1)$ and $(A_2,B_2)$, then C$^*$-algebraically the obstruction comes from the fact that there is no natural $*$-homomorphism from $\widetilde B_1\otimes\widetilde B_2$ to $(B_1\otimes B_2)^\sim$ whose restriction to $B_1\otimes B_2$ is the canonical inclusion, where the tensor product is either maximal or minimal.

This fallback indeed causes the failure of the standard proof of the Bott periodicity and its consequences on exact sequences when we try to construct KK-theory with augmented Hilbert bimodules.
Most importantly, we cannot naively construct an operator homotopy by a norm continuous path of adjointable operators in an augmented Hilbert bimodule $E$ over a pair $(A,B)$ of C$^*$-algebras, because $C([0,1],E)$ is not a Hilbert module over $C([0,1],B)^\sim$, but over $C([0,1],\widetilde B)$.
To resolve this, we define an ad-hoc cylinder object for augmented Hilbert bimodules such that if $E\subset EB$ then its cylinder is same as the classical case, $C([0,1],E)$, by the internal tensor product.

\begin{definition}[Cylinder of augmented Hilbert bimodules]
Let $A$ and $B$ be C$^*$-algebras.
Let $E$ be an augmented Hilbert bimodule over $(A,B)$.
We define the \emph{cylinder} of $E$ as an augmented Hilbert bimodule
\[\mathrm{cyl}\,E:=E\otimes_BC([0,1],B)^\sim=\{\xi\in C([0,1],E):\xi\otimes1\in C([0,1],E\otimes_B\mathbb{C})\text{ is constant along }[0,1]\}\]
over $(A,C([0,1],B))$.
\end{definition}

\begin{proposition}
Let $A$ and $B$ be C$^*$-algebras.
Let $E$ be an augmented Hilbert bimodule over $(A,B)$.
If $[0,1]\to B(E):t\mapsto F_t$ is a strictly continuous bounded path such that $F_tE\subset EB$ for all $t\in[0,1]$, then there is $F\in B(\mathrm{cyl}\,E)$ such that $F(t)=F_t$ for all $t\in[0,1]$ and $C^*(F)(\mathrm{cyl}\, E)\subset(\mathrm{cyl}\,E)C([0,1],B)$.
If we further have $F_t\in K(E)$ for each $t\in[0,1]$ and the norm continuity of $[0,1]\to K(E):t\mapsto F_t$ in addition, then $F\in K(\mathrm{cyl}\,E)$.
\end{proposition}
\begin{proof}
Clear by definition.
\end{proof}

\section{Freely complementable Hilbert modules}

For a non-separable or non-$\sigma$-unital extension of the technical theorem, we suggest a property of Hilbert modules, designed to mimic projective modules over unital rings.
To avoid the conflict over other notions of projective Hilbert modules such as in \cite{MR4844140}, we use a new terminology, free complementability.
Although the category of Hilbert modules or Hilbert spaces are not suitable for homological algebra due to the lack of abelian category structures, and indeed freely complementable Hilbert modules do not enjoy any categorical projectivity as a lifting property, but the representability as a direct summand of a free module is strong enough to extend several arguments in KK-theory to non-separable C$^*$-algebras.
One might say that the existence of an operator-valued inner product on a Hilbert module already alludes a sort of projectivity in some sense by the stabilization theorem for countably generated Hilbert modules.

Before entering to the definition, we would like to briefly discuss the technical necessity of free complementability.
Strictly speaking, although they play a central role in our non-$\sigma$-unital extension of the technical theorem, one cannot say it is necessary in the construction of analytic Kasparov product, i.e.~the existence of the theory.
However, they are inevitable for nice properties of the theory.
One important result we could obtain by using free complementability is a sufficiently large class of non-separable C$^*$-algebras in which our theory enjoys desirable properties of KK-theory, as stated in the introduction, which includes all $\sigma$-unital C$^*$-algebras and $C_0(X)$ with paracompact locally compact Hausdorff $X$.

\begin{definition}[Freely complementable Hilbert modules]
Let $\widetilde B$ be a unital C$^*$-algebra.
A Hilbert module over $\widetilde B$ is called \emph{free} if it is unitarily equivalent to $H\otimes\widetilde B$ as a Hilbert module over $\widetilde B$ for a Hilbert space $H$, and \emph{freely complementable} if it is unitarily equivalent to a complemented submodule of a free module over $\widetilde B$.

An augmented Hilbert bimodule over a pair $(A,B)$ of C$^*$-algebras is called \emph{freely complementable} if it is freely complementable as a Hilbert module over the unitization $\widetilde B$ of $B$.
\end{definition}

Every countably generated Hilbert module over a C$^*$-algebra $B$ is freely complementable over $\widetilde B$ by the Kasparov stabilization theorem.
In particular, if $B$ is a $\sigma$-unital C$^*$-algebra, then $H\otimes B$ is freely complementable over $\widetilde B$ for any Hilbert space $H$.
Nevertheless, there exists a C$^*$-algebra that is not freely complementable as a Hilbert module over its unitization, which is exactly an example of C$^*$-algebra that is not para-unital.
We will see this later in this section.

If a graded Hilbert module $E$ over a unital C$^*$-algebra $\widetilde B$ is freely complementable, then it is easy to see that we can find a graded complemented embedding $E\subset H\otimes\widetilde B$ of graded Hilbert modules over $\widetilde B$, where $H$ is a graded Hilbert space such that the even and odd parts have the same dimensions.

\subsection{Cropping lemma}

The following lemma is the most important feature of freely complementable Hilbert modules, and this cropping technique will be used repeatedly in the rest of paper.

\begin{lemma}[Cropping lemma]
Let $\widetilde B$ be a unital C$^*$-algebra.
Let $E$ be a freely complementable graded Hilbert module over $\widetilde B$.
For a $\sigma$-unital C$^*$-subalgebra $J_0\subset K(E)$ and a separable C$^*$-subalgebra $A_0\subset B(E)$, there exists a countably generated complemented graded Hilbert submodule $S\subset E$ over $\widetilde B$ such that $J_0\subset K(S)$ and $A_0$ is reduced by $S$.

If a graded complemented embedding $E\subset H\otimes\widetilde B$ is given with the projection $P\in B(H\otimes\widetilde B)$ onto $E$, then $S$ can be taken by the image of $PP_0$, where $P_0\in B(H\otimes\widetilde B)$ is a projection onto $H_0\otimes\widetilde B$ for a graded Hilbert subspace $H_0\subset H$ such that $PP_0=P_0P$.
\end{lemma}
\begin{proof}
Fix a graded complemented embedding $E\subset\ell^2(\mathcal{I},\widetilde B)$ of Hilbert module over $\widetilde B$, where $\mathcal{I}$ is a graded set, and let $P\in B(\ell^2(\mathcal{I},\widetilde B))$ be the projection onto $E$, which is even.
Denote by $\gamma\in B(\ell^2(\mathcal{I}))\subset B(\ell^2(\mathcal{I},\widetilde B))$ the grading operator.
Since the Hilbert module $J_0E$ is countably generated over $\widetilde B$ by the $\sigma$-unitality of $J_0\subset K(E)$, there is a countable subset $\mathcal{I}_0\subset\mathcal{I}$ such that $J_0E\subset\ell^2(\mathcal{I}_0,\widetilde B)$.
Now the separability of $C^*(P,\gamma,A_0)\subset B(E)$ implies that there is an non-decreasing sequence of countable subsets $\mathcal{I}_n\subset\mathcal{I}$ such that $C^*(P,\gamma,A_0)\ell^2(\mathcal{I}_n,\widetilde B)\subset\ell^2(\mathcal{I}_{n+1},\widetilde B)$ for $n\ge0$.
Then, for the countable subset $\mathcal{I}_\infty:=\bigcup_n\mathcal{I}_n\subset\mathcal{I}$, we have $C^*(P,\gamma,A_0)\ell^2(\mathcal{I}_\infty,\widetilde B)\subset\ell^2(\mathcal{I}_\infty,\widetilde B)$, so the orthogonal projection in $P_\infty\in B(\ell^2(\mathcal{I}))\subset B(\ell^2(\mathcal{I},\widetilde B))$ onto $\ell^2(\mathcal{I}_\infty)$ commutes with $P$, $\gamma$, and $A_0$ in $B(\ell^2(\mathcal{I},\widetilde B))$.
Since $P_\infty$ commutes with $\gamma$, the subset $\mathcal{I}_\infty$ is graded and $P_\infty$ is even.
Since $P_\infty$ commutes with $P$, the product $PP_\infty$ is then an even projection commuting with $A_0$, which also satisfies $P_\infty J_0=J_0$ by $J_0E\subset\ell^2(\mathcal{I}_0,\widetilde B)\subset\ell^2(\mathcal{I}_\infty,\widetilde B)$.
Therefore, $S:=P_\infty E=P\ell^2(\mathcal{I}_\infty,\widetilde B)$ provides our desired module.
\end{proof}

Kaplansky observed in \cite{MR100017} that a projective module over a unital ring is isomorphic to a direct sum of countably generated projective modules.
An analogous result can be obtained for freely complementable Hilbert modules from the same idea as the cropping lemma.
Note that we do not use the Zorn lemma directly on the families of orthogonal countably generated Hilbert submodules, because the sum of orthogonal projections in a multiplier algebra need not strictly converge to form an enveloping projection.
In other words, we cannot say the sum of mutually orthogonal countably generated complemented Hilbert submodules is also complemented in general, let $B:=C(\{0\}\cup\{n^{-1}:n\in\mathbb{N}\})$ and $E_n:=\mathbb{C}1_{\{n^{-1}\}}\subset B$ with $n\in\mathbb{N}$ for example.

\begin{corollary}[Kaplansky theorem on freely complementable Hilbert modules]
Over a unital C$^*$-algebra, a (graded) freely complementable Hilbert module is unitarily equivalent to a direct sum of (graded) countably generated Hilbert modules.
\end{corollary}
\begin{proof}
Using the same notations as the proof of the cropping lemma but $A_0=\{0\}$, apply the Zorn lemma on the families of disjoint non-empty countable subsets $\mathcal{I}'$ of $\mathcal{I}$ such that $P\ell^2(\mathcal{I}',\widetilde B)\subset\ell^2(\mathcal{I}',\widetilde B)$, or $C^*(P,\gamma)\ell^2(\mathcal{I}',\widetilde B)\subset\ell^2(\mathcal{I}',\widetilde B)$ for graded version.
\end{proof}

\subsection{Para-unital C\texorpdfstring{$^*$}{*}-algebras}

For useful special cases of freely complementable Hilbert modules, and to find a restriction of our theory on a well-behaved subcategory of C$^*$-algebras, we suggest in this paper a new class of C$^*$-algebras including all $\sigma$-unital algebras, called para-unital C$^*$-algebras, which can be regarded as non-commutative extensions of paracompact locally compact Hausdorff spaces.
Recall that we defined in the introduction a para-unital C$^*$-algebra as a C$^*$-algebra $B$ such that there is a mutually orthogonal family $p_i\in M(B)$ of $\sigma$-unital projections whose sum $\sum_ip_i$ strictly converges to the unit of $M(B)$.

\begin{proposition}
A C$^*$-algebra is para-unital if and only if it is freely complementable as a Hilbert module over its unitization.
\end{proposition}
\begin{proof}
The forward direction is clear because a direct sum of countably generated Hilbert modules is freely complementable.
For the converse, let $A$ be a C$^*$-algebra that is freely complementable over $\widetilde A$.
By the Kaplansky theorem, the previous corollary, we can decompose $A$ into the direct sum $A=\bigoplus_iE_i$ of countably generated Hilbert modules over $A$, and let $P_i\in M(A)$ be the corresponding projections to each factor $E_i$.
Then, for every $a\in A$ the sum $\sum_iP_ia$ converges to $a$ in norm by definition of infinite direct sum of Hilbert modules, so the sum $\sum_iP_i$ converges strictly.
\end{proof}

\begin{proposition}
A commutative C$^*$-algebra is para-unital if and only if its spectrum is paracompact.
\end{proposition}
\begin{proof}
A multiplier projection in $M(C_0(X))\cong C_b(X)$ is exactly a clopen subset of $X$, and a direct sum of commutative C$^*$-algebras is exactly a disjoint union of locally compact Hausdorff spaces.
Then, the equivalence follows by the well-known fact that a locally compact Hausdorff space is paracompact if and only if it admits a partition of $\sigma$-compact clopen subsets.
\end{proof}

\begin{proposition}[Non-para-unitality criterion]\label{Proposition: Non-para-unitality criterion}
If $B$ is a non-$\sigma$-unital C$^*$-algebra such that there is $h\in B$ such that $bh\ne0$ for all non-zero $b\in B$, then $B$ is not para-unital.
\end{proposition}
\begin{proof}
Suppose $B$ is freely complementable over $\widetilde B$ so that there is a coisometry $p\in B(\ell^2(\mathcal{I},\widetilde B),B)$ for some set $\mathcal{I}$.
Let $\varepsilon_i\in\ell^2(\mathcal{I})\subset\ell^2(\mathcal{I},\widetilde B)$ be the standard orthonormal basis indexed by $i\in\mathcal{I}$.
Then, $p\varepsilon_i\ne0$ for uncountably many $i\in\mathcal{I}$ since $B$ is not $\sigma$-unital, while we have a norm convergent sum
\[h=pp^*h=p\sum_i\varepsilon_i\langle\varepsilon_i,p^*h\rangle=\sum_ip\varepsilon_i\langle p\varepsilon_i,h\rangle=\sum_i(p\varepsilon_i)(p\varepsilon_i)^*h.\]
It implies that $(p\varepsilon_i)(p\varepsilon_i)^*h=0$ for all but countably many $i$, and the uncountability of $\mathcal{I}$ finds $i$ leading a contradiction $p\varepsilon_i=0$ by the assumption on $h$.
\end{proof}

\begin{corollary}
The C$^*$-algebra $C_0(\beta\mathbb{N}\setminus\{\omega\})$, where $\omega\in\beta\mathbb{N}\setminus\mathbb{N}$ is a free ultrafilter on $\mathbb{N}$, is not para-unital.
\end{corollary}
\begin{proof}
One may use one of the two previous propositions.
\end{proof}

Recall that for a C$^*$-algebra $A$, the C$^*$-algebra $qA$, called the Cuntz $q$-construction, is defined as the kernel of the natural $*$-homomorphism $A*A\to A$ induced by the double copy of identity on $A$.
We prove that it is unfortunate that $qA$ is always not para-unital whenever $A$ is non-separable.
Since every split exact sequence together with a section is naturally replaced into a $*$-homomorphism from the Cuntz $q$-construction, this fact is a strong evidence there would be a critical obstruction for freely complementable Hilbert modules compatible to split exact theories for non-separable C$^*$-algebras.
To prove this, we use two lemmas.
The first lemma seems a folklore theorem, however we left a proof.
An easy consequence of this lemma that $qA$ is essential in $A*A$ was obtained in a private conversation with Hiroto Nishikawa using the norm continuity of matrix rotation, but is meaningfully generalized in here by a similar argument.
The second lemma is by Hofmann-Neeb \cite{MR1338245}.
Let $\iota_0,\iota_1:A\to A*A$ be the two canonical inclusions, and let $q:A\to qA$ be a $*$-preserving bounded map such that $q(a):=\iota_0(a)-\iota_1(a)$ for $a\in A$.

\begin{lemma}[Non-zero-divisor property]\label{Lemma: Non-zero-divisor property}
For a C$^*$-algebra $A$, if $b\in A\setminus\{0\}$ and $c\in A*A\setminus\{0\}$, then $q(b)c\ne0$.
\end{lemma}
\begin{proof}
Fixing a faithful representation $A*A\subset B(H)$, choose $\xi\in H$ such that $c\xi\ne0$.
Choose any representation $\pi':A\to B(H')$ together with a bounded linear operator $w:H\to H'$ such that
\[(w\iota_1(b)-\pi'(b)w)c\xi\ne0.\]
Such a choice of $\pi'$ and $w$ is always possible, for example by letting $H=H'$, set $\pi'=0$ and $w=1$ if $\iota_1(b)c\xi\ne0$, and set $\pi'$ and $w$ arbitrarily such that $\pi'(b)\ne0$ and $wc\xi\not\in\ker\pi'(b)\subset H$ if $\iota_1(b)c\xi=0$.

Define a path of unitaries $u_t\in B(H\oplus H')$ for $t\in[0,1]$ by
\[z:=\begin{pmatrix}0&-w^*\\w&0\end{pmatrix},\qquad u_t:=e^{tz},\qquad t\in[0,1],\]
and define a path of representations $\pi_t:A*A\to B(H\oplus H')$ by
\[\pi_t(\iota_0(a)):=\iota_0(a)\oplus\pi'(a),\qquad\pi_t(\iota_1(a)):=u_t(\iota_1(a)\oplus\pi'(a))u_t^*,\qquad a\in A,\ t\in[0,1].\]
Then, we can compute
\[\pi_t(q(b))=\begin{pmatrix}q(b)&-t(\iota_1(b)w^*-w^*\pi'(b))\\-t(w\iota_1(b)-\pi'(b)w)&0\end{pmatrix}+O(t^2),\qquad t\to0,\]
so that the $H'$-component of $\pi_t(q(b)c)(\xi,0)$ is non-zero for a sufficiently small $t>0$, hence $q(b)c\ne0$.
\end{proof}

\begin{lemma}[Hofmann-Neeb]
An epic $*$-homomorphism between C$^*$-algebras is surjective.
\end{lemma}
\begin{proof}
Let $\varphi:A\to B$ be an epimorphism in the category $\mathrm{C^*Alg}$.
By taking unitization and quotient by kernel, we may assume it is a unital inclusion.
Let $\pi:B\subset B(H)$ be the universal representation of $B$.
For any unitary $u\in\pi(A)'$, the coincidence $\pi\circ\varphi=(\mathrm{Ad}\,u)\circ\pi\circ\varphi:A\to B(H)$ implies $\pi=(\mathrm{Ad}\,u)\circ\pi:B\to B(H)$ by the assumption that $\varphi$ is epic, which leads $u\in\pi(B)'$.
Then, we get $\pi(A)''=\pi(B)''$ and $A^{**}=B^{**}$, hence it forces $A=B$ by the Hahn-Banach extension.
\end{proof}

\begin{proposition}[Non-para-unital Cuntz construction]\label{Proposition: Non-para-unital Cuntz construction}
If $A$ is a non-separable C$^*$-algebra, then $qA$ is not para-unital.
\end{proposition}
\begin{proof}
If we take a positive non-zero $b\in A$, then $h:=q(b)^*\in qA$ satisfies the condition of the previous proposition (Proposition \ref{Proposition: Non-para-unitality criterion}) by the non-zero-divisor property (Lemma \ref{Lemma: Non-zero-divisor property}) of the Cuntz $q$-construction.
Thus it suffices to prove the non-separability of $A$ implies the non-$\sigma$-unitality $qA$.

Suppose $qA$ is $\sigma$-unital with a strictly positive element $h$, and find a separable C$^*$-subalgebra $A_0\subset A$ such that $h$ belongs to the ideal of $qA$ generated by $q(A_0)=\{q(a):a\in A_0\}$.
No proper ideal contains an approximate unit, so $qA=qAq(A_0)qA$.
If $\varphi_1,\varphi_2:A\to B$ are $*$-homomorphisms such that $\varphi_1|_{A_0}=\varphi_2|_{A_0}$, then the induced $*$-homomorphism $qA\to A*A\to B$ vanishes on $q(A_0)$, and it implies $\varphi_1=\varphi_2$ so that $A_0\subset A$ is an epic inclusion.
Since a epic $*$-homomorphism is surjective, $A=A_0$ is separable.
\end{proof}

\subsection{Subcategories of freely complementable Hilbert bimodules}

Let $\mathrm{Corr}^\mathrm{fc,aug}$ be a class of morphisms in $\mathrm{Corr}^\mathrm{aug}$ consisting of freely complementable augmented Hilbert bimodules, and let $\mathrm{Corr}_\mathrm{para}^\mathrm{fc,ndg}$ be the full subcategory of para-unital C$^*$-algebras in $\mathrm{Corr}^\mathrm{fc,aug}\cap\mathrm{Corr}^\mathrm{ndg}$.

\begin{proposition}
The class $\mathrm{Corr}^\mathrm{fc,aug}$ and $\mathrm{Corr}_\mathrm{para}^\mathrm{fc,ndg}$ define subcategories of $\mathrm{Corr}^\mathrm{aug}$ and $\mathrm{Corr}^\mathrm{ndg}$.
\end{proposition}
\begin{proof}
For $E_1$ and $E_2$ freely complementable augmented Hilbert bimodules over pairs $(A,B)$ and $(B,C)$ of C$^*$-algebras respectively, it suffices to show the internal tensor product $E_1\otimes_BE_2$ freely complementable over $\widetilde C$.
Suppose $E_1\oplus E_1'\cong H_1\otimes\widetilde B$ and $E_2\oplus E_2'\cong H_2\otimes\widetilde C$ as augmented Hilbert bimodules for some $E_1'$ and $E_2'$.
Then,
\[(E_1\otimes_BE_2)\oplus(E_1'\otimes_BE_2)\oplus(E_1\otimes_BE_2')\oplus(E_1'\otimes_BE_2')\cong H_1\otimes H_2\otimes\widetilde C\]
implies that $E_1\otimes_BE_2$ is freely complementable.
\end{proof}

It is well-known to experts that both maximal and minimal external tensor products provide symmetric monoidal structures on the (bi-)category of C$^*$-algebras and non-degenerate Hilbert bimodules, but we leave a sketch of proof.
The following lemma will be used to establish these monoidal structures.
It is briefly discussed in Section 13.5 of \cite{MR1656031} without proof.

\begin{lemma}[Compact operators on external tensor products]\label{Lemma: Compact operators on external tensor products}
Let $B_1$ and $B_2$ be C$^*$-algebras.
Let $E_1$ and $E_2$ be Hilbert modules over $B_1$ and $B_2$ respectively.
Then, $K(E_1)\otimes_*K(E_2)\cong K(E_1\otimes_*E_2)$ for $*\in\{\mathrm{min},\mathrm{max}\}$.
\end{lemma}
\begin{proof}
Let $L_1$ and $L_2$ be C$^*$-algebras with multiplier projections $P_1\in M(L_1)$ and $P_2\in M(L_2)$.
By the preservation of inclusion for minimal tensor product, and by the universal property for maximal tensor product, for each case $*=\mathrm{min}$ and $*=\mathrm{max}$ we can define a $*$-homomorphism
\[P_1L_1P_1\otimes_*P_2L_2P_2\to(P_1\otimes P_2)(L_1\otimes_*L_2)(P_1\otimes P_2),\]
where $P_1\otimes P_2\in M(L_1\otimes_*L_2)$ is well-defined via the canonical $*$-homomorphism $M(L_1)\otimes_*M(L_2)\to M(L_1\otimes_*L_2)$.
It is clearly isometric for minimal case, but also for the maximal case by the existence of a contractive retract $L_1\otimes_\mathrm{max}L_2\to P_1L_1P_1\otimes_\mathrm{max}P_2L_2P_2$.
By the dense range containing the algebraic tensor product, it is an isomorphism.

Now we prove the original statement.
Recall that we can identify an Hilbert module $E$ over a C$^*$-algebra $B$ with $K(B,E)$ by $E\to K(B,E):\xi\mapsto T_\xi$, and the inner product becomes the multiplication $\langle\eta,\xi\rangle=T_\eta^*T_\xi\in B$.
Define the linking algebras $L_1:=K(E_1\oplus B_1)$ and $L_2:=K(E_2\oplus B_2)$.
Let $P_1,Q_1\in M(L_1)$ and $P_2,Q_2\in M(L_2)$ be projections onto $E_1,B_1$ and $E_2,B_2$ respectively.
Let $L:=L_1\otimes_*L_2$, and $P:=P_1\otimes P_2$, $Q:=Q_1\otimes Q_2$ be projections in $M(L)$.
By the argument in the first paragraph, their corners are computed as
\[PLP\cong P_1L_1P_1\otimes_*P_2L_2P_2\cong K(E_1)\otimes_*K(E_2),\qquad QLQ\cong Q_1L_1Q_1\otimes_*Q_2L_2Q_2\cong B_1\otimes_*B_2.\]
Then, since the norm on $E_1\otimes_*E_2$ is given by $\|\xi_1\otimes\xi_2\|_*^2=\|T_{\xi_1}^*T_{\xi_1}\otimes T_{\xi_2}^*T_{\xi_2}\|_{B_1\otimes_*B_2}=\||T_{\xi_1}\otimes T_{\xi_2}|^2\|_{QLQ}$, we have an isometric isomorphism $E_1\otimes_*E_2\cong K(QLQ,PLQ)\cong PLQ$.
It gives rise to a canonical isometry $K(E_1\otimes_*E_2)=K(PLQ)=PLQLP\subset PLP=K(E_1)\otimes_*K(E_2)$.
Since it has dense range, the proof is complete.
\end{proof}

\begin{proposition}
The categories $\mathrm{Corr}^\mathrm{ndg}$ and $\mathrm{Corr}_\mathrm{para}^\mathrm{fc,ndg}$ are given monoidal categories by external tensor products, with respect to any of either the minimal or maximal tensor product.
In particular, the minimal and maximal tensor products of para-unital C$^*$-algebras is para-unital.
\end{proposition}
\begin{proof}
Let $E_1$ and $E_2$ be Hilbert modules over C$^*$-algebras $B_1$ and $B_2$ respectively.
We define the monoidal bifunctor by the external tensor product $\otimes_*$ of non-degenerate Hilbert bimodules, with fixed $*\in\{\mathrm{min},\mathrm{max}\}$.
The left action is defined by a canonical $*$-homomorphism $M(K(E_1))\otimes_*M(K(E_2))\to M(K(E_1)\otimes_*K(E_2))$, which should be constructed separately but not in a difficult way for each case $*=\mathrm{min}$ and $*=\mathrm{max}$.
Since the external tensor product of unitarily equivalent non-degenerate Hilbert bimodules are unitarily equivalent, the monoidal bifunctor is well-defined.
For its functoriality, it is routine to directly and canonically construct a unitary to prove the interchange law.
The monoidal unit is set to be $\mathbb{C}$.
The associator, two unitors, and swapper are same as the category of C$^*$-algebras and $*$-homomorphisms.
Then, the category $\mathrm{Corr}^\mathrm{ndg}$ is a symmetric monoidal category with respect to the minimal or maximal tensor product.

For the subcategory $\mathrm{Corr}_\mathrm{para}^\mathrm{fc,ndg}$, it suffices to show an external tensor product of freely complementable Hilbert modules is freely complementable, but it can be easily proven by the decomposition into countably generated complemented submodules.
\end{proof}

The canonical functors $\mathrm{C^*Alg}\to\mathrm{Corr}^\mathrm{ndg}$ and $\mathrm{C^*Alg}_\mathrm{para}\to\mathrm{Corr}_\mathrm{para}^\mathrm{fc,ndg}$ through the ``non-degenerate reduction'' $E\mapsto AE$ for an augmented Hilbert bimodule $E$ over a pair of C$^*$-algebras $(A,B)$ are monoidal functors, also with respect to either the minimal or maximal tensor product.

Since the identity morphism at an object $A$ in the category $\mathrm{Corr}^\mathrm{ndg}$ is freely complementable if and only if $A$ is para-unital, the object class of $\mathrm{Corr}_\mathrm{para}^\mathrm{fc,ndg}$ contains only para-unital C$^*$-algebras, while the other three bimodule categories have all C$^*$-algebras as objects.
This is why the ``non-degenerate reduction'' functor $\mathrm{Corr}^\mathrm{aug}\to\mathrm{Corr}^\mathrm{ndg}$ does not restrict to the subcategory $\mathrm{Corr}^\mathrm{fc,aug}$.
\[\begin{tikzcd}
\mathrm{C^*Alg}_\mathrm{para}\dar[phantom,sloped,"\subset"]\rar
&\mathrm{Corr}_\mathrm{para}^\mathrm{fc,ndg}\rar[phantom,"\subset"]
&\mathrm{Corr}^\mathrm{ndg}\\
\mathrm{C^*Alg}\rar
&\mathrm{Corr}^\mathrm{fc,aug}\rar[phantom,"\subset"]\uar{(\times)}
&\mathrm{Corr}^\mathrm{aug}\uar
\end{tikzcd}\]
In the above diagram of functors, while the most categorically nice bimodule category for C$^*$-algebras is the semi-additive monoidal category $\mathrm{Corr}^\mathrm{ndg}$, but the most general category to which the Kasparov technical theorem is extended in this paper is $\mathrm{Corr}^\mathrm{fc,aug}$.

\section{Approximate Kasparov bimodules}

Now we define approximate Kasparov bimodules.
After defining equivalence classes between them, we will discuss representatives of approximate Kasparov bimodules we can take in each equivalence class.

\subsection{Kasparov conditions}

\begin{definition}[Support algebras]
Let $A_0$ and $B$ be C$^*$-algebras.
Let $E$ be an augmented Hilbert bimodule over $(A_0,B)$.
We define the \emph{support algebra} of an adjointable operator $F\in B(E)$ under $A_0$ as the C$^*$-subalgebra
\[\Sigma(F,A_0):=C^*([F,a],(F-F^*)a,(F^2-1)a:a\in A_0)\subset B(E).\]
\end{definition}

The above definition is introduced just for convenience in this paper, and it is not a standard terminology.
The notation $\Sigma$ for support algebras will be frequently used to simplify and shorten many proofs in this paper.
By abusing terminologies, we might be able to say that a condition for an odd adjointable operator to be a Kasparov operator is the ``compact support condition''.

\begin{definition}[Kasparov operators]
Let $A_0$ and $B$ be C$^*$-algebras.
Let $E$ be a graded augmented Hilbert bimodule over $(A_0,B)$.
An adjointable operator $F\in B(E)$ is called \emph{Kasparov} under $A_0$ if it is odd with $C^*(F)E\subset EB$ and $\Sigma(F,A_0)\subset K(E)$.
We define a \emph{Kasparov bimodule} over $(A_0,B)$ as a graded freely complementable augmented Hilbert bimodule $E$ over $(A_0,B)$ together with a Kasparov operator $F\in B(E)$.
\end{definition}

We can say then a Kasparov bimodule $(E,F)$ over a pair of C$^*$-algebras $(A_0,B)$ is classical if $E$ is countably generated and it satisfies $E\subset EB$.
Note that if $E\subset EB$, then the condition $C^*(F)E\subset EB$ is automatic.

Let $A$ and $B$ be C$^*$-algebras, and let $\varphi:A\to B$ be a $*$-homomorphism.
We have a canonical Kasparov bimodule associated with $\varphi$ is a Kasparov bimodule $(\widetilde B,0)$ over $(A,B)$, denoted occasionally by the same notation as $\varphi:=(\widetilde B,0)$, where the left action is determined by the $*$-homomorphism $\varphi$.
When $A$ is para-unital, we will also often take a canonical Kasparov bimodule associated with $\varphi$ as $\varphi:=(AB,0)$.

\begin{definition}[Kasparov families]
Let $A_0$ and $B$ be C$^*$-algebras.
Let $E$ be a graded augmented Hilbert bimodule over $(A_0,B)$.
A family $\{F_i\}\subset B(E)$ of adjointable operators is called \emph{Kasparov} under $A_0$ if it is a bounded family of Kasparov operators mutually operator homotopic under $A_0$, which means that every pair of elements in the family is connected by a norm continuous path of Kasparov operators under $A_0$.
\end{definition}

\begin{definition}[Approximate Kasparov bimodules]
Let $A$ and $B$ be C$^*$-algebras.
An \emph{approximate Kasparov bimodule} over $(A,B)$ is a graded freely complementable augmented Hilbert bimodule $E$ over $(A,B)$ together with a bounded net $F_i\in B(E)$ of odd adjointable operators satisfying that for each separable C$^*$-subalgebra $A_0\subset A$ there is $i_0$ such that the family $\{F_i\}_{i\ge i_0}$ is Kasparov under $A_0$.
When the net $F_i$ is constant over indices $i$, we will sometimes say the approximate Kasparov bimodule just as a \emph{Kasparov bimodule}.
\end{definition}

Note that for an approximate Kasparov bimodule $(E,(F_i)_i)$ over a pair of C$^*$-algebras $(A,B)$ if $i,i'\ge i_0$ and $i,i'\ge i_1$, where the indices $i_0$ and $i_1$ are taken depending on two different separable C$^*$-subalgebras $A_0,A_1\subset A$, then the choice of the norm continuous path $F_{ii'}\in C([0,1],B(E))\cap B(\mathrm{cyl}\,E)$ connecting $F_i$ and $F_{i'}$ depends on the choice of the separable algebras $A_0,A_1$.
However, we may always assume $\|F_{ii'}\|\le\max\{1,\|F_i\|,\|F_{i'}\|\}$ by replacing $F_{ii'}$ to $\overline F_{ii'}\in C([0,1],B(E))$ defined such that
\[\overline F_{ii'}(s):=\begin{cases}
(1-5s)F_i+5s(F_i+F_i^*)/2&\text{ if }s\in[0,\frac15],\\
(2-5s)(F_i+F_i^*)/2+(5s-1)g((F_i+F_i^*)/2)&\text{ if }s\in[\frac15,\frac25],\\
g((F_{ii'}+F_{ii'}^*)/2)(5s-2)&\text{ if }s\in[\frac25,\frac35],\\
(4-5s)g((F_{i'}+F_{i'}^*)/2)+(5s-3)(F_{i'}+F_{i'}^*)/2&\text{ if }s\in[\frac35,\frac45],\\
(5-5s)(F_{i'}+F_{i'}^*)/2+(5s-4)F_{i'}&\text{ if }s\in[\frac45,1],
\end{cases}\]
where $g\in C_b(\mathbb{R})$ is an odd function such that $g(s)=s$ if $s\in[0,1]$ and $g(s)=1$ if $s\in[1,\infty)$.

\subsection{Equivalence relations}

Traditionally for C$^*$-algebras $A_0$ and $B$ with $A_0$ separable, the KK-group $KK(A_0,B)$ is defined as the set of classical homotopy classes of classical Kasparov bimodules over $(A_0,B)$.
Here the classical homotopy is nothing but a classical Kasparov bimodule over $(A_0,C([0,1],B))$.
Since we are generally using augmented Hilbert bimodules in the definition of approximate Kasparov bimodules, we can not define a homotopy relation in the same manner to the classical KK-theory, as an approximate Kasparov bimodule over $(A,C([0,1],B))$.
For example, even if $\varphi:A\to B$ is a zero $*$-homomorphism between general C$^*$-algebras, its canonical Kasparov bimodule is defined by $(\widetilde B,0)$, and it has a non-zero fiber at infinity $\widetilde B\otimes\mathbb{C}\cong\mathbb{C}$.
When we restrict the case of separable $A$ and $B$, it does not agree with the classical picture that has zero fiber at infinity because an augmented Hilbert bimodule $E$ over $(A,C([0,1],B))$ for a homotopy has a constant fiber at infinity in the sense that $E\otimes_{C([0,1],B)}\mathbb{C}=E/EC([0,1],B)$ is a single Hilbert space without dependence on $[0,1]$.

In this situation, recall that the homotopy relation of classical Kasparov bimodules is essentially the same as the equivalence relation generated by operator homotopy and addition of degenerate cycles, as discussed in Section 18.5 of \cite{MR1656031} and Theorem 2.2.17 in \cite{MR1124848}.
Following this principle, in order to cancel out the non-trivial fibers at infinity of trivial approximate Kasparov bimodules by an equivalence relation, we define a homotopy relation as an equivalence relation generated by two moves, so called the homotopy moves and the degeneracy moves, where the latter means the addition of degenerate approximate Kasparov bimodules.
They are defined precisely as follows.
First, we define degenerate approximate Kasparov bimodules, and prove a lemma that will be used later.

\begin{definition}[Degenerate approximate Kasparov bimodules]
Let $A_0$, $A$, and $B$ be C$^*$-algebras.
Let $(E,F)$ be a Kasparov bimodule over $(A_0,B)$.
We say $(E,F)$ is \emph{degenerate} or $F$ is \emph{degenerate} under $A_0$ if $\Sigma(F,A_0)=\{0\}$.
An approximate Kasparov bimodule $(E,(F_i)_i)$ over $(A,B)$ is called \emph{degenerate} if for each separable C$^*$-subalgebra $A_0\subset A$ there is $i_0$ such that the family $\{F_i\}_{i\ge i_0}$ is \emph{degenerate} under $A_0$, meaning that every pair of elements in the family is connected by a norm continuous path of degenerate Kasparov operators under $A_0$.
We say an approximate Kasparov bimodule is \emph{trivial} if it has zero left action.
\end{definition}

\begin{lemma}[Escaping to zero left action]\label{Lemma: Escaping to zero left action}
Let $A$ and $B$ be C$^*$-algebras.
Let $(H\otimes\widetilde B,(F_{0i})_i)$ be a degenerate approximate Kasparov bimodule over $(A,B)$, where $H$ is a graded Hilbert space such that even and odd parts have the same infinite dimensions.
Then, there is a degenerate approximate Kasparov bimodule $(H\otimes C([0,1],B)^\sim,(F_i)_i)$ over $(A,C([0,1],B))$ connecting $(H\otimes\widetilde B,(F_{0i})_i)$ and a trivial approximate Kasparov bimodule.
\end{lemma}
\begin{proof}
Denote by $\varphi_0:A\to B(H\otimes\widetilde B)$ the left action of the degenerate approximate Kasparov bimodule $(H\otimes\widetilde B,(F_{0i})_i)$.
Choose a ($\sigma$-)strongly continuous path $s:[0,1)\to B(H):t\mapsto s_t$ of even isometries with $s_0=1$ and $s_ts_t^*\to0$ as $t\to1$.
For example, one can construct the path as a continuous shift escaping to the infinity but acting as an identity on $H'$ in a factorization $H\cong L^2([0,\infty))\otimes H'$, where $H'$ is a graded Hilbert space.
Using the same notation $s_t\in B(H)\subset B(H\otimes\widetilde B)$ for $t\in[0,1]$, define $\varphi:A\to C([0,1],B(H\otimes\widetilde B))$ and $F_i\in C([0,1],B(H\otimes\widetilde B))$ by
\[\varphi(a)(t):=s_t\varphi_0(a)s_t^*,\qquad F_i(t):=s_tF_{0i}s_t^*,\qquad t\in[0,1),\]
where the strict topology is endowed on $B(H\otimes\widetilde B)$.
Since they vanish at infinity fiber, they give a left action and an odd adjointable operator on $H\otimes C([0,1],B)^\sim$.
We can easily check they also define a degenerate approximate Kasparov bimodule by computing the support algebra point-wise along $[0,1]$ for degenerate operator homotopies connecting members in degenerate Kasparov families, satisfying the condition of the lemma.
\end{proof}

If a path $F_{0ii'}$ connecting $F_{0i}$ and $F_{0i'}$ in a same Kasparov family is not degenerate, then the above construction $t\mapsto s_tF_{0ii'}s_t^*$ does not give rise to a desired path of paths since the norm continuity fails at $t=1$ as well as the compactness for Kasparov conditions.
This is the main reason why we assume the existence of degenerate paths connecting members of a Kasparov family for a degenerate approximate Kasparov bimodule.

\begin{definition}[Homotopy approximate Kasparov bimodules]
Let $A$ and $B$ be C$^*$-algebras.
We define a \emph{homotopy} between two approximate Kasparov bimodules $(E_0,(F_{0i})_i)$ and $(E_1,(F_{1j})_j)$ over $(A,B)$ to be an approximate Kasparov bimodule $(E,(F_l)_l)$ over $(A,C([0,1],B))$ together with monotone cofinal functions $l\mapsto i(l)$ and $l\mapsto j(l)$ such that $(E(0),F_l(0))$ and $(E(1),F_l(1))$ are unitarily equivalent to $(E_0,(F_{0i(l)})_l)$ and $(E_1,(F_{1j(l)})_l)$.
\end{definition}

\begin{definition}[Homotopy relation]
Let $A$ and $B$ be C$^*$-algebras.
On the class of approximate Kasparov bimodules over $(A,B)$, we define a \emph{homotopy relation} as an equivalence relation generated by the following two moves: for two approximate Kasparov bimodules over $(A,B)$ they are homotopic if
\begin{enumerate}
\item there is a homotopy connecting them, or
\item one is unitarily equivalent to a direct sum of another with a degenerate approximate Kasparov bimodule.
\end{enumerate}
\end{definition}

We also define approximate version of operator homotopies, which are more convenient to check than general homotopies.
Note that in the following definition an operator homotopy is a net of paths of operators which are not just strictly continuous but norm continuous along $[0,1]$.
An operator homotopy also vanishes at infinity, so we will usually regard it as a net of operators on the cylinder Hilbert module.
An important point for the definition of an operator homotopy between approximate Kasparov bimodules is that it automatically satisfies the Kasparov family condition and defines an approximate Kasparov bimodule so that operator homotopic approximate Kasparov bimodules are homotopic.
We also give a slightly modified version of Connes-Skandalis operator homotopy for augmented cycles.
Recall that we defined a Kasparov family such that tails have mutually operator homotopic Kasparov operators under suitably chosen separable C$^*$-subalgebra.

\begin{definition}[Operator homotopies]
Let $A$ and $B$ be C$^*$-algebras.
Let $(E_0,(F_{0i})_i)$ and $(E_1,(F_{1j})_j)$ be approximate Kasparov bimodules over $(A,B)$.
They are called \emph{operator homotopic} if $E:=E_0=E_1$ and there is a bounded net $F_l\in C([0,1],B(E))$ of norm continuous paths satisfying $C^*(F_l(t))E\subset EB$ for all $t\in[0,1]$, such that for every separable C$^*$-subalgebra $A_0\subset A$ there is $l_0$ satisfying that for each $l\ge l_0$ the operator $F_l$ is Kasparov under $A_0$, together with monotone cofinal functions $l\mapsto i(l)$ and $l\mapsto j(l)$ such that $F_l(0)=F_{0i(l)}$ and $F_l(1)=F_{1j(l)}$.
The net $F_l$ is called an \emph{operator homotopy}, and called \emph{normalized} if $\|F_l\|\le\max\{1,\|F_l(0)\|,\|F_l(1)\|\}$.
\end{definition}

\begin{proposition}\label{Proposition: Operator homotopic homotopic}
Operator homotopic approximate Kasparov bimodules are homotopic.
\end{proposition}
\begin{proof}
Let $(E,(F_{0i})_i)$ and $(E,(F_{1j})_j)$ be approximate Kasparov bimodules over an ordered pair of C$^*$-algebras $(A,B)$, and let $F_l\in C([0,1],B(E))$ be an operator homotopy between them.
It suffices to show that $(\mathrm{cyl}\,E,(F_l)_l)$ is an approximate Kasparov bimodule over $(A,C([0,1],B))$.
Let $A_0\subset A$ be a separable C$^*$-subalgebra, take $l_0$ such that $F_l$ is Kasparov under $A_0$ for each $l\ge l_0$.
We claim that the family $\{F_l\}_{l\ge l_0}$ is Kasparov under $A_0$.
Fix $l\ge l_0$.
There are operator homotopies $F_{0i(l_0)i(l)},F_{1j(l_0)j(l)}\in C([0,1],B(E))$ under $A_0$ connecting $F_{0i(l_0)}$ and $F_{0i(l)}$, $F_{1j(l_0)}$ and $F_{1j(l)}$ respectively.
If we define $F_{l_0l}\in C([0,1],C([0,1],B(E)))$ as a square with four vertices $(F_{l_0}(0),F_{l_0}(1),F_l(0),F_l(1))$ degenerate to the three paths $F_{l_0}$, $F_l$, and $F_{0i(l_0)i(l)}$ or $F_{1j(l_0)j(l)}$, precisely such that for $(s,t)\in[0,1]^2$ we can write
\[(F_{l_0l}(s))(t):=\begin{cases}
F_{l_0}((1-3s)t)&\text{ if }s\in[0,\frac13],\\
F_{0i(l_0)i(l)}(3s-1)&\text{ if }s\in[\frac13,\frac23],\\
F_l((3s-2)t)&\text{ if }s\in[\frac23,1],
\end{cases}
\qquad\text{ or }\qquad
\begin{cases}
F_{l_0}((1-3s)t+3s)&\text{ if }s\in[0,\frac13],\\
F_{1j(l_0)j(l)}(3s-1)&\text{ if }s\in[\frac13,\frac23],\\
F_l((3-3s)+(3s-2)t)&\text{ if }s\in[\frac23,1],
\end{cases}\]
then it defines an operator homotopy such that $F_{l_0l}(0)=F_{l_0}$ and $F_{l_0l}(1)=F_l$ under $A_0$.
\end{proof}

\begin{theorem}[Connes-Skandalis operator homotopy]\label{Theorem: Connes-Skandalis operator homotopy}
Let $A_0$ and $B$ be C$^*$-algebras.
Let $E$ be a graded augmented Hilbert bimodule over $(A_0,B)$.
For Kasparov operators $F_0,F_1\in B(E)$ under $A_0$, if $a^*[F_0,F_1]a\ge0$ in $Q(E)$ for all $a\in A_0$, then $F_0,F_1$ are operator homotopic under $A_0$.
\end{theorem}
\begin{proof}
First, we may assume $F_0$ and $F_1$ are self-adjoint.
Consider C$^*$-subalgebras
\[D(E):=\{x\in B(E):[x,a]\in K(E)\text{ for }a\in A_0\},\qquad C(E):=\{x\in B(E):xa,ax\in K(E)\text{ for }a\in A_0\}.\]
Then, $C(E)$ is a closed ideal of $D(E)$, and $F_0,F_1$ are self-adjoint unitaries in $D(E)/C(E)$.

Fix $a\in A_0$.
Then, the assumption implies $(-a^*[F_0,F_1]a)^+=0$ in $Q(E)$.
Since $[F_0,F_1]$ is self-adjoint in $B(E)$, by pure state evaluations in the commutative C$^*$-subalgebra $C^*([F_0,F_1],a^*a)\in Q(E)$, we can write $(-[F_0,F_1])^+a^*a=0$ in $Q(E)$.
This proves $(-[F_0,F_1])^+\in C(E)$ and $[F_0,F_1]\ge0$ in $D(E)/C(E)$.

Choose a positive $P\in D(E)$ such that $P=[F_0,F_1]$ in $D(E)/C(E)$.
Define $F\in C([0,1],B(E))$ using $P\ge0$ by
\[F(t):=(1+c_ts_tP)^{-\frac12}(c_tF_0+s_tF_1),\qquad t\in[0,1],\]
where $c_t:=\cos\frac{\pi t}2$ and $s_t:=\sin\frac{\pi t}2$, and it is an operator homotopy connecting $F_0$ and $F_1$ under $A_0$.
\end{proof}

Now we define freely complementable KK-groups.
We will see that a freely complementable KK-group is indeed an abelian group.
See Section 5.5 and 7.1.

\begin{theorem}[Set theoretic well-definedness]
Let $A$ and $B$ be C$^*$-algebras.
Every approximate Kasparov bimodule over $(A,B)$ is homotopic to an approximate Kasparov bimodule of the form $(\ell^2(\kappa^{\aleph_0},\widetilde B),(F_{1j})_j)$, where $\kappa$ is the density character of $A$, and the net $F_{1j}$ is indexed by the directed set of all separable C$^*$-subalgebras of $A$.
\end{theorem}
\begin{proof}
Let $(E,(F_{0i})_i)$ be an approximate Kasparov bimodule over $(A,B)$.
We first show that we may assume the indexing directed set of the net of Kasparov operators is the directed set of all separable C$^*$-subalgebras of $A$.
Let $j$ be the index parameter, and $A_j\subset A$ be the corresponding the separable C$^*$-subalgebra.
By definition of approximate Kasparov bimodules, for each $j$, we can choose $i_j$ such that $\{F_{0i}\}_{i\ge i_j}$ is a Kasparov family under $A_j$.
If we let $F_{1j}:=F_{0i_j}$, then $(E,(F_{0i})_i)$ and $(E,(F_{1j})_j)$ are easily shown to be operator homotopic.

Now we prove that we can generally set the underlying Hilbert module to be $\ell^2(\kappa^{\aleph_0},\widetilde B)$.
Suppose we have an identification $E=P\ell^2(\mathcal{I},\widetilde B)$ for a projection $P\in B(\ell^2(\mathcal{I},\widetilde B))$ as Hilbert modules over $\widetilde B$, for a graded set $\mathcal{I}$.
We may assume $\aleph_0\le\kappa$ and $\kappa^{\aleph_0}<|\mathcal{I}|$.
Let $\gamma$ be the grading operator.
If we choose suitable operator homotopies $F_{1j_0j}$ connecting $F_{1j_0}$ and $F_{1j}$ under $A_{j_0}$ for each pair $j_0\le j$, since
\[\bigcup_{j_0,j:j_0\le j}\bigcup_{s\in[0,1]\cap\mathbb{Q}}(\Sigma(F_{1j_0j}(s),A_{j_0})\cup C^*(P,\gamma,F_{1j_0j}(s),A))\]
has density characters less than or equal to $\kappa^{\aleph_0}$, as we did in the proof of cropping lemma, then we can find a subset $\mathcal{I}_\infty\subset\mathcal{I}$ with $|\mathcal{I}_\infty|\le\kappa^{\aleph_0}$ such that the projection $P_\infty\in B(\ell^2(\mathcal{I}))$ onto $\ell^2(\mathcal{I}_\infty)$ has $((1-P_\infty)E,((1-P_\infty)F_{1j}(1-P_\infty))_j)$ as a degenerate approximate Kasparov bimodule so that $(E,(F_{1j})_j)$ is homotopic to $(P_\infty E,(P_\infty F_{1j}P_\infty)_j)$.
Then, by adding $(1-P)P_\infty\ell^2(\mathcal{I},\widetilde B)$ with zero left action, we are done.
\end{proof}

\begin{definition}[Freely complementable KK-groups]
Let $A$ and $B$ be C$^*$-algebras.
We define the \emph{freely complementable KK-group} of $(A,B)$ as the set $KK^\mathrm{fc}(A,B)$ of all homotopy classes of approximate Kasparov bimodules over $(A,B)$.
\end{definition}

\subsection{Gluing problems and homotopy reduction}

We defined the homotopy relation for approximate Kasparov bimodules as an equivalence relation generated by two moves.
We can ask if one can reduces the finite steps of zig-zags connecting a pair of homotopic approximate Kasparov bimodules.
For this question, we need first to check whether the gluing approximate Kasparov bimodules is safe to concatenate homotopies.
Fortunately, the resulting underlying augmented Hilbert bimodule is freely complementable.
When concatenating homotopies, since we use the end-point evaluation of a continuous field of C$^*$-algebras, the surjectivity condition in the following lemma is satisfied in this case.

\begin{lemma}[Free complementability of gluing]\label{Lemma: Free complementability of gluing}
Let $A_1$, $A_2$, and $B$ be C$^*$-algebras.
Let $\varphi_1:A_1\to B$ and $\varphi_2:A_2\to B$ be $*$-homomorphisms with $\varphi_1$ surjective.
Let $E_1$ and $E_2$ be freely complementable Hilbert modules over $\widetilde A_1$ and $\widetilde A_2$ respectively together with an identification $E_1\otimes_{A_1}\widetilde B\cong E_2\otimes_{A_2}\widetilde B$.
Then,
\[E:=\{(\xi_1,\xi_2)\in E_1\oplus E_2:\xi_1\otimes1=\xi_2\otimes1\}\]
is a freely complementable Hilbert module over $(A_1\oplus_BA_2)^\sim$.
\end{lemma}
\begin{proof}
It is a Hilbert module over $(A_1\oplus_BA_2)^\sim$ since $\varphi_1(\langle\eta_1,\xi_1\rangle)=\varphi_2(\langle\eta_2,\xi_2\rangle)$ for $(\xi_1,\xi_2),(\eta_1,\eta_2)\in E$.
We prove it is freely complementable over $(A_1\oplus_BA_2)^\sim$.
With a sufficiently large graded Hilbert space $H$, consider even unitaries $u_1:E_1\oplus(H\otimes\widetilde A_1)\to H\otimes\widetilde A_1$ and $u_2:E_2\oplus(H\otimes\widetilde A_2)\to H\otimes\widetilde A_2$ using the decomposition into countably generated Hilbert submodules and the Kasparov stabilization.
Canonically gluing the domains of $u_1$ and $u_2$ we get $E\oplus(H\otimes(A_1\oplus_BA_2)^\sim)$, but the targets are glued to give a Hilbert module
\[P_u:=\{(\zeta_1,\zeta_2)\in(H\otimes\widetilde A_1)\oplus(H\otimes\widetilde A_2):\zeta_1\otimes1=u(\zeta_2\otimes1)\}\]
over $(A_1\oplus_BA_2)^\sim$, where the unitary twist is given by $u:=(u_1\otimes1)(u_2\otimes1)^*\in B(H\otimes\widetilde B)$.
It suffices to show $P_u$ is freely complementable over $(A_1\oplus_BA_2)^\sim$.
With $u\oplus u^*\in B(H^{\oplus2}\otimes\widetilde B)$, consider the direct sum
\[P_u\oplus P_{u^*}=\{(\zeta_1,\zeta_2)\in(H^{\oplus2}\otimes\widetilde A_1)\oplus(H^{\oplus2}\otimes\widetilde A_2):\zeta_1\otimes1=(u\oplus u^*)(\zeta_2\otimes1)\}.\]
Since $\widetilde B$ is unital, we have a decomposition $H=\bigoplus_iH_i$ into separable Hilbert spaces such that $u$ is reduced by $H_i\otimes\widetilde B$ for each $i$ with the same argument in the cropping lemma.
Observing $M(K(H_i)\otimes\widetilde A_1)\to M(K(H_i)\otimes\widetilde B)$ is surjective by the $\sigma$-unitality of $K(H_i)\otimes\widetilde A_1$ and the non-commutative Tietze extension theorem, we can choose a contractive lift $v\in B(H\otimes\widetilde A_1)$ of $u$ by the sum of lifts taken, for each $i$.
Then,
\[w:=\begin{pmatrix}v&-(1-vv^*)^{1/2}\\(1-v^*v)^{1/2}&v^*\end{pmatrix}\in B(H^{\oplus 2}\otimes\widetilde A_1)\]
is a unitary that lifts $u\oplus u^*$.
Then, $P_u\oplus P_{u^*}\to H^{\oplus2}\otimes(A_1\oplus_BA_2)^\sim:(\zeta_1,\zeta_2)\mapsto(w^*\zeta_1,\zeta_2)$ is a well-defined unitary since $(w^*\zeta_1)\otimes1=(u\oplus u^*)^*(\zeta_1\otimes1)=\zeta_2\otimes1$.
Therefore, $P_u$ is freely complementable over $(A_1\oplus_BA_2)^\sim$, so is $E$.
\end{proof}

However, the nets of Kasparov operators cannot be glued in general because we have not included the data of connecting paths among Kasparov operators in the cycles.
For the constant net cases, we can simplify the homotopy relation as follows.
Observe that except the gluing homotopies in the beginning of the proof, all the other arguments are available also for general approximate Kasparov bimodules.

\begin{proposition}[Homotopy reduction for separables]\label{Proposition: Homotopy reduction for separables}
Let $A_0$ and $B$ be C$^*$-algebras with $A_0$ separable.
Let $\varphi_0$ and $\varphi_1$ be homotopic Kasparov bimodules over $(A_0,B)$.
Then, there are degenerate Kasparov bimodules $\varphi_0'$ and $\varphi_1'$ over $(A_0,B)$ such that $\varphi_0\oplus\varphi_0'$ and $\varphi_1\oplus\varphi_1'$ are connected by a homotopy.
\end{proposition}
\begin{proof}
Note that two Kasparov bimodules over $(A_0,B)$ connected by a homotopy are connected by a homotopy that is a Kasparov bimodule, so by a homotopy we may assume to mean a Kasparov bimodule over $(A_0,C([0,1],B))$.
By the gluing homotopies (Lemma \ref{Lemma: Free complementability of gluing}), it suffices to prove the following two claims.
\begin{enumerate}
\item
Let $\varphi_0,\varphi_0',\varphi_1,\varphi_1'$ be (approximate) Kasparov bimodules over $(A,B)$.
Suppose $\varphi_0$ is connected to $\varphi_1$ by a homotopy $\varphi$, and $\varphi_0'$ and $\varphi_1'$ are degenerate.
Then, there exists a degenerate Kasparov bimodule $\varphi_0''$ and $\varphi_1''$ over $(A,B)$ such that $\varphi_0\oplus\varphi_0'\oplus\varphi_0''$ is connected to $\varphi_1\oplus\varphi_1'\oplus\varphi_1''$ by a homotopy $\varphi''$.
\item
Let $\varphi_0$, $\varphi_1$, $\varphi_1'$ be (approximate) Kasparov bimodules over $(A,B)$.
Suppose $\varphi_0$ is connected to $\varphi_1$ by a homotopy $\varphi$, and $\varphi_1'$ is degenerate.
Then, there exists a degenerate Kasparov bimodule $\varphi_0'$ over $(A,B)$ such that $\varphi_0\oplus\varphi_0'$ is connected to $\varphi_1\oplus\varphi_1'$ by a homotopy $\varphi'$.
\end{enumerate}

(1)
Since the underlying augmented Hilbert modules of $\varphi_0'$ and $\varphi_1'$ are freely complemented, we can find degenerate Kasparov bimodules $\varphi_0''$ and trivial $\varphi_1''$ with zero Kasparov operator such that the underlying augmented Hilbert modules of $\varphi_0'\oplus\varphi_0''$ and $\varphi_1'\oplus\varphi_1''$ are unitarily equivalent to $H\otimes\widetilde B$ as graded Hilbert bimodules over $\widetilde B$, where $H$ is a graded Hilbert space such that the even and odd parts have the same infinite-dimensions.
By the previous escaping lemma (Lemma \ref{Lemma: Escaping to zero left action}) the Kasparov bimodules $\varphi_0'\oplus\varphi_0''$ and $\varphi_1'\oplus\varphi_1''$ are connected by a homotopy.
Since $\varphi_0$ and $\varphi_1$ are connected by a homotopy, the claim follows by adding homotopies to get $\varphi''$.

(2)
This is clear by letting $\varphi_0':=\varphi_1'$.
\end{proof}

\subsection{Non-degenerate reduction}

In the rest of this section, we investigate some reduction results on approximate Kasparov bimodules, which are about taking nice representatives.
Among these, we first prove results on para-unital C$^*$-algebras.
What we mean by a non-degenerate reduction is that it is allowed to replace the underlying augmented Hilbert bimodule $E$ of an approximate Kasparov bimodule over a pair $(A,B)$ of C$^*$-algebras to $AE$.

\begin{definition}[Non-degenerate Kasparov bimodules]
Let $A$ and $B$ be C$^*$-algebras.
An approximate Kasparov bimodule $(E,(F_i)_i)$ over $(A,B)$ is called \emph{non-degenerate} if the left action is non-degenerate, i.e.~$AE=E$.
\end{definition}

We use a matrix rotation trick, and its well-definedness essentially follows from the property of Grassmann connections, which will be discussed in the next section.
Moreover, using the non-degerate reduction, at least when the first argument is para-unital, we can show the freely complementable KK-group is a group.
It can be shown that it is always a group even for non-para-unital C$^*$-algebras, but it will be proved in Theorem \ref{Theorem: General additive inverses}.

\begin{theorem}[Non-degenerate reduction]\label{Theorem: Non-degenerate reduction}
Let $A$ and $B$ be C$^*$-algebras with $A$ para-unital.
Let $(E_0,(F_{0i})_i)$ be an approximate Kasparov bimodule over $(A,B)$.
Then, there is a homotopic approximate Kasparov bimodule $(AE_0,(F_{1i})_i)$ over $(A,B)$, which is in particular non-degenerate.
\end{theorem}
\begin{proof}
Since $A$ is para-unital and $E_0$ is freely complementable, the internal tensor product $AE_0=A\otimes_AE_0$ is also freely complementable, over $\widetilde B$.
Let $p\in B(H\otimes\widetilde A,A)$ be a coisometry, where $H$ is a graded Hilbert space, and consider the induced coisometry $p\in B(H\otimes E_0,AE_0)$ with the same notation.
With $F_{1i}:=p(1\otimes F_{0i})p^*\in B(AE_0)$, we prove that there is an approximate Kasparov bimodule $(E,(F_i)_i)$ over $(A,C([0,1],B))$ defined by rotating only the left action, precisely as
\[E:=\mathrm{cyl}\,(E_0\oplus AE_0),\qquad\varphi(a)(t):=\begin{pmatrix}c_t^2a&c_ts_ta\\s_tc_ta&s_t^2a\end{pmatrix},\qquad F_i:=\begin{pmatrix}F_{0i}&0\\0&F_{1i}\end{pmatrix},\]
where $c_t:=\cos\frac{\pi t}2$ and $s_t:=\sin\frac{\pi t}2$, for $t\in[0,1]$, and $\varphi:A\to B(E)$ denotes the left action.

For a separable C$^*$-subalgebra $A_0\subset A$, define a separable C$^*$-subalgebra $A_1\subset A$ such that $A_0\subset A_1$ and $p^*A_0\subset H\otimes\widetilde A_1$, and take an index $i_0$ satisfying that $\{F_i\}_{i\ge i_0}$ is Kasparov over $(A_1,B)$.
For any $a\in A_0$ and $i\ge i_0$, expanding as $p^*a=\sum_{n=1}^\infty\varepsilon_n\otimes a_n\in H\otimes\widetilde A$ for $a_n\in\widetilde A_1$, or abusing notation to write the left action as $a\in B(E,AE)$ so that $p^*a=\sum_{n=1}^\infty T_{\varepsilon_n}a_n\in B(E,H\otimes E)$, we can check by the compactness of $[F_{0i},a_n]$ that
\[F_{1i}a-aF_{0i}=p(1\otimes F_{0i})p^*a-pp^*aF_{0i}=p\sum_{n=1}^\infty T_{\varepsilon_n}[F_{0i},a_n]\in K(E,AE),\]
thus we can compute the Kasparov condition
\[[F_i,a](t)=\begin{pmatrix}{[F_{0i},a]}c_t^2&(F_{0i}a-aF_{1i})s_tc_t\\(F_{1i}a-aF_{0i})s_tc_t&[F_{1i},a]s_t^2\end{pmatrix}\in K(E),\qquad t\in[0,1],\]
as well as $(F_i-F_i^*)a,(F_i^2-1)a\in K(E)$.
This completes the proof that $(E,(F_i)_i)$ is an approximate Kasparov bimodule.
The homotopy condition is clear by evaluation at $t=0$ and $t=1$.
\end{proof}

\begin{corollary}
Let $A$ and $B$ be C$^*$-algebras with $A$ para-unital.
For a $*$-homomorphism $\varphi:A\to B$, the canonical Kasparov bimodules $(\widetilde B,0)$ and $(AB,0)$ over $(A,B)$ such that the left actions are given by $\varphi$ are homotopic.
\end{corollary}

\begin{corollary}[Additive inverses for para-unitals]\label{Corollary: Additive inverses for para-unitals}
Let $A$ and $B$ be C$^*$-algebras with $A$ para-unital.
Then, $KK^\mathrm{fc}(A,B)$ admits a canonical abelian group structure by direct sum.
\end{corollary}
\begin{proof}
The only non-trivial part is the existence of inverses.
Since $A$ is para-unital, for an approximate Kasparov bimodule $(E,(F_{0i})_i)$ over $(A,B)$ that can be always assumed to satisfy $E\subset EB$ by the non-degenerate reduction (Theorem \ref{Theorem: Non-degenerate reduction}), its inverse can be defined as $(E_-\boxplus E_+,(-F_{0i})_i)$ where $E=E_+\boxplus E_-$, with an operator homotopy $F_i\in C([0,1],(E_+\boxplus E_-)\oplus(E_-\boxplus E_+))$ to a degenerate Kasparov bimodule on the direct sum as
\[F_i(t):=\begin{pmatrix}F_{0i}c_t&s_t\\s_t&-F_{0i}c_t\end{pmatrix},\qquad t\in[0,1],\]
where $c_t:=\cos\frac{\pi t}2$ and $s_t:=\sin\frac{\pi t}2$ for $t\in[0,1]$.
\end{proof}

As a remark, we leave a comment that the argument of non-degenerate reduction is rather necessary in the classical KK-theory than the freely complementable KK-theory.
It is used to prove the identity axiom for the Kasparov product to present a category, which is far from trivial in some sense.
This is because for a classical Kasparov bimodule $(E,F)$ over a pair of separable C$^*$-algebras $(A,B)$ the internal tensor product $A\otimes_AE=AE$ is not complemented in $E$ if the left action is not strictly continuous.
On the other hand, tensoring with the canonical augmented Hilbert bimodule $\widetilde A$ over $A$ does not change the original Hilbert module as $\widetilde A\otimes_AE=E$, so the identity axiom is clear without using non-degenerate reduction in freely complementable theory.

\subsection{Separable reduction}

We next prove that if the first variable of the freely complementable KK-group is separable, then we may assume the cycles are classical Kasparov bimodules.
This implies that our freely complementable KK-theory by approximate Kasparov bimodules coincides with the classical KK-theory by classical Kasparov bimodules when the first variable is separable.
It also shows that the usual hypothesis of the $\sigma$-unitality of second variable is in fact not necessary even in the classical picture of Kasparov, if we admit a restriction that standard Hilbert modules cannot be used.
For these results, see also \cite{MR775997} and compare with Appendix B of \cite{carrion2026classifyinghomomorphismsiunital}.

\begin{theorem}[Separable reduction I]\label{Theorem: Separable reduction I}
Let $A_0$ and $B$ be C$^*$-algebras with $A_0$ separable.
Then, $KK^\mathrm{fc}(A_0,B)$ is realized as the set of classical homotopy classes of classical Kasparov bimodules over $(A_0,B)$.
\end{theorem}
\begin{proof}
Denote by $KK(A_0,B)$ the set of homotopy classes of classical Kasparov bimodules over $(A_0,B)$.
Then, there is a natural group homomorphism $KK(A_0,B)\to KK^\mathrm{fc}(A_0,B)$.
We claim that it is bijective.

(Surjectivity)
Since $A_0$ is separable, it is clear that every approximate Kasparov bimodule $(E,(F_i)_i)$ is operator homotopic to a Kasparov bimodule $(E,F)$, over $(A_0,B)$.
Since $E$ is freely complementable and $A_0$ is separable, we can crop a countably generated complemented graded Hilbert submodule $S\subset E$ over $\widetilde B$ such that $\Sigma(F,A_0)\subset K(S)$ and $C^*(F,A_0)$ is reduced by $S$.
Let $P\in B(E)$ be the projection onto $S$, then the corner of $(E,F)$ at $1-P$ generates a degenerate Kasparov bimodule, so $(E,F)$ is homotopic to $(S,PFP)$ over $(A_0,B)$.
It implies that we may assume $(E,F)$ is countably generated.
To get $E\subset EB$, apply the non-degenerate reduction (Theorem \ref{Theorem: Non-degenerate reduction}).

(Injectivity)
By the homotopy reduction (Proposition \ref{Proposition: Homotopy reduction for separables}), and since $A_0$ is separable, it suffices to show that if $(E_0\oplus E_0',F_0\oplus F_0')$ and $(E_1\oplus E_1',F_1\oplus F_1')$ are Kasparov bimodules over $(A_0,B)$ such that they are connected by a homotopy $(E',F')$, $(E_0,F_0)$ and $(E_1,F_1)$ are classical, and $(E_0',F_0')$ and $(E_1',F_1')$ are degenerate, then $(E_0,F_0)$ and $(E_1,F_1)$ are connected by a classical homotopy $(E,F)$.

For a complemented embedding $E'\subset\ell^2(\mathcal{I},C([0,1],B)^\sim)$, by choosing suitable countable subsets of $\mathcal{I}$ after evaluating at zero and one of $[0,1]$, we can crop out a countably generated complemented Hilbert submodule $S'\subset E'$ such that $E_0\subset S'(0)$ and $E_1\subset S'(1)$, which are complemented, and $C^*(F',A_0)$ is reduced by $S'$ so that $C^*(F'_0,A_0)$ and $C^*(F_1',A_0)$ are reduced by $S'(0)$ and $S'(1)$ respectively.
If we let
\[S_c':=C^*(1,\gamma)C^*(F',A_0)S'\subset S',\qquad E_{0c}:=C^*(1,\gamma)C^*(F_0,A_0)E_0\subset E_0,\]
then $E_{0c}\subset S_c'(0)$ is complemented, and $(S_c'(0),F'(0))$ is a direct sum of $(E_{0c},F_0)$ and a degenerate classical Kasparov bimodule with the Kasparov operator $F_0'$.
Then, $(E_0,F_0)$, $(E_{0c},F_0)$, $(S_c'(0),F'(0))$ are all connected by classical homotopies, and $(S_c'(0),F'(0))$ is connected by a classical homotopy $(S_c',F')$ to $(S_c'(1),F'(1))$, by symmetry we can find a classical homotopy $(E,F)$ connecting $(E_0,F_0)$ and $(E_1,F_1)$.
\end{proof}

\begin{theorem}[Separable reduction II]
Let $A_0$ and $B$ be C$^*$-algebras with $A_0$ separable.
Then,
\[\operatorname*{colim}_{B_0\subset_\mathrm{sep}B}KK(A_0,B_0)\to KK^\mathrm{fc}(A_0,B)\]
is an isomorphism in the category of abelian groups.
\end{theorem}
\begin{proof}

(Surjectivity)
It is enough to show that every element of $KK^\mathrm{fc}(A_0,B)$ belongs to an image of the natural homomorphism $KK(A_0,B_0)\to KK^\mathrm{fc}(A_0,B)$ for a separable C$^*$-subalgebra $B_0\subset B$.
By the previous separable reduction theorem (Theorem \ref{Theorem: Separable reduction I}) any element of $KK^\mathrm{fc}(A_0,B)$ can be represented by a classical Kasparov bimodule, so we fix one $(E,F)$ over $(A_0,B)$.
Choose a separable graded closed linear subspace $V_0\subset E$ such that
\[\Sigma(F,A_0)\subset C^*(\Theta_{\xi,\eta}:\xi,\eta\in V_0),\qquad C^*(F,A_0)V_0\subset V_0,\]
and take a separable C$^*$-subalgebra $B_0\subset B$ such that $\langle V_0,V_0\rangle\subset B_0$ so that we can define a Hilbert module $E_0:=\overline{V_0B_0}$ over $B_0$ generated by $V_0$.
Then, we have $\Sigma(F,A_0)\subset K(E_0)$, and the restrictions of $F$ and $A_0$ are well-defined in $B(E_0)$, so we get a classical Kasparov bimodule $(E_0,F)$ over $(A_0,B_0)$.
Its image in $KK^\mathrm{fc}(A_0,B)$ is represented by a Kasparov bimodule $(E_0\otimes_{B_0}\widetilde B=\overline{E_0B},F)$, and by the assumption that $E$ is classical, we have a homotopy $(\{\xi\in C([0,1],E):\xi(1)\in\overline{E_0B}\},F)$ with constant $F$ connecting $(E,F)$ and $(\overline{E_0B},F)$.
Therefore, the surjectivity follows.

(Injectivity)
Suppose $(E_1,F_1)$ is a classical Kasparov bimodule over $(A_0,B_0)$ for a separable C$^*$-subalgebra $B_0\subset B$ that is homotopic to zero as a classical Kasparov bimodule over $(A_0,B)$.
By the previous first separable reduction (Theorem \ref{Theorem: Separable reduction I}), we have a classical Kasparov bimodule $(E,F)$ over $(A_0,C([0,1],B))$ such that $E(0)=\{0\}$ and $(E(1),F(1))=(\overline{E_1B},F_1)$.
Choose, again as in the surjectivity proof, a separable graded closed linear subspace $V_c\subset E$ such that
\[\Sigma(F,A_0)\subset C^*(\Theta_{\xi,\eta}:\xi,\eta\in V_c),\qquad C^*(F,A_0)V_c\subset V_c,\]
and take a separable C$^*$-subalgebra $B_1\subset B$ such that $\langle V_c(t),V_c(t)\rangle\cup B_0\subset B_1$ for all $t\in[0,1]\cap\mathbb{Q}$ and $V_c(1)\subset\overline{E_1B_1}$.
From $\langle V_c,V_c\rangle\subset C([0,1],B)$ and the above conditions on $V_c$, if we let $E_c:=\overline{V_cC([0,1],B_1)}$, then we can define a classical Kasparov bimodule $(E_c,F)$ over $(A_0,C([0,1],B_1))$, satisfying $E_c(0)=\{0\}$ and $E_c(1)\subset\overline{E_1B_1}$.
Then, the classical homotopy $(E_c,F)$ connects the zero and $(E_c(1),F)$, which is again connected to $(\overline{E_1B_1},F_1)$ by a classical homotopy $(\{\xi\in C([0,1],\overline{E_1B_1}):\xi(0)\in E_c(1)\},F_1)$.
Therefore, $(\overline{E_1B_1},F_1)=(E_1\otimes_{B_0}B_1,F_1\otimes1)$ is classically null-homotopic over $(A_0,B_1)$.
\end{proof}

\begin{corollary}
Let $B$ be a C$^*$-algebra.
Then, $KK^\mathrm{fc}(\mathbb{C},B)\cong K_0(B)$.
\end{corollary}
\begin{proof}
It follows from the second separable reduction theorem and the filtered cocontinuity of operator K-theory.
\end{proof}

\subsection{Finite down set reduction}

The simplification technique introduced in here has a great importance in the construction of Kasparov product for approximate Kasparov bimodules.
As well as in the resolution of the set-theoretic size issue for the KK-groups, we force the indexing set of Kasparov operators into a fixed form, in which every index has finite down sets.
The idea of considering finite down sets was indirectly influenced by Wulff's simpltotic algebras in non-separable E-theory \cite{wulff2022generalizedasymptoticalgebrasmathrmetheory}.

The key fact for the implementation of finite down sets is that for any C$^*$-algebra $A$ the net of all separable C$^*$-subalgebras of $A$ has a cofinal system indexed by the directed set of all finite subsets of $ K(\ell^2(\mathbb{N}))\otimes\widetilde A$.
We use a result of Olsen-Zame \cite{MR388114} on singly generated C$^*$-algebras.
In fact the finiteness of down sets is superfluous because the countability down sets are sufficient to construct a Kasparov product, but we consider finite subsets because the power set of a countably infinite set is no more a countable set.

\begin{definition}[Finite down set reduction]
Let $A$ and $B$ be C$^*$-algebras.
A \emph{finite down set reduction} of an approximate Kasparov bimodule $(E,(F_{0i})_i)$ over $(A,B)$ is another approximate Kasparov bimodule $(E,(F_{1j})_j)$ together with a cofinal directed system $(A_j)_j$ of the directed set of all separable C$^*$-subalgebras of $A$, such that for each $j$ the set $\{j_0:j_0\le j\}$ is finite and there is a monotone assignment $j\mapsto i_j$ satisfying $F_{1j}=F_{0i_j}$ and $\{F_{0i}\}_{i\ge i_j}$ is Kasparov under $A_j$.
\end{definition}

\begin{lemma}[Finite down set reduction]\label{Lemma: Finite down set reduction}
Let $A$ and $B$ be C$^*$-algebras.
Every approximate Kasparov bimodule over $(A,B)$ admits an operator homotopic finite down set reduction.
\end{lemma}
\begin{proof}
Let $(E,(F_{0i})_i)$ be an approximate Kasparov bimodule over $(A,B)$.
Let $\mathcal{J}$ be the directed set of all finite subsets of $K(\ell^2(\mathbb{N}))\otimes\widetilde A$, and define $A_j:=e_{11}C^*(j,e_{11})e_{11}\cap A$ for $j\in\mathcal{J}$, where $A$ is considered to be embedded in the first corner.
Then, $A_j$ is clearly a directed system of separable C$^*$-subalgebras of $A$ such that the set $\{j_0:j_0\le j\}$ is finite for each $j$.
For $A_0\subset A$ any separable C$^*$-subalgebra, by the result by Olsen-Zame \cite{MR388114} that the tensor product of $K(\ell^2(\mathbb{N}))$ and a unital separable C$^*$-algebra is always singly generated, there is $j\in\mathcal{J}$ such that $K(\ell^2(\mathbb{N}))\otimes\widetilde A_0\subset C^*(j)$, which implies $A_0\subset A_j$ so that $A_j$ is a cofinal system.
Now for each $j$ if we take an index $i_j$ such that the family $\{F_{0i}\}_{i\ge i_j}$ is Kasparov under $A_j$ and define $F_{1j}:=F_{0i_j}$, then the proof of existence is complete.
Note that we may assume the assignment $j\mapsto i_j$ is monotone by induction on the size of $j$.

Now we prove the operator homotopy.
Consider the directed set $\mathcal{L}:=\{(i,j):i\ge i_j\}$ with product order.
For $l:=(i,j)\in\mathcal{L}$, since $\{F_{0i}\}_{i\ge i_j}$ is a Kasparov family under $A_j$, there is an operator homotopy $F_l\in C([0,1],B(E))$ under $A_j$ connecting $F_{0i}$ and $F_{0i_j}=F_{1j}$.
Then for given $j_0$, the operator homotopy $F_l$ is also Kasparov under $A_{j_0}$ for all $l\ge l_0:=(i_{j_0},j_0)$.
This implies that $(E,(F_{0i})_i)$ and $(E,(F_{1j})_j)$ are operator homotopic.
\end{proof}

\section{Kasparov product}

\subsection{Technical theorem for non-\texorpdfstring{$\sigma$}{sigma}-unital algebras}

The notion of mediators is standard in the classical KK-theory, but its name is originally introduced in here for convenience.
They are used to construct preferable Kasparov operators on the internal tensor product.

\begin{definition}[Mediators]
Let $A_0$ and $C$ be C$^*$-algebras.
Let $E$ be a graded augmented Hilbert bimodule over $(A_0,C)$.
Consider a norm compact subset $\Delta\subset B(E)$ and separable C$^*$-subalgebras $J_1,J_2\subset B(E)$.
We define a \emph{mediator} for the triple $(\Delta,J_1,J_2)$ as an even positive contraction $M\in B(E)$ such that
\[[\Delta,M],\ MJ_1,\ (1-M)J_2\subset K(E).\]
For $\widehat F_1,\widehat F_2\in B(E)$, we say the triple $(\Delta,J_1,J_2)$ \emph{supports} $(\widehat F_1,\widehat F_2)$ under $A_0$ if
\[A_0\cup\{\widehat F_1,\widehat F_2\}\subset\overline{\mathbb{C}\Delta},\qquad\Sigma(\widehat F_1,A_0)\subset J_1,\qquad\Sigma(\widehat F_2,A_0)\cup\{[\widehat F_1,\widehat F_2]\}\subset J_2.\]
We define a \emph{mediator} for $(\widehat F_1,\widehat F_2)$ as a mediator for some supporting triple of $(\widehat F_1,\widehat F_2)$ under $A_0$.
\end{definition}

\begin{proposition}
Let $A_0$ and $C$ be C$^*$-algebras.
Let $E$ be a graded augmented Hilbert bimodule over $(A_0,C)$.
If $\widehat F_1,\widehat F_2\in B(E)$ are odd adjointable operators with $C^*(\widehat F_1,\widehat F_2)E\subset EC$ such that there is a mediator $M\in B(E)$ for $(\widehat F_1,\widehat F_2)$, then
\[F_{12}:=M^{\frac12}\widehat F_1+(1-M)^{\frac12}\widehat F_2\]
is Kasparov under $A_0$.
\end{proposition}
\begin{proof}
Modulo $K(E)$, since $M$ commutes with $A_0$, $\widehat F_1$, and $\widehat F_2$, we can check
\begin{align*}
[F_{12},a]&\equiv M^{\frac12}[\widehat F_1,a]+(1-M)^{\frac12}[\widehat F_2,a]\equiv0,\\
(F_{12}-F_{12}^*)a&\equiv M^{\frac12}(\widehat F_1-\widehat F_1^*)a+(1-M)^{\frac12}(\widehat F_2-\widehat F_2^*)a\equiv0,\\
(F_{12}^2-1)a&\equiv M(\widehat F_1^2-1)a+M^{\frac12}(1-M)^{\frac12}[\widehat F_1,\widehat F_2]a+(1-M)(\widehat F_2^2-1)a\equiv0
\end{align*}
by $\Sigma(\widehat F_1,A_0)\subset J_1$ and $\Sigma(\widehat F_2,A_0)\cup\{[\widehat F_1,\widehat F_2]\}\subset J_2$, so $\Sigma(F_{12},A_0)\subset K(E)$.
The condition $C^*(F_{12})E\subset EC$ is clear.
\end{proof}

\begin{proposition}
Let $A_0$ and $C$ be C$^*$-algebras.
Let $E$ be a graded augmented Hilbert bimodule over $(A_0,C)$.
For $\widehat F_1,\widehat F_2\in B(E)$, the set of all mediators for $(\widehat F_1,\widehat F_2)$ under $A_0$ is convex.
\end{proposition}
\begin{proof}
It is easy to see that for a given triple $(\Delta,J_1,J_2)$ the set of all mediators is convex.
If $M$ and $M'$ be mediators for $(\widehat F_1,\widehat F_2)$ under $A_0$, and if they respectively mediate supporting triples $(\Delta,J_1,J_2)$ and $(\Delta',J_1',J_2')$ of $(\widehat F_1,\widehat F_2)$ under $A_0$, then $M$ and $M'$ are both mediators for the triple $(\{(n\|a_n\|)^{-1}a_n\}_n\cup\{0,\widehat F_1,\widehat F_2\},\;J_1\cap J_1',\;J_2\cap J_2')$, where $a_n\in A_0\setminus\{0\}$ is a dense sequence of $A_0$.
This implies that the set of all mediators for $(\widehat F_1,\widehat F_2)$ under $A_0$ is also convex.
\end{proof}

\begin{theorem}[Technical theorem]\label{Theorem: Technical theorem}
Let $B$ and $C$ be C$^*$-algebras.
Let $E_1$ be a graded freely complementable Hilbert module over $\widetilde B$, $E_2$ be a graded freely complementable augmented Hilbert bimodule over $(B,C)$, and $E_{12}:=E_1\otimes_BE_2$.
Consider
\begin{enumerate}
\item
a norm compact subset $\Delta_1\subset B(E_1)\otimes1$,
\item
a norm compact subset $\Delta_2\subset B(E_{12})$,
\item
a separable C$^*$-subalgebra $J_1\subset B(E_1)\otimes 1$ which is graded and $J_1E_{12}\subset E_{12}$ is complemented,
\item
a separable C$^*$-subalgebra $J_2\subset B(E_{12})$,
\end{enumerate}
such that
\[[\Delta_1,J_1]\subset J_1,\qquad[\Delta_2,J_1]\subset K(E_{12}),\qquad J_1J_2\subset K(E_{12}).\]
Then, there exists a mediator $M\in B(E_{12})$ for $(\Delta_1\cup\Delta_2,J_1,J_2)$.
\end{theorem}
\begin{proof}
Let $P_1\in B(E_{12})$ be the even projection onto $J_1E_{12}$.
Take a graded countably generated complemented Hilbert submodule $S_{12}\subset E_{12}$ over $\widetilde C$ by cropping such that $[\Delta_2,J_1]\cup J_1J_2\subset K(S_{12})$ and $S_{12}$ reduces $C^*(\Delta_1,\Delta_2,J_1,J_2,P_1)$.
Let $P_{12}\subset B(E_{12})$ be the even projection onto $S_{12}$.
Then, $P_1$ commutes with $\Delta_1$ and $J_1$, and $P_{12}$ commutes with $\Delta_1$, $\Delta_2$, $J_1$, $J_2$, and $P_1$.

By applying the classical technical theorem on the triple $(P_{12}(\Delta_1\cup\Delta_2),C^*(P_{12}J_1,K(S_{12})),P_{12}J_2)$, we can find an even $M_{12}\subset B(S_{12})$ with $0\le M_{12}\le P_{12}$ such that
\[[\Delta_1\cup\Delta_2,M_{12}],\ M_{12}J_1,\ (P_{12}-M_{12})J_2\subset K(S_{12}).\]
If we let $M_1:=(1-P_{12})(1-P_1)$ be an even projection, then it easily follows
\[[\Delta_1,M_1]=\{0\},\qquad M_1J_1=\{0\}.\]
On the other hand, since $(1-P_{12})([\Delta_2,J_1]\cup J_1J_2)=\{0\}$, if we consider an approximate unit $e_{1n}\in J_1$ of $J_1$, then $(1-P_{12})e_{1n}\to(1-P_{12})P_1=1-P_{12}-M_1$ strictly as $n\to\infty$, so we also have
\[[\Delta_2,M_1]=\{0\},\qquad(1-P_{12}-M_1)J_2=\{0\}\]
from $(1-P_{12})([\Delta_2,e_{1n}]\cup e_{1n}J_2)=\{0\}$.
Then, $M:=M_{12}+M_1$ defines a mediator for the triple $(\Delta_1\cup\Delta_2,J_1,J_2)$.
\end{proof}

\subsection{Grassmann connections and separable fibering}

\begin{definition}[Grassmann connections]\label{Definition: Grassmann connections}
Let $B$ and $C$ be C$^*$-algebras.
Let $E_1$ be a Hilbert module over $\widetilde B$, $E_2$ be an augmented Hilbert bimodule over $(B,C)$, and $E_{12}:=E_1\otimes_BE_2$.
Let $p_1\in B(H\otimes\widetilde B,E_1)$ be a coisometry for a Hilbert space $H$, which induces a coisometry $p_1\in B(H\otimes E_2,E_{12})$ with the same notation.
We define the \emph{Grassmann connection} of $F_2\in B(E_2)$ along $p_1$ as $\widehat F_2:=p_1(1\otimes F_2)p_1^*\in B(E_{12})$, sometimes denoted by $\widehat F_2^{p_1}:=\widehat F_2$ to emphasize which isometry is used.

If $E_1$ and $E_2$ are graded, and if $F_2$ is odd and $C^*(F_2)E_2\subset E_2C$, we define an \emph{$F_2$-connection} to be an odd $\widehat F_2\in B(E_{12})$ such that $C^*(\widehat F_2)E_{12}\subset E_{12}C$ and $(\widehat F_2T_{\xi_1}-T_{\overline\xi_1}F_2)T_{\eta_1}^*,(\widehat F_2^*T_{\xi_1}-T_{\overline\xi_1}F_2^*)T_{\eta_1}^*\in K(E_{12})$ for $\xi_1,\eta_1\in E_1$.
\end{definition}

\begin{definition}[Separable fibering]
Let $A$ and $B$ be C$^*$-algebras.
Let $E$ be a graded augmented Hilbert bimodule over $(A,B)$.
With respect to a single or a family of pairs $\{(F_i,A_i)\}_i$ such that $F_i\in B(E)$ is Kasparov under separable $A_i\subset A$ for each $i$, a \emph{separable fibering} is a triple $(E_0,p_0,B_0)$ of the following data:
\begin{enumerate}
\item
$E_0\subset E$ is a graded separable closed linear subspace such that for all the given pairs $(F_i,A_i)$ we have
\[\Sigma(F_i,A_i)\subset C^*(\Theta_{\xi,\eta}:\xi,\eta\in E_0),\qquad C^*(F_i,A_i)E_0\subset E_0,\]
\item
$p_0\in B(H\otimes\widetilde B,E)$ is an even coisometry such that there is a graded Hilbert subspace $H_0\subset H$ satisfying
\[\overline{E_0B^\sim}=p_0(H_0\otimes\widetilde B),\]
\item
$B_0\subset B$ is a separable C$^*$-subalgebra such that
\[E_0\widetilde B_0\subset E_0,\qquad\langle E_0,E_0\rangle\subset\widetilde B_0,\qquad p_0^*E_0\subset H_0\otimes\widetilde B_0.\]
\end{enumerate}
\end{definition}

Let $(E_0,\pi,B_0)$ be a separable fibering with respect to a pair $(F_0,A_0)$, in the setting of the above definition.
The inclusions $E_0\widetilde B\subset E_0$ and $\langle E_0,E_0\rangle\subset\widetilde B_0$ in (3) exactly mean that $E_0$ is a Hilbert module over $\widetilde B_0$, and in this case the condition $\Sigma(F_0,A_0)\subset C^*(\Theta_{\xi,\eta}:\xi,\eta\in E_0)$ in (1) is equivalent to $\Sigma(F_0,A_0)\subset K(E_0)$.
Moreover, the condition $C^*(F_0,A_0)E_0\subset E_0$ implies that $E_0$ is an augmented Hilbert bimodule over $(A_0,B_0)$ and the restriction $F_0\in B(E_0)$ is well-defined, satisfying that $C^*(F_0,A_0)E_0\subset E_0B_0$.
The Hilbert module $\overline{E_0B^\sim}$ over $\widetilde B$ generated by $E_0$ in (2) is nothing but the internal tensor product $E_0\otimes_{B_0}\widetilde B$ along the unital inclusion $\widetilde B_0\subset\widetilde B$.

\begin{lemma}
Let $A$, $B$, and $C$ be C$^*$-algebra.
Let $E_1$ be a graded freely complemented augmented Hilbert bimodule over $(A,B)$, and let $E_2$ be an augmented Hilbert bimodule over $(B,C)$.
Let $(E_{10},p_1,B_0)$ be a separable fibering with respect to some Kasparov operators on $E_1$.
Then, $E_{10}\otimes_{B_0}E_2\subset E_1\otimes_BE_2$ is complemented, and the projection in $B(E_1\otimes_BE_2)$ onto $E_{10}\otimes_{B_0}E_2$ commutes with every Grassmann connection of $F_2\in B(E_2)$ along $p_1$.
\end{lemma}
\begin{proof}
Let $E_{12}:=E_1\otimes_BE_2$ and $E_{102}:=E_{10}\otimes_{B_0}E_2$.
We have a coisometry $p_1\in B(H\otimes\widetilde B,E_1)$ by the condition (2), and we use the same notation for the induced coisometry $p_1\in B(H\otimes E_2,E_{12})$.
Let $P_1,P_0\in B(H\otimes E_2)$ be projections onto $p_1^*E_{12}$ and $H_0\otimes E_2$ respectively.
Then, clearly $P_1=p_1^*p_1$.
By the equality of (2) and the last inclusion of (3), we have
\[P_1P_0(H\otimes E_2)=p_1^*p_1P_0(H\otimes E_2)=p_1^*p_1(H_0\otimes E_2)=p_1^*(E_{10}\otimes_{B_0}E_2)\subset H_0\otimes E_2=P_0(H\otimes E_2),\]
so $(1-P_0)P_1P_0=0$ implies that $P_1$ and $P_0$ commute.
Then, by the equality of (2), we have
\begin{align*}
E_{102}&=E_{10}\otimes_{B_0}E_2=E_{10}\otimes_{B_0}\widetilde B\otimes_BE_2=\overline{E_{10}B^\sim}\otimes_BE_2=p_1(H_0\otimes E_2)=p_1P_0(H\otimes E_2)\\
&=p_1P_1P_0(H\otimes E_2)=p_1P_0P_1(H\otimes E_2)=p_1P_0p_1^*p_1(H\otimes E_2)=p_1P_0p_1^*E_{12}.
\end{align*}
This proves that $E_{102}$ is complemented in $E_{12}$, with the support projection $p_1P_0p_1^*\in B(E_{12})$.
Since $P_0$ and $1\otimes F_2$ clearly commute, we can easily see the Grassmann connection $p_1(1\otimes F_2)p_1^*$ commutes with $p_1P_0p_1^*$.
\end{proof}

Separable fiberings are devices for taking appropriate Grassmann connections.
The data of $p_0$ and $B_0$ are not used if a net of connections is already given as in the proof of the uniqueness of Kasparov products, but they are useful when we construct Grassmann connections.
We prove the related results on the existence of separable fiberings and connections in the below.

\begin{lemma}[Existence of separable fiberings]\label{Lemma: Existence of separable fiberings}
Let $A$ and $B$ be C$^*$-algebras.
Let $E$ be a graded freely complementable augmented Hilbert bimodule over $(A,B)$.
Let $(F_n,A_n)_n$ be a sequence of pairs such that $F_n\in B(E)$ is Kasparov under separable $A_n\subset A$.
With respect to this sequence, there exists a separable fibering $(E_0,p_0,B_0)$ with arbitrary $p_0$.
\end{lemma}
\begin{proof}
Let $p_0\in B(H\otimes\widetilde B,E)$ be any even coisometry.
Since $\bigcup_n\Sigma(F_n,A_n)\subset K(E)$ generates a $\sigma$-unital C$^*$-subalgebra and $\bigcup_nC^*(A_n,F_n)\subset B(E)$ generates a separable C$^*$-subalgebra, we can crop out a graded countably generated complemented Hilbert submodule $S\subset E$ over $\widetilde B$ with respect to the embedding $p_0^*:E\to H\otimes\widetilde B$ such that $\Sigma(F_n,A_n)\subset K(S)$ and $C^*(A_n,F_n)$ is reduced by $S$ for all $n$.
We can take the projection $E\to S$ to be induced from $B(H)$ so that we may write $S=p_0(H_0\otimes\widetilde B)$ for a graded Hilbert subspace $H_0\subset H$ because the cropping is done by choosing appropriate basis vectors of $H$ as seen in the proof of the cropping lemma.

Choose a graded separable closed linear subspace $V_0\subset S$ containing a generating sequence of $S$ as a Hilbert module over $\widetilde B$, and such that we further have
\[\Sigma(F_n,A_n)\cup\{h\}\subset C^*(\Theta_{\xi,\eta}:\xi,\eta\in V_0),\qquad C^*(A_n,F_n)V_0\subset V_0\]
for all $n$, where $h$ is a strictly positive element of $K(S)$.
Define a separable C$^*$-subalgebra $B_0\subset B$ such that we have $b_i\in\widetilde B_0$ whenever we expand $\xi\in V_0$ into $p_0^*\xi=\sum_{i=1}^\infty\varepsilon_i\otimes b_i\in H\otimes\widetilde B$.

If we define a Hilbert module $E_0:=\overline{V_0B_0^\sim}$ over $\widetilde B_0$, then
\[\Sigma(F_n,A_n)\subset K(E_0),\qquad C^*(F_n,A_n)\subset B(E_0)\]
for all $n$.
Furthermore, we can check $\overline{E_0B^\sim}=\overline{V_0B^\sim}=S=p_0(H_0\otimes\widetilde B)$ and $p_0^*E_0\subset p_0^*S\cap(H\otimes\widetilde B_0)\subset H_0\otimes\widetilde B_0$.
\end{proof}

\begin{lemma}[Existence of connections]
Let $A$, $B$, and $C$ be C$^*$-algebras.
Let $E_1$ and $E_2$ be graded augmented Hilbert bimodules over $(A,B)$ and $(B,C)$ respectively, and let $E_{12}:=E_1\otimes_BE_2$.
Let $(E_{10},p_1,B_0)$ be a separable fibering, with respect to some Kasparov operators on $E_1$.
If an odd adjointable $F_2\in B(E_2)$ commutes with $B_0$ modulo $K(E_2)$ and $C^*(F_2)E_2\subset E_2C$, then the Grassmann connection $\widehat F_2\in B(E_{12})$ of $F_2$ along $p_1$ is indeed an $F_2$-connection on $E_{10}$.
Furthermore, if $F_2$ is norm continuous, where we suppose $E_2$ and $C$ are trivial continuous fields over a compact Hausdorff space, then so is $\widehat F_2$ with a fixed coisometry.
\end{lemma}
\begin{proof}
Let $\xi_1,\eta_1\in E_{10}$, and write $p_1^*\xi_1=\sum_{n=1}^\infty\varepsilon_n\otimes b_n$ with $\varepsilon_n\in H_0$ and $b_n\in\widetilde B_0$ in the coisometry $p_1\in B(H\otimes\widetilde B,E_1)$, where $H_0\subset H$ is an inclusion of graded Hilbert spaces.
Since one can compute
\[\widehat F_{2k}T_{\xi_1}T_{\eta_1}^*(\zeta_1\otimes\zeta_2)=\sum_{n=1}^\infty p_1(\overline\varepsilon_n\otimes F_{2k}b_n\langle\overline\eta_1,\zeta_1\rangle\zeta_2),\qquad T_{\overline\xi_1}F_{2k}T_{\eta_1}^*(\zeta_1\otimes\zeta_2)=\sum_{n=1}^\infty p_1(\overline\varepsilon_n\otimes b_nF_{2k}\langle\overline\eta_1,\zeta_1\rangle\zeta_2)\]
for $\zeta_1\in E_1$ and $\zeta_2\in E_2$, it follows from $[F_2,b_n]\in K(E_2)$ that
\[(\widehat F_2T_{\xi_1}-T_{\overline\xi_1}F_2)T_{\eta_1}^*=\sum_{n=1}^\infty p_1T_{\overline\varepsilon_n}[F_2,b_n]T_{\eta_1}^*\in K(E_{12}).\]
We can get the same for the adjoint $\widehat F_2^*$, so $\widehat F_2$ is an $F_2$-connection on $E_{10}$.
If $F_2$ is odd and norm continuous, then $\widehat F_2$ is odd because $p_{10}$ is even, and the norm continuity is clear by the construction of Grassmann connection $\widehat F_2$.
\end{proof}

\begin{lemma}[Three-step mediation]
Let $A$, $B$, and $C$ be C$^*$-algebras.
Let $E_1$ and $E_2$ be graded freely complementable augmented Hilbert bimodules over $(A,B)$ and $(B,C)$ respectively, and let $E_{12}:=E_1\otimes_BE_2$.
Consider
\begin{enumerate}
\item $\{\widehat F_{1j}\}_j\subset B(E_{12})$ a norm compact family given by $\widehat F_{1j}:=F_{1j}\otimes1$ for each $F_{1j}\in B(E_1)$ Kasparov under $A_j$ with $\bigcup_jA_j$ separable,
\item $(E_{10},p_1,B_0)$ a separable fibering with respect to $\{(F_{1j},A_j)\}_j$,
\item $\{\widehat F_{2k}\}_k\subset B(E_{12})$ a norm compact family such that $\widehat F_{2k}$ is an $F_{2k}$-connection on $E_{10}$ for each $F_{2k}\in B(E_2)$ Kasparov under $B_0$.
\end{enumerate}
Then, there exists a mediator $M\in B(E_{12})$ for $(\widehat F_{1j},\widehat F_{2k})$ under $A_j$ such that the Kasparov operator $F_{12(j,k)}:=M^{\frac12}\widehat F_{1j}+(1-M)^{\frac12}\widehat F_{2k}$ over $(A_j,C)$ is $F_{1j}$-positive under $A_j$ and is an $F_{2k}$-connection over $E_{10}$, for all $j$ and $k$.
\end{lemma}
\begin{proof}
Letting $a_n\in A\setminus\{0\}$ be a dense sequence of a separable C$^*$-algebra $A_0:=C^*(\bigcup_jA_j)\subset A$, define norm compact subsets
\[\Delta_1:=\{(n\|a_n\|)^{-1}a_n\}_n\cup\{0\}\cup\{\widehat F_{1j}\}_j\subset B(E_1)\otimes1,\qquad\Delta_2:=\{\widehat F_{2k}\}_k\subset B(E_{12}),\]
and separable C$^*$-subalgebras
\[J_1:=C^*(\Theta_{\xi_1,\eta_1}\otimes1=T_{\xi_1}T_{\eta_1}^*:\xi_1,\eta_1\in E_{10})\subset K(E_1)\otimes1,\qquad J_2:=C^*(\Sigma(\widehat F_{2k},A_0),[\widehat F_{1j},\widehat F_{2k}]:\forall j,k)\subset B(E_{12}).\]
Then, $J_1E_1=\overline{E_{10}B^\sim}$ is complemented in $E_1$ by the second condition in the definition of separable fiberings, and because $E_{10}$ is graded $J_1$ is also graded.

First, $[\Delta_1,J_1]\subset J_1$ is clear from $\Delta_1E_{10}\subset E_{10}$.
By the assumption that $\widehat F_{2k}$ is an $F_{2k}$-connection on $E_{10}$ and $F_{2k}$ is Kasparov under $B_0$, we obtain $[\Delta_2,J_1]\subset K(E_{12})$ and $J_1J_2\subset K(E_{12})$ by computations.
Apply the technical theorem (Theorem \ref{Theorem: Technical theorem}) to get a mediator $M\in B(E_{12})$ for $(\Delta_1\cup\Delta_2,J_1,J_2)$, supporting all pairs $(\widehat F_{1j},\widehat F_{2k})$ for all $j$ and $k$.

The $F_{1j}$-positivity follows from
\[a^*[F_{12(j,k)},F_{1j}\otimes1]a\equiv a^*M^{\frac14}[\widehat F_{1j},F_{1j}\otimes1]M^{\frac14}a+a^*(1-M)^{\frac12}[\widehat F_{1j},\widehat F_{2k}]a\equiv2|\widehat F_{1j}M^{\frac14}a|^2\ge0\]
modulo $K(E_{12})$ for all $a\in A_j$.
Next, $F_{12(j,k)}$ is an $F_{2k}$-connection on $E_{10}$ by
\[(F_{12(j,k)}T_{\xi_1}-T_{\overline\xi_1}F_{2k})T_{\eta_1}^*\equiv M^{\frac12}\widehat F_{1j}T_{\xi_1}T_{\eta_1}^*+(1-M)^{\frac12}\widehat F_{2k}T_{\xi_1}T_{\eta_1}^*-T_{\overline\xi_1}F_{2k}T_{\eta_1}^*\equiv0\]
modulo $K(E_{12})$ for all $\xi_1,\eta_1\in E_{10}$ together with the similar computation $(F_{12l}^*T_{\xi_1}-T_{\overline\xi_1}F_{2k}^*)T_{\eta_1}^*\equiv0$ for the adjoints, by $M^{\frac12}T_{\xi_1}T_{\eta_1}^*\in M^{\frac12}J_1\subset K(E_{12})$ and $(1-M)^{\frac12}-1=-((1-M)^{\frac12}+1)^{-1}M$.
\end{proof}

\subsection{Construction of Kasparov product}

\begin{definition}[Kasparov product]
Let $A$, $B$, and $C$ be C$^*$-algebras.
Let $(E_1,(F_{0i})_i)$ and $(E_2,(F_{2k})_k)$ be approximate Kasparov bimodules over $(A,B)$ and $(B,C)$ respectively, and let $E_{12}=E_1\otimes_BE_2$.
Their \emph{Kasparov product} is an approximate Kasparov bimodule $(E_{12},(F_{12l})_l)$ over $(A,C)$, together with
\begin{enumerate}
\item a constant $\mathcal{C}>0$,
\item monotone cofinal functions $l\mapsto i(l)$ and $l\mapsto k(l)$,
\item an even coisometry $p_1\in B(H\otimes\widetilde B,E_1)$ for a graded Hilbert space $H$,
\end{enumerate}
which satisfy the following statement.
For every separable C$^*$-subalgebra $A_0\subset A$ there is $l_0$ such that whenever $l\ge l_0$, there is a quadruple $(F_{1l},E_{1l},p_1,B_l)$ such that $(E_{1l},p_1,B_l)$ is a separable fibering with respect to a Kasparov operator $F_{1l}\in B(E_1)$ under $A_0$ with a bound $\|F_{1l}\|\le\mathcal{C}$, satisfying that
\begin{enumerate}
\item
$F_{1l}$ is operator homotopic to $F_{0i(l)}$ under $A_0$,
\item
$F_{12l}$ is $F_{1l}$-positive under $A_0$,
\item
$F_{12l}$ is an $F_{2k(l)}$-connection on $E_{1l}$,
\item
$\{F_{2k}\}_{k\ge k(l)}$ is Kasparov under $B_l$.
\end{enumerate}
\end{definition}

If we take larger $l_0$ for each separable $A_0\subset A$, we may always assume $\{F_{0i}\}_{i\ge i(l_0)}$ is Kasparov under $A_0$.
It is immediate to prove the following propositions on special cases of Kasparov products.

\begin{proposition}[Kasparov product with $*$-homomorphisms]
Let $A$, $B$, and $C$ be C$^*$-algebras.
Let $(E_1,(F_{1j})_j)$ and $(E_2,(F_{2k})_k)$ be approximate Kasparov bimodules over $(A,B)$ and $(B,C)$ respectively, and let $\varphi_1:A\to B$ and $\varphi_2:B\to C$ be $*$-homomorphisms.
Then, $(E_2,(F_{2k})_k)$ over $(A,C)$ is itself a Kasparov product of $\varphi_1$ and $(E_2,(F_{2k})_k)$, and $(E_1\otimes_B\widetilde C,(F_{1j}\otimes1)_j)$ over $(A,C)$ is a Kasparov product of $(E_1,(F_{1j})_j)$ and $\varphi_2$.
\end{proposition}

\begin{proposition}[Classical Kasparov product]
Let $A_0$, $B_0$, and $C$ be C$^*$-algebras.
Let $(E_1,F_1)$ and $(E_2,F_2)$ be classical Kasparov bimodules over $(A_0,B_0)$ and $(B_0,C)$ respectively.
If $(E_{12},F_{12})$ is a Kasparov bimodule over $(A_0,C)$ that is a classical Kasparov product, meaning that $E_{12}:=E_1\otimes_{B_0}E_2$ and $F_{12}$ is an $F_1$-positive $F_2$-connection, under $A_0$ and on $E_1$, then it is a Kasparov product.
\end{proposition}

We now construct the Kasparov product for approximate Kasparov bimodules.

\begin{theorem}[Existence of Kasparov products]
Let $A$, $B$, and $C$ be C$^*$-algebras.
A Kasparov product of approximate Kasparov bimodules over $(A,B)$ and $(B,C)$ always exists.
\end{theorem}
\begin{proof}
Let $(E_1,(F_{0i})_i)$ and $(E_2,(F_{2k})_k)$ be approximate Kasparov bimodules over $(A,B)$ and $(B,C)$ respectively.
Let $(E_1,(F_{1j})_j)$ be a finite down set reduction of $(E_1,(F_{0i})_i)$.
By definition of finite down set reductions, associated to this we have a direct system $A_j$ of separable C$^*$-subalgebras of $A$ such that $\{F_{1j}\}_{j\ge j_0}$ is Kasparov under $A_{j_0}$ for any $j_0$, and a monotone assignment $j\mapsto i_j$ such that $F_{1j}=F_{0i_j}$ and $\{F_{0i}\}_{i\ge i_j}$ is Kasparov under $A_j$.
For each fixed $j$,
\begin{enumerate}
\item taking a normalized operator homotopy $F_{1j_0j}$ under $A_{j_0}$ connecting $F_{1j_0}$ and $F_{1j}$ for each $j_0\le j$,
\item if we choose a separable fibering $(E_{1j},p_1,B_j)$ with respect to $\{(F_{1j_0j}(s),A_{j_0}):j_0\le j,\;s\in[0,1]\}$,
\item we can take a Kasparov family $\{F_{2k}\}_{k\ge k_j}$ under $B_j$ with some large $k_j$.
\end{enumerate}
We define a directed set $\mathcal{L}:=\{(i,j,k):i\ge i_j,\;k\ge k_j\}$ with the product order.
We will construct a net $F_{12l}\in B(E_{12})$ parametrized by $\mathcal{L}$.
For fixed $l=(i,j,k)\in\mathcal{L}$, let $\widehat F_{1j_0j}:=F_{1j_0j}\otimes1$ for $j_0\le j$, and define $\widehat F_{2k}$ as the Grassmann $F_{2k}$-connection along $p_1$.
Then, there exists a mediator $M_l\in B(E_{12})$ for all pairs $(\widehat F_{1j_0j}(s),\widehat F_{2k})$ with $s\in[0,1]$ and $j_0\le j$, under $A_{j_0}$ for each $j_0$.
With $\widehat F_{1j}:=F_{1j}\otimes1$, we define a Kasparov operator
\[F_{12l}:=M_l^{\frac12}\widehat F_{1j}+(1-M_l)^{\frac12}\widehat F_{2k}\in B(E_{12}),\]
which is $F_{1j}$-positive under $A_j$ and is an $F_{2k}$-connection on $E_{1j}$.

We first prove that the net $F_{12l}$ defines an approximate Kasparov bimodule over $(A,C)$.
The difficult part is the Kasparov family condition.
Choose any $l_0:=(i_0,j_0,k_0)\in\mathcal{L}$ and fix $l=(i,j,k)\ge l_0$.
It suffices to show that $F_{12l_0}$ and $F_{12l}$ are operator homotopic Kasparov operators on $E_{12}$ under $A_{j_0}$.
Observe that
\begin{enumerate}
\item
for a Kasparov operator $F_{1j_0}$ under $A_{j_0}$,
\item
using the separable fibering $(E_{1j_0},p_1,B_{j_0})$ we chosen in the above, which is with respect to $(F_{1j_0},A_{j_0})$,
\item
we can find an operator homotopy $F_{2k_0k}$ under $B_{j_0}$ connecting $F_{2k_0}$ and $F_{2k}$ by the assumption $k_{j_0}\le k_0\le k$.
\end{enumerate}
Be cautious that the path $F_{2k_0k}$ may depend on the choice of $j_0$ even though the endpoints $F_{2k_0}$ and $F_{2k}$ do not.
Take a norm continuous path $\widehat F_{2k_0k}\in C([0,1],B(E_{12}))$ as the Grassmann $F_{2k_0k}$-connection along $p_1$, so that $\widehat F_{2k_0k}(s)$ is an $F_{2k_0k}(s)$-connection on $E_{1j_0}$ for all $s\in[0,1]$.
Then, there is a mediator $M_{l_0l}\in B(E_{12})$ for all pairs $(\widehat F_{1j_0},\widehat F_{2k_0k}(s))$ with $s\in[0,1]$ under $A_{j_0}$.
Define a norm continuous path $F_{12l_0l}\in C([0,1],B(E_{12}))$ such that
\[F_{12l_0l}(s):=\begin{cases}
M_{l_0l}^{\frac12}\widehat F_{1j_0}+(1-M_{l_0l})^{\frac12}\widehat F_{2k_0k}(4s-1)&\text{ if }s\in[\frac14,\frac12],\\[5pt]
M_l^{\frac12}\widehat F_{1j_0j}(4s-3)+(1-M_l)^{\frac12}\widehat F_{2k}&\text{ if }s\in[\frac34,1],
\end{cases}\]
and on $s\in[0,\frac14]\cup[\frac12,\frac34]$ the path is defined by interpolating mediators linearly using the convexity of the set of mediators.
Then, we can check the path $F_{12l_0l}$ is Kasparov under $A_{j_0}$ by evaluating at each $s\in[0,1]$ to show the support algebras belong to $K(E_{12})$, so $F_{12l_0}$ and $F_{12l}$ are connected under $A_{j_0}$ by an operator homotopy $F_{12l_0l}$.

Now we prove that the net $F_{12l}$ defines a Kasparov product.
Let $l_0:=(i_0,j_0,k_0)\in\mathcal{L}$.
For each $l:=(i,j,k)\ge l_0$, consider the separable fibering $(E_{1l},p_1,B_l):=(E_{1j},p_1,B_j)$ we have chosen, which is with respect to a Kasparov operator $F_{1l}:=F_{1j}=F_{0i_j}$ under $A_j$.
First, $F_{1l}$ is operator homotopic to $F_{0i}$ under $A_j$ by $i\ge i_j$ with the trivial bound $\|F_{1l}\|\le\sup_i\|F_{0i}\|$.
Next, we know by its construction that $F_{12l}$ is a $F_{1j}$-positive $F_{2k}$-connection under $A_j$ and on $E_{1j}$ respectively.
Finally, the family $\{F_{2k'}\}_{k'\ge k}$ is Kasparov under $B_j$ also by $k\ge k_j$.
Therefore, the net $F_{12l}$ defines our desired Kasparov product.
\end{proof}

The above proof also shows the following narrowly tuned form for the case the finite down set reduction is already applied.
It proves the existence of Kasparov product with an additional assumption for a prescribed family of separable fiberings, and will be used in the well-definedness and associativity proofs.

\begin{corollary}\label{Corollary: Tuned form}
Let $A$, $B$, and $C$ be C$^*$-algebras.
Let $(E_1,(F_{1j})_j)$ and $(E_2,(F_{2k})_k)$ be approximate Kasparov bimodules over $(A,B)$ and $(B,C)$ respectively, already given by finite down set reductions equipped with the cofinal systems $A_j\subset A$ and $B_k\subset B$ of separable C$^*$-subalgebras respectively.
Suppose we are given for each $j$ a separable fibering $(E_{1j},p_1,B_j)$ with respect to the family $\{(F_{1j_0j}(s),A_{j_0}):j_0\le j,\ s\in[0,1]\}$, where $F_{1j_0j}$ is an operator homotopy under $A_{j_0}$ connecting $F_{1j_0}$ and $F_{1j}$ taken for each pair $j_0\le j$.
Then, there is a Kasparov product $(E_{12},(F_{12(j,k)})_{(j,k)})$ together with a monotone assignment $j\mapsto k_j$, where $(j,k)$ parametrizes $\{(j,k):k\ge k_j\}$, such that $\{F_{12(j,k)}\}_{(j,k)\ge(j_0,k_{j_0})}$ is Kasparov under $A_{j_0}$ for any $j_0$, and the four conditions in the definition of Kasparov product are satisfied by the quadruples $\{(F_{1j},E_{1j},p_1,B_j)\}_{(j,k)\ge (j_0,k_0)}$.
Furthermore, we may have $F_{12(j,k)}=\widehat F_{1j}$ on the complementary summand of $\overline{E_{1j}B^\sim}\otimes_BE_2\subset E_{12}$ for all $(j,k)$ with $k\ge k_j$.
\end{corollary}
\begin{proof}
We only prove the last assertion $F_{12l}=\widehat F_{1j}$ on $(1-P_{1j})E_{12}$ can be obtained without loss of generality, where $P_{1j}\in B(E_{12})$ be the projection onto $\overline{E_{1j}B^\sim}\otimes_BE_2\subset E_{12}$.
The others clearly follow in the proof of the previous theorem.
Recall the definition $F_{12l}=M_l^{1/2}\widehat F_{1j}+(1-M_l)^{1/2}\widehat F_{2k}$ in the existence proof above, where $\widehat F_{1j}:=F_{1j}\otimes1\in B(E_{12})$ is the trivial extension and $\widehat F_{2k}\in B(E_{12})$ is the Grassmann $F_{2k}$-connection along $p_1$.
First, since $C^*(F_{1j},A_j)E_{1j}\subset E_{1j}$, we can easily see $P_{1j}$ commutes with $A_j$ and $F_{1j}$.
Next, we have a graded Hilbert subspace $H_0\subset H$ such that $\overline{E_{1j}B^\sim}=p_1(H_0\otimes\widetilde B)$, so since the Grassmann connection $\widehat F_{2k}$ along $p_1$ only acts on the right tensor factor $E_2$ in $p_1^*E_{12}=p_1^*p_1(H\otimes E_2)\subset H\otimes E_2$, we can also see $P_{1j}$, which only acts on the left tensor factor $H$ projecting to $H_0$, commutes with $\widehat F_{2k}$.

Then, by letting the triple $(\Delta,J_1,J_2)$ be minimal to support $(\widehat F_{1j},\widehat F_{2k})$ under $A_j$ so that $P_{1j}$ also commutes with $\Delta\cup J_2$ and $J_1\le P_{1j}$ by $\Sigma(\widehat F_{1j},A_j)\subset C^*(\Theta_{\xi_1,\eta_1}:\xi_1,\eta_1\in E_{1j})$, we can check that $M_l':=P_{1j}M_lP_{1j}+(1-P_{1j})$ is also a mediator by computing that $[\Delta,M']$, $M'J_1$, and $(1-M')J_2$ are all belonging to $K(E_{12})$.
The replacement
\[F_{12l}':=M_l'^{\frac12}\widehat F_{1j}+(1-M_l')^{\frac12}\widehat F_{2k}\]
of $F_{12l}$ then satisfies $F_{12l}'=\widehat F_{1j}$ on $(1-P_{1j})E_{12}$.
\end{proof}

\begin{theorem}[Uniqueness of Kasparov products]
Let $A$, $B$, and $C$ be C$^*$-algebras.
A Kasparov product of approximate Kasparov bimodules over $(A,B)$ and $(B,C)$ is unique up to operator homotopy.
\end{theorem}
\begin{proof}
Let $(E_1,(F_{0i})_i)$ and $(E_2,(F_{2k})_k)$ be approximate Kasparov bimodules over $(A,B)$ and $(B,C)$ respectively.
Let $j$ be the index parameter for a cofinal system in the directed set of all separable C$^*$-subalgebras of $A$, and let $A_j\subset A$ be the corresponding separable C$^*$-subalgebra.
Let $(E_{12},(F_{12l})_l)$ and $(E_{12},(F_{12l'}')_{l'})$ be Kasparov products of $(E_1,(F_{0i})_i)$ and $(E_2,(F_{2k})_k)$.
For each $j$, we define $i_j$, $l_j$, $l_j'$, and $k_j$ as follows.
\begin{itemize}
\item
Choose $i_j$ such that $\{F_{0i}\}_{i\ge i_j}$ is Kasparov under $A_j$, and let $F_{1j}:=F_{0i_j}$.
It is important to note that the choices of $i_j$ and $F_{1j}$ do not depend on any of both Kasparov products.
\item
Find $l_j$ with the property that, for any $l\ge l_j$ we have a separable fibering $(E_{1l},p_1,B_l)$ with respect to $(F_{1l},A_j)$, such that $F_{1l}$ is operator homotopic to $F_{0i(l)}$ under $A_j$, $F_{12l}$ is $F_{1l}$-positive under $A_j$ and $F_{12l}$ is an $F_{2k(l)}$-connection on $E_{1l}$, and finally the family $\{F_{2k}\}_{k\ge k(l)}$ is Kasparov under $B_l$, just following the definition of Kasparov products.
Here we may assume the family $\{F_{12l}\}_{l\ge l_j}$ is Kasparov under $A_j$ by taking larger $l_j$.
We have $p_1'$ and $l_j'$ similarly.
\item
With respect to a single Kasparov operator $F_{1j}$ over $A_j$, take two separable fiberings $(E_{1j},p_1,B_j)$ and $(E_{1j}',p_1',B_j')$ with the given $p_1$ and $p_1'$, and introduce a new separable fibering $(E_{1j}'',p_1'',B_j'')$ such that $E_{1j}+E_{1j}'\subset E_{1j}''$ and $C^*(B_j,B_j')\subset B_j''$.
Take $k_j$ such that $\{F_{2k}\}_{k\ge k_j}$ is Kasparov under $B_j''$.
It is also important to note that $k_j$ is chosen symmetrically to the two Kasparov products.
\end{itemize}
Until the end of the proof, we consider a fixed index $m:=(j,l,l')$, which parametrizes the directed set
\[\mathcal{M}:=\{(j,l,l'):i(l),i(l')\ge i_j\text{ and }(l,l')\ge(l_j,l_j')\text{ and }k(l),k(l')\ge k_j\}.\]
We will also assume that when we say two Kasparov operators are just operator homotopic, or we say an operator or a family is just Kasparov, they are all under $A_j$, as well as every mediator is assumed to be under $A_j$.

Now, it suffices to construct an operator homotopy, depending on but fixed $m$, which connects $F_{12l}$ and $F_{12l'}'$.
We construct three mediators $M_2$, $M_3$, and $M_1$ in $B(E_{12})$ as follows.
First,
\begin{enumerate}
\item
if we let $F_{1li(l)}\in C([0,1],B(E_1))$ be a normalized operator homotopy connecting $F_{1l}$ and $F_{0i(l)}$, and let $F_{0i(l)i_j}\in C([0,1],B(E_1))$ be a normalized operator homotopy connecting $F_{0i(l)}$ and $F_{1j}$,
\item
then choosing a suitable separable fibering along $p_1$ with respect to these two paths under $A_j$,
\item
and finding sufficiently large $l_1\ge l$ to make $F_{2k(l_1)}$ Kasparov under $B_0$, where $B_0\subset B$ is the separable C$^*$-subalgebra taken in the fibering,
\end{enumerate}
we can construct a mediator $M_2$ for all pairs
\[(\widehat F_{1li(l)}(s),\widehat F_{2k(l_1)}^{p_1}),\qquad(\widehat F_{0i(l)i_j}(s),\widehat F_{2k(l_1)}^{p_1}),\qquad s\in[0,1].\]
Take $l_1'$ similarly.
Next,
\begin{enumerate}
\item
for a single Kasparov operator $F_{1j}$ under $A_j$,
\item
using the separable fiberings $(E_{1j},p_1,B_j)$, $(E_{1j}',p_1',B_j')$, and $(E_{1j}'',p_1'',B_j'')$ with respect to $(F_{1j},A_j)$,
\item
since $\{F_{2k}\}_{k\ge k_j}$ is Kasparov under $B_j''$ (and hence $C^*(B_j,B_j')$) by the assumption $k_j\le k(l)\le k(l_1),k(l_1')$,
\end{enumerate}
we can construct a mediator $M_3$ for all pairs
\[(\widehat F_{1j},\widehat F_{2k_jk(l_1)}^{p_1}(s)),\qquad(\widehat F_{1j},(1-s)\widehat F_{2k_j}^{p_1}+s\widehat F_{2k_j}^{p_1''}),\qquad s\in[0,1],\]
and another mediator $M_3'$ for all pairs
\[(\widehat F_{1j},\widehat F_{2k_jk(l_1')}^{p_1'}(s)),\qquad(\widehat F_{1j},(1-s)\widehat F_{2k_j}^{p_1'}+s\widehat F_{2k_j}^{p_1''}),\qquad s\in[0,1].\]
Finally,
\begin{enumerate}
\item
for a single Kasparov operator $F_{1l}$ under $A_j$,
\item
using the separable fibering $(E_{1l},p_1,B_{1l})$, which is with respect to $(F_{1l},A_j)$,
\item
because $F_{12l}$ and $\widehat F_{2k(l)}^{p_1}$ are both $F_{2k(l)}$-connections on $E_{1l}$, and because we have a Grassmann $F_{2k(l)k(l_1)}$-connection $\widehat F_{2k(l)k(l_1)}^{p_1}\in C([0,1],B(E_{12}))$ on $E_{1l}$,
\end{enumerate}
we can construct a mediator $M_1$ for all pairs
\[(F_{1l},(1-s)F_{12l}+s\widehat F_{2k(l)}^{p_1}),\qquad(F_{1l},\widehat F_{2k(l)k(l_1)}^{p_1}(s)),\qquad s\in[0,1].\]

Then, since $F_{12l}$ is $F_{1l}$-positive under $A_j$, we can check $F_{12l}$ is Connes-Skandalis operator homotopic (Theorem \ref{Theorem: Connes-Skandalis operator homotopy}) to a Kasparov operator $M_1^{1/2}\widehat F_{1l}+(1-M_1)^{1/2}F_{12l}$, which is connected to $M_3^{1/2}\widehat F_{1j}+(1-M_3)^{1/2}\widehat F_{2k_j}^{p_1''}$ by an operator homotopy $F''\in C([0,1],B(E_{12}))$ defined such that
\[F''(s):=\begin{cases}
M_1^{\frac12}\widehat F_{1l}+(1-M_1)^{\frac12}((1-8s)F_{12l}+8s\widehat F_{2k(l)}^{p_1})&\text{ if }s\in[0,\frac18],\\
M_1^{\frac12}\widehat F_{1l}+(1-M_1)^{\frac12}\widehat F_{2k(l)k(l_1)}^{p_1}(8s-1)&\text{ if }s\in[\frac18,\frac28],\\
M_2^{\frac12}\widehat F_{1li(l)}(8s-3)+(1-M_2)^{\frac12}\widehat F_{2k(l_1)}^{p_1}&\text{ if }s\in[\frac38,\frac48],\\
M_2^{\frac12}\widehat F_{0i(l)i_j}(8s-4)+(1-M_2)^{\frac12}\widehat F_{2k(l_1)}^{p_1}&\text{ if }s\in[\frac48,\frac58],\\
M_3^{\frac12}\widehat F_{1j}+(1-M_3)^{\frac12}\widehat F_{2k_jk(l_1)}^{p_1}(7-8s)&\text{ if }s\in[\frac68,\frac78],\\
M_3^{\frac12}\widehat F_{1j}+(1-M_3)^{\frac12}((8-8s)\widehat F_{2k_j}^{p_1}+(8s-7)\widehat F_{2k_j}^{p_1''})&\text{ if }s\in[\frac78,1],\\
\end{cases}\]
and on $s\in[\frac28,\frac38]\cup[\frac58,\frac68]$ the path is defined by connecting mediators linearly.
Then, since the operators $\widehat F_{1j}$ and $\widehat F_{2k_j}^{p_1''}$ used to represent the final endpoint is constructed symmetrically in the sense that it remains as it is even if we switch $l$ and $l'$ each other during the construction, and since both operators $M_3$ and $M_3'$ mediate $\widehat F_{1j}$ and $\widehat F_{2k_j}^{p_1''}$ so that these two mediators can be connected linearly, it is also operator homotopic to $F_{12l'}'$ in a symmetric way.
\end{proof}

\begin{theorem}[Well-definedness of Kasparov products]
Let $A$, $B$, and $C$ be C$^*$-algebras.
A Kasparov product of approximate Kasparov bimodules over $(A,B)$ and $(B,C)$ defines a monoid homomorphism $KK^\mathrm{fc}(A,B)\otimes KK^\mathrm{fc}(B,C)\to KK^\mathrm{fc}(A,C)$.
\end{theorem}
\begin{proof}
We prove that Kasparov products of homotopic approximate Kasparov bimodules over $(A,B)$ and $(B,C)$ are homotopic by dividing homotopy moves and degeneracy moves.
The biadditivity will be separately proved during the proof because it should be used for the addition in the degeneracy moves.

For the homotopy moves, similarly to the classical KK-theory for separable C$^*$-algebras, it follows from that the defining conditions of Kasparov product for approximate Kasparov bimodules are all invariant under the evaluation at each point $0,1\in[0,1]$.
Use the existence and uniqueness of a Kasparov product of approximate Kasparov bimodules over $(A,B)$ and $(B,C([0,1],C))$, and over $(A,C([0,1],B))$ and $(C([0,1],B),C([0,1],C))$, where the approximate Kasparov bimodule over $(C([0,1],B),C([0,1],C))$ is given on the cylinder augmented Hilbert bimodule, not on the external tensor product with $C([0,1])$.

Now we prove a Kasparov product with a degenerate approximate Kasparov bimodule is homotopic to zero.
We do not require the cancellation property for freely complementable monoids, so it does not depend on Theorem \ref{Theorem: General additive inverses}.
Let $\varphi_1$ and $\varphi_2$ be approximate Kasparov bimodules over $(A,B)$ and $(B,C)$ respectively.
Suppose first $\varphi_1$ is degenerate.
Choose a trivial Kasparov bimodule $\varphi_1'$ over $(A,B)$ with zero Kasparov operator such that the underlying Hilbert module of $\varphi_1\oplus\varphi_1'$ is unitarily equivalent to $H\otimes\widetilde B$ for $H$ a graded Hilbert space with the even and odd parts of the same infinite dimensions.
By considering escaping left actions to infinity (Lemma \ref{Lemma: Escaping to zero left action}), we can construct a degenerate approximate Kasparov bimodule $\varphi$ over $(A,C([0,1],B))$ connecting $\varphi_1\oplus\varphi_1'$ and a trivial Kasparov bimodule.
With an approximate Kasparov bimodule over $(C([0,1],B),C([0,1],C))$ defined on the cylinder augmented Hilbert bimodule from $\varphi_2$, consider a Kasparov product with $\varphi$.
Then, as we have seen in the above, it connects a Kasparov product of $\varphi_1$ and $\varphi_2$ and a trivial Kasparov bimodule after adding a trivial Kasparov bimodule.
We can prove the case that $\varphi_2$ is degenerate in the same way.

Finally, we prove the biadditivity.
Since we have seen that Kasparov products are well-defined modulo homotopy moves, and since finite down set reduction exists up to operator homotopy (Lemma \ref{Lemma: Finite down set reduction}), we may assume the approximate Kasparov bimodules in here are taken by finite down set reductions.
First to check the additivity in the first variable, fix approximate Kasparov bimodules $\varphi_1:=(E_1,(F_{1j})_j)$ and $\varphi_1':=(E_1',(F_{1j'}')_{j'})$ over $(A,B)$, and $\varphi_2:=(E_2,(F_{2k})_{k\in\mathcal{K}})$ over $(B,C)$.
We may assume $\varphi_1$ and $\varphi_1'$ have the same indexing set and there is a cofinal system of separable C$^*$-subalgebras of $A$ indexed by this, say $\mathcal{J}$.
For each $j\in\mathcal{J}$, if we consider separable fiberings $(E_{1j},p_j,B_j)$ and $(E_{1j}',p_j',B_j)$ as in Corollary \ref{Corollary: Tuned form} by enlarging to let them have same $B_j$, then by Corollary \ref{Corollary: Tuned form} we may consider Kasparov products $\varphi_{12}:=(E_{12},(F_{12l})_{l\in\mathcal{L}})$ and $\varphi_{12}':=(E_{12}',(F_{12l}')_{l\in\mathcal{L}})$ of $\varphi_1$, $\varphi_2$, and $\varphi_1'$, $\varphi_2$ respectively, where $\mathcal{L}\subset\mathcal{J}\times\mathcal{K}$.
Then, $E_{12}\oplus E_{12}'\cong(E_1\oplus E_1')\otimes_BE_2$ and the family of separable fiberings $(E_{1j}\oplus E_{1j}',p_1\oplus p_1',B_j)$ with respect to $F_{1j}\oplus F_{1j}'$ under $A_j$ show that $\varphi_{12}\oplus\varphi_{12}'$ with indexing set $\mathcal{L}$ is a Kasparov product of $\varphi_1\oplus\varphi_1'$ with indexing set $\mathcal{J}$ and $\varphi_2$.
For the additivity in the second variable, it is more simple because the same separable fibering $(E_{1j},p_1,B_j)$ can be used to prove the sum is indeed a Kasparov product.
\end{proof}

\begin{theorem}[Associativity of Kasparov products]\label{Theorem: Associativity of Kasparov products}
There exists a category whose objects are C$^*$-algebras and morphisms are homotopy classes of approximate Kasparov bimodules.
\end{theorem}
\begin{proof}
We need to check the two axioms for categories, the associativity and the identity.
The latter is clear as the identity element of $KK^\mathrm{fc}(A,A)$ can be chosen as a Kasparov bimodule $(\widetilde A,0)$ over $(A,A)$, for $A$ a C$^*$-algebra.
Now we prove the associativity.

Let $A$, $B$, $C$, and $D$ be C$^*$-algebras.
Take approximate Kasparov bimodules $(E_1,(F_{1j})_j)$, $(E_2,(F_{2k})_k)$, and $(E_3,(F_{3l})_l)$ over $(A,B)$, $(B,C)$, and $(C,D)$, respectively.
We assume that they are the results of the finite down set reductions.
By considering the product $\mathcal{I}:=\mathcal{J}\times\mathcal{K}\times\mathcal{L}$ of the three directed sets $\mathcal{J},\mathcal{K},\mathcal{L}$ parametrized by $j,k,l$ respectively, we also suppose the indices $j,k,l\in\mathcal{I}$ are in a same directed set in which every element has finite down set.
Pairs are in a subset of the power as $(j,k),(k,l)\in\mathcal{I}^2$ and $(j,k,l)\in\mathcal{I}^3$.
We may also assume that all Kasparov operators on every approximate Kasparov bimodule in the proof are self-adjoint.

Recall that, in Lemma 6.8, when we choose a separable fibering, say $(E_0,p,B_0)$, first we take a coisometry $p\in B(H\otimes\widetilde B,E)$ arbitrarily, next we take $E_0$ such that $C^*(\Theta_{\xi,\eta}:\xi,\eta\in E_0)$ contains a given separable subset of compact operators and a separable C$^*$-subalgebra is reduced by $E_0$, and at last we take $B_0$ such that $p^*E_0\subset H\otimes\widetilde B_0$.
Following this order, we may take $E_0$ and $B_0$ sufficiently large enough.
For the proof of associativity, we will choose three families of separable fiberings together with monotone assignments $\mathcal{I}\to\mathcal{I}:j\mapsto k_j$ and $\mathcal{I}\mapsto\mathcal{I}:k\mapsto l_k$, with $k_j\ge j$ and $l_k\ge k$, as follows.
First, fix even coisometries $p_1\in B(H_1\otimes\widetilde B,E_1)$ and $p_2\in B(H_2\otimes\widetilde C,E_2)$ for graded Hilbert spaces, and let $p_{12}:=(p_1\otimes_B1)(1\otimes_\mathbb{C}p_2)\in B(H_1\otimes H_2\otimes\widetilde C,E_{12})$.
\begin{enumerate}
\item
For each $j$, with respect to $\{(F_{1j_0j}(t),A_{j_0}):j_0\le j,\ t\in[0,1]\}$, we take $(E_{1j},p_1,B_j)$, where the net $B_j$ is increasing, and take $k_j$ such that $\{F_{2k}\}_{k\ge k_j}$ is Kasparov under $B_j$ and $B_j\subset B_{k_j}$.
\item
For each $(j,k)$ with $k\ge k_j$, with respect to $(F_{12(j,k)},A_j)$, we take $(E_{12(j,k)},p_{12},C_{(j,k)})$, with the conditions (i), (ii), (iii) in the below, such that
\[E_{12(j,k)}\subset E_{1j}\otimes_{B_j}E_2,\]
\item
For each $k$, with respect to $\{(F_{2k_0k}(t),B_{k_0}):k_0\le k,\ t\in[0,1]\}$, we take $(E_{2k},p_2,C_k)$, where the net $C_k$ is increasing, and take $l_k$ such that $\{F_{3l}\}_{l\ge l_k}$ is Kasparov under $C_k$ and
\[E_{12(j,k)}\subset E_{1j}\otimes_{B_j}E_{2k},\qquad C_{(j,k)}\subset C_k,\]
for all $j$ such that $k\ge k_j$.
\end{enumerate}
The paths $F_{1j_0j}$ and $F_{2k_0k}$ are suitably chosen.
Be cautious that $(B_j)_{j\in\mathcal{I}}$ and $(B_k)_{k\in\mathcal{I}}$ denote the different directed systems.
The first separable fibering can be freely taken just by using Lemma \ref{Lemma: Existence of separable fiberings} and induction on the size of $j$.
Once the second separable fibering is successfully chosen, the third separable fibering is easy to choose with the two additional inclusion assumptions because for each $k$ there are only finitely many $j$ such that $k\ge k_j$.

The choice of the second separable fibering is problematic.
Before this, we need to consider a Kasparov product $(E_{12},(F_{12(j,k)})_{(j,k)})$ of $(E_1,(F_{1j})_j)$ and $(E_2,(F_{2k})_k)$, constructed with the prescribed first separable fibering $(E_{1j},p_1,B_j)$, with the indexing set $\{(j,k):k\ge k_j\}\subset\mathcal{I}^2$.
Furthermore, if we let $P_{1j}\in B(E_{12})$ be the projection onto the complemented Hilbert submodule $E_{1j}\otimes_{B_j}E_2=E_{1j}\otimes_{B_j}\widetilde B\otimes_BE_2=\overline{E_{1j}B^\sim}\otimes_BE_2\subset E_{12}$, then it commutes with $A_j$ and $\widehat F_{1j}$, and we may assume $F_{12(j,k)}=\widehat F_{1j}$ in the corner $B((1-P_{1j})E_{12})$ for all $(j,k)$.
This is by Corollary \ref{Corollary: Tuned form}.
Then, when we crop to get $E_{12(j,k)}$ following the proof of \ref{Lemma: Existence of separable fiberings}, since $\Sigma(F_{12(j,k)},A_j)\subset B(P_{1j}E_{12})$ and $C^*(F_{12(j,k)},A_j)$ commutes with $P_{1j}$, by cropping inside $P_{1j}E_{12}=(p_1\otimes_B1)(H_0\otimes E_2)$ for some graded Hilbert subspace $H_0\subset H$, we may choose $E_{12(j,k)}\subset P_{1j}E_{12}=E_{1j}\otimes_{B_j}E_2$.
Furthermore, we need to assume that
\begin{enumerate}
\item[(i)]
$C^*(\widehat F_{1j},\widehat\Theta_{\xi_1,\eta_1}:\xi_1,\eta_1\in E_{1j})E_{12(j,k)}\subset E_{12(j,k)}$,
\item[(ii)]
$(F_{12(j,k)}T_{\xi_1}-T_{\overline\xi_1}F_{2k})T_{\eta_1}^*\in C^*(\Theta_{\xi_{12},\eta_{12}}:\xi_{12},\eta_{12}\in E_{12(j,k)})$ for all $\xi_1,\eta_1\in E_{1j}$,
\item[(iii)]
$(-a^*[F_{12(j,k)},\widehat F_{1j}]a)^+\in C^*(\Theta_{\xi_{12},\eta_{12}}:\xi_{12},\eta_{12}\in E_{12(j,k)})$ for all $a\in A_j$,
\end{enumerate}
where the hat notation denotes $-\otimes1$.
The justifications of the above three assumptions are as follows.
First, (i) follows by the commutation $[\widehat F_{1j},P_{1j}]=0$ and the inclusion $\widehat\Theta_{\xi_1,\eta_1}\in B(P_{1j}E_{12})$ for $\xi_1,\eta_1\in E_{1j}$, which makes the cropping inside $P_{1j}E_{12}$ possible.
Next, (ii) is by that $P_{1j}$ acts as a unit on $(F_{12(j,k)}T_{\xi_1}-T_{\overline\xi_1}F_{2k})T_{\eta_1}^*\in K(E_{12})$.
For the last (iii), this is because in the corner of $P_{1j}$ we have $a^*[F_{12(j,k)},\widehat F_{1j}]a\ge0$ modulo $K(E_{12})$, but in the corner of $1-P_{1j}$ the negative part literally vanishes as $a^*[F_{12(j,k)},\widehat F_{1j}]a=2a^*\widehat F_{1j}^2a\ge0$ without taking modulo $K(E_{12})$.
By this, we complete to choose the second separable fibering $(E_{12(j,k)},p_{12},C_{(j,k)})$, satisfying $E_{12(j,k)}\subset E_{1j}\otimes_{B_j}E_2$.

Now within the third and the first separable fiberings, we can consider a Kasparov product $(E_{23},(F_{23(k,l)})_{(k,l)})$ of $(E_2,(F_{2k})_k)$ and $(E_3,(F_{3l})_l)$ indexed by $\{(k,l):l\ge l_k\}$ such that $\{F_{23(k,l)}\}_{(k,l)\ge(k_j,l_{k_j})}$ is Kasparov under $B_j$ for all $j$, and a Kasparov product $(E_{123},(F_{1(23)(j,k,l)})_{(j,k,l)})$ of $(E_1,(F_{1j})_j)$ and $(E_{23},(F_{23(k,l)})_{(k,l)})$ indexed by $\{(j,k,l):k\ge k_j,\ l\ge l_k\}$ such that $\{F_{1(23)(j,k,l)}\}_{(j,k,l)\ge(j,k_j,l_{k_j})}$ is Kasparov under $A_j$.
We prove that $(E_{123},(F_{1(23)(j,k,l)})_{(j,k,l)})$ is operator homotopic to a Kasparov product of $(E_{12},(F_{12(j,k)})_{(j,k)})$ and $(E_3,(F_{3l})_l)$.

Fix $(j,k,l)$ in the directed set $\{(j,k,l):k\ge k_j,\ l\ge l_k\}$, and write $F_1:=F_{1j}$, $F_2:=F_{2k}$, $F_3:=F_{3l}$, $F_{12}:=F_{12(j,k)}$, $F_{23}:=F_{23(k,l)}$, and $F_{1(23)}:=F_{1(23)(j,k,l)}$.
Define separable C$^*$-subalgebras of $B(E_{123})$ by
\[J_1:=C^*(\widehat{\widehat\Theta}_{\xi_1,\eta_1}:\xi_1,\eta_1\in E_{1j})\subset K(E_1)\otimes1\otimes1,\qquad
J_2:=C^*(\widehat\Theta_{\xi_{12},\eta_{12}}:\xi_{12},\eta_{12}\in E_{12(j,k)})\subset K(E_{12})\otimes1,\]
and
\[J_{13}:=C^*([\widehat{\widehat F}_1,F_{1(23)}]),\qquad J_{23}:=C^*(a^*(-[\widehat F_{12},F_{1(23)}])^+a:a\in A_j),\qquad J_3:=C^*(J_{13},J_{23}).\]
For a dense sequence $a_n\in A_j\setminus\{0\}$ of $A_j$, let
\[\Delta_1:=\{(n\|a_n\|)^{-1}a_n\}_n\cup\{0\}\cup\{\widehat{\widehat F}_1\}\subset B(E_1)\otimes1\otimes1,\qquad
\Delta_2:=\{\widehat F_{12}\}\subset B(E_{12})\otimes1,\qquad
\Delta_3:=\{F_{1(23)}\}.\]
Then, using the connection conditions, we can compute
\[[\Delta_1,J_1]\subset J_1,\qquad[\Delta_2,J_1]\subset J_2,\qquad[\Delta_3,J_1]\subset K(E_{123}),\qquad[\Delta_1\cup\Delta_2,J_2]\subset J_2,\qquad[\Delta_3,J_2]\subset K(E_{123}),\]
where the conditions (i) and (ii) are used to show $[\Delta_1,J_2]\subset J_2$ and $[\Delta_2,J_1]\subset J_2$ respectively, as well as
\[J_1J_{13}\subset K(E_{123}),\qquad J_2J_{13}\subset K(E_{123}),\qquad J_2J_{23}\subset K(E_{123}),\qquad J_1J_{23}\subset K(E_{123}),\]
where the condition (i) is used again to show $J_2J_{13}\subset K(E_{123})$.
We only show the last inclusion $J_1J_{23}\subset K(E_{123})$ in detail.
Observing $F_{1(23)}$ and $F_{23}$ are an $F_{23}$-connection and an $F_3$-connection, we can check $F_{1(23)}$ commutes modulo $K(E_{123})$ with $(\widehat F_{12}T_{\xi_1}-T_{\overline\xi_1}\widehat F_2)T_{\eta_1}^*$ for all $\xi_1,\eta_1\in E_{1j}$ by (ii) and $E_{12(j,k)}\subset E_{1j}\otimes_{B_j}E_{2k}$.
Then,
\[T_{\xi_1}(T_{\xi_1}^*[\widehat F_{12},F_{1(23)}]-[\widehat F_2,F_{23}]T_{\xi_1}^*)\in K(E_{123}),\qquad a\in A_j,\ \xi_1\in E_{1j}.\]
If we multiply $(|\xi_1|^2+\varepsilon)^{-1}T_{\xi_1}^*\in B(E_{123},E_{23})$ on the left and let $\varepsilon\downarrow0$, then since $[\widehat F_2,F_{23}]$ commutes with $T_{\xi_1}^*T_{\xi_1}=|\xi_1|^2\in\widetilde B_j$ modulo $K(E_{23})$, the estimate $\|(1-|\xi_1|^2(|\xi_1|^2+\varepsilon)^{-1})T_{\xi_1}^*\|^2=\|\varepsilon^2|\xi_1|^2(|\xi_1|^2+\varepsilon)^{-2}\|\le\varepsilon/4$ implies that, for a trivial monomial $f$ of degree one and hence for any complex coefficient polynomial $f$ we get
\[T_{\xi_1}^*f([\widehat F_{12},F_{1(23)}])-f([\widehat F_2,F_{23}])T_{\xi_1}^*\in K(E_{123},E_{23}).\]
Modulo $K(E_{23})$, using the fact that $b\in B_j$ and $[\widehat F_2,F_{23}]$ commute, we can show $(-[\widehat F_2,F_{23}])^+b\equiv 0$ for all $b\in B_j$ by pure state evaluations for a commutative C$^*$-subalgebra of $Q(E_{23})$ generated by each single normal $b\in B_j$ and $[\widehat F_2,F_{23}]$.
Thus it follows by a polynomial approximation of $t\mapsto(-t)^+$ and by $a^*a\eta_1\in A_jE_{1j}\subset\overline{E_{1j}B_j}$ that
\[T_{\xi_1}T_{\eta_1}^*a^*(-[\widehat F_{12},F_{1(23)}])^+a\equiv T_{\xi_1}(-[\widehat F_2,F_{23}])^+T_{a^*a\eta_1}^*\equiv0\]
modulo $K(E_{123})$, so $J_1J_{23}\subset K(E_{123})$.

Since $E_{12(j,k)}\subset E_{1j}\otimes_BE_2$, we can see $C^*(J_1,J_2)E_{12}=J_1E_{12}$ is graded and complemented in $E_{12}$.
Therefore we can find a mediator $M$ for the triple $(\Delta_1\cup\Delta_2\cup\Delta_3,C^*(J_1,J_2),J_3)$.
Define
\[F_{(12)3(j,k,l)}:=M^{\frac14}\widehat{\widehat F}_1M^{\frac14}+(1-M)^{\frac14}F_{1(23)}(1-M)^{\frac14}.\]
Since $M$ mediates $(\Delta_1\cup\Delta_3,J_1,J_{13})$, it is a Kasparov operator, and since $M$ mediates $(\Delta_2,J_2,J_{23})$, it defines a Kasparov product because $[\Delta_2,M]$ and $(1-M)J_{23}\subset K$ together with (iii) imply the $F_{12}$-positivity, and $MJ_2\subset K$ implies the $F_3$-connection condition.
It is easy to see it is a Connes-Skandalis operator homotopic (Theorem \ref{Theorem: Connes-Skandalis operator homotopy}) to $F_{1(23)}$.
\end{proof}

\begin{definition}
The category defined in the previous theorem is called the \emph{freely complementable Kasparov category}, and is denoted by $\mathrm{KK}^\mathrm{fc}$.
\end{definition}

\begin{corollary}[Functoriality of freely complementable KK-theory]
There is a canonical functor $\mathrm{C^*Alg}\to\mathrm{KK}^\mathrm{fc}$.
\end{corollary}
\begin{proof}
For given a $*$-homomorphism $A\to B$, we associate an element of $KK^\mathrm{fc}(A,B)$ by the Kasparov bimodule $(\widetilde B,0)$ over $(A,B)$, where the $*$-homomorphism is used to define the left action.
The functoriality conditions are clear.
\end{proof}

\section{Properties of freely complementable KK-theory}

\subsection{Additive structure}
We prove the freely complementable KK-theory is an additive category, and the canonical functor from the category of all C$^*$-algebras is additive in the sense that it preserves direct sums.
We first generalize Corollary \ref{Corollary: Additive inverses for para-unitals} from para-unital to general C$^*$-algebras.

\begin{theorem}[General additive inverses]\label{Theorem: General additive inverses}
Let $A$ and $B$ be C$^*$-algebras.
Then, the commutative monoid $KK^\mathrm{fc}(A,B)$ is always an abelian group.
\end{theorem}
\begin{proof}
Let $\mathrm{id}_B:=(\widetilde B\boxplus0,0)$ and $\overline{\mathrm{id}}_B:=(0\boxplus\widetilde B,0)$ be Kasparov bimodules over $(B,B)$, where the left action is given by the left multiplication of $B$.
We prove their direct sum $(\widetilde B\boxplus\widetilde B,0)$ is null-homotopic over $(B,B)$.
Let $p:[0,1)\to B(H)$ be a ($\sigma$-)strongly continuous path of rank-one projections on a separable infinite-dimensional Hilbert space such that $p(t)\to0$ as $t\to1$, obtained by translation in $H\cong L^2(\mathbb{R})$ for example.
Let $B_j\subset B$ be a cofinal system of separable C$^*$-subalgebras, and let $e_j\in C_b([0,1),B_j)$ be a continuous approximate unit of $B_j$ for each $j$ in the sense that $e_j(0)=0$ and $e_j(t)\uparrow1$ strictly in $M(B_j)$ as $t\to1$.
Define an approximate Kasparov bimodule $(E,(F_j)_j)$ over $(B,C([0,1],B))$ by
\[E:=(H\otimes C([0,1],B)^\sim)^{\boxplus2},\qquad\varphi(b)(t):=(p(t)\otimes b)\oplus(p(t)\otimes b),\qquad F_j(t):=\begin{bmatrix}0&p(t)\otimes e_j(t)\\p(t)\otimes e_j(t)&0\end{bmatrix}\]
with $\varphi(b)(1)=F_j(1)=0$ for $b\in B$ and $t\in[0,1)$, where $\varphi:B\to B(E)$ denotes the left action.
Then, it connects a direct sum of $(\widetilde B\boxplus\widetilde B,0)$ with a degenerate Kasparov bimodule at $t=0$ and a zero left action Kasparov bimodule at $t=1$.
Therefore, we have $[\mathrm{id}_B]+[\overline{\mathrm{id}}_B]=0$ in $KK^\mathrm{fc}(B,B)$.
By the Kasparov product, every $[\varphi]\in KK^\mathrm{fc}(A,B)$ has an additive inverse $[\overline{\mathrm{id}}_B]\circ[\varphi]$.
\end{proof}

\begin{proposition}[Additive structure]
The category $\mathrm{KK}^\mathrm{fc}$ is additive, and the functor $\mathrm{C^*Alg}\to\mathrm{KK}^\mathrm{fc}$ preserves direct sums.
\end{proposition}
\begin{proof}
The zero object is given by the zero C$^*$-algebra, so to prove $\mathrm{KK}^\mathrm{fc}$ is semi-additive, we show for C$^*$-algebras $A$, $B$, and $C$ we have natural group homomorphisms
\[KK^\mathrm{fc}(A\oplus B,C)\to KK^\mathrm{fc}(A,C)\oplus KK^\mathrm{fc}(B,C),\qquad KK^\mathrm{fc}(A,B\oplus C)\to KK^\mathrm{fc}(A,B)\oplus KK^\mathrm{fc}(A,C)\]
determined by functoriality are isomorphisms.
Because each isomorphism implies $A\oplus B$ is a coproduct of $A$ and $B$ and $B\oplus C$ is a product of $B$ and $C$ in $\mathrm{KK}^\mathrm{fc}$ respectively, it follows that the finite direct sum of C$^*$-algebras defines a finitary biproduct in $\mathrm{KK}^\mathrm{fc}$.
The statement that the functor preserves direct sums is automatic.
Surjectivity of both homomorphisms are proved just by considering direct sum of approximate Kasparov bimodules.

(Injectivity of first homomorphism)
Let $(E_0,(F_{0i})_i)$ be an approximate Kasparov bimodule over $(A\oplus B,C)$ with the left action $\varphi_0:A\oplus B\to B(E_0)$.
Define an approximate Kasparov bimodule $(E,(F_i)_i)$ over $(A\oplus B,C([0,1],C))$ by
\[E:=\mathrm{cyl}(E_0\oplus E_0),\qquad\varphi(a,b)(t):=\begin{pmatrix}\varphi_0(a,0)+c_t^2\varphi_0(0,b)&c_ts_t\varphi_0(0,b)\\s_tc_t\varphi_0(0,b)&s_t^2\varphi_0(0,b)\end{pmatrix},\qquad F_i(t):=F_{0i}\oplus F_{0i},\]
for $(a,b)\in A\oplus B$ and $t\in[0,1]$, where $c_t:=\cos\frac{\pi t}2$ and $s_t:=\sin\frac{\pi t}2$.
Then, at $t=0$ it is a direct sum of $(E,(F_i)_i)$ with a degenerate approximate Kasparov bimodule, and at $t=1$ it is the direct sum of the restrictions of left actions to $A$ and $B$ respectively, so the injectivity follows.

(Injectivity of second homomorphism)
Even though we can prove directly a canonical approximate Kasparov bimodule with the underlying Hilbert module
\[(E\otimes_{B\oplus C}\widetilde B\otimes_B(B\oplus C)^\sim)\oplus(E\otimes_{B\oplus C}\widetilde C\otimes_C(B\oplus C)^\sim)\cong E\oplus((E/E(B\oplus C))\otimes(B\oplus C)^\sim)\]
is homotopic to a given approximate Kasparov bimodule $(E,(F_i)_i)$ over $(A,B\oplus C)$, because it needs a complicated argument using the free complementability of $E$ to establish the above unitary equivalence, we provide a simple proof by applying the Kasparov product as follows.
Let $\varphi_B:(b,c)\mapsto(b,0)$ and $\varphi_C:(b,c)\mapsto(0,c)$ be $*$-endomorphisms on $B\oplus C$.
If we prove $[\varphi_B]+[\varphi_C]=[\mathrm{id}_{B\oplus C}]$, then the proof is complete.
Let $D:=B\oplus C$ and define a Kasparov bimodule $((C([0,1],D)^\sim)^{\oplus2},0)$ over $(D,C([0,1],D))$ such that the left action is
\[\varphi(d)(t):=\begin{pmatrix}\varphi_B(d)+c_t^2\varphi_C(d)&c_ts_t\varphi_C(d)\\s_tc_t\varphi_C(d)&s_t^2\varphi_C(d)\end{pmatrix}\]
for $d\in D$ and $t\in[0,1]$, where $c_t:=\cos\frac{\pi t}2$ and $s_t:=\sin\frac{\pi t}2$.
Then, at $t=0$ it is a direct sum of $\mathrm{id}_D$ and a degenerate Kasparov bimodule, and at $t=1$ it represents $[\varphi_B]+[\varphi_C]$.
\end{proof}

The rotation trick in the above can be found for example in Lemma 4.1.9 of \cite{MR1124848}.
Instead of directly giving definition of direct sum for approximate Kasparov bimodules, the abelian group structure $KK^\mathrm{fc}(A,B)\oplus KK^\mathrm{fc}(A,B)\to KK^\mathrm{fc}(A,B)$ on $KK^\mathrm{fc}(A,B)$ also can be described by either a $*$-homomorphism $A\to A\oplus A:a\mapsto(a,a)$ or a Kasparov bimodule $(\widetilde B\oplus\widetilde B,0)$ over $(B\oplus B,B)$.
A canonical morphism $B\oplus B\to B$ does not exist in $\mathrm{C^*Alg}$, but exists in $\mathrm{KK}^\mathrm{fc}$ so that the direct sum becomes a biproduct.

\subsection{Symmetric monoidal structures}

Recall that $\mathrm{Corr}^\mathrm{ndg}$ and its subcategory $\mathrm{Corr}_\mathrm{para}^\mathrm{fc,ndg}$ are symmetric monoidal categories.
The monoidal functor can be both minimal or maximal tensor product.
Let $\mathrm{KK}_\mathrm{para}^\mathrm{fc}$ be the full subcategory of $\mathrm{KK}^\mathrm{fc}$ given by the image of the restriction of the freely complementable KK-theory functor $\mathrm{C^*Alg}_\mathrm{para}\to\mathrm{KK}_\mathrm{para}^\mathrm{fc}$.
If we restrict our theory to para-unital C$^*$-algebras, which allows us to remove the fiber at infinity, then the well-defined external tensor products provide symmetric monoidal structures on $\mathrm{KK}_\mathrm{para}^\mathrm{fc}$, as well as the category of non-degenerate Hilbert bimodule.
The fastest way to check the existence of the symmetric monoidal structures on the usual KK-theory is to use its universal property, but since freely complementable KK-theory has no universal property, we need an analytic proof involving the technical theorem.
This argument is necessary even in the classical KK-theory.

\begin{lemma}[Whiskering]
Let $D$ be a para-unital C$^*$-algebra and fix $*\in\{\mathrm{min},\mathrm{max}\}$.
Then, the operation $(E,(F_i)_i)\mapsto(E,(F_i)_i)\otimes_*D:=(E\otimes_*D,(F_i\otimes1)_i)$ on non-degenerate approximate Kasparov bimodules defines an additive functor on $\mathrm{KK}_\mathrm{para}^\mathrm{fc}$ by preserving Kasparov products.
\end{lemma}
\begin{proof}
First, $(E\otimes_*D,(F_i\otimes1)_i)$ satisfies the condition for an approximate Kasparov bimodule because $\Sigma(F_i\otimes1,A_0\otimes_*D_0)\subset\Sigma(F_i,A_0)\otimes_*D_0$, where $A_0\subset A$ and $D_0\subset D$ are separable C$^*$-subalgebras, and $i$ is a suitable index.
We can also check the four conditions for Kasparov products are preserved by the external tensor product with the identity $*$-homomorphism on $D$, sketched as follows.
Suppose $\varphi_{12}=(E_{12},(F_{12l})_k)$ is a Kasparov product of $\varphi_1=(E_1,(F_{1j})_j)$ and $\varphi_2=(E_2,(F_{2k})_k)$, which are over $(A,B)$ and $(B,C)$ respectively.
Let $A_0\subset A$ and $D_0\subset D$ be separable C$^*$-subalgebras, and choose a separable fibering $(E_0,p_D,D_1)$ with respect to the zero Kasparov operator under $D_0$ on the identity Kasparov bimodule $D$ over $(D,D)$, for a fixed coisometry $p_D\in B(H\otimes\widetilde D,D)$.
If $(E_{1l},p_1,B_l)$ is a separable fibering used in the Kasparov product of $\varphi_1$ and $\varphi_2$, then the separable fibering $(\overline{E_{1l}\odot E_0},p_1\otimes p_D,\overline{B_l\odot D_1})$ satisfies the four conditions so that $\varphi_{12}\otimes_*D$ is a Kasparov product of $\varphi_1\otimes_*D$ and $\varphi_2\otimes_*D$.
\end{proof}

\begin{proposition}[Monoidal structures]
The category $\mathrm{KK}_\mathrm{para}^\mathrm{fc}$ is symmetric monoidal, and the functor $\mathrm{C^*Alg}_\mathrm{para}\to\mathrm{KK}_\mathrm{para}^\mathrm{fc}$ is strict symmetric monoidal, with respect to either the maximal or minimal tensor product.
\end{proposition}
\begin{proof}
We always assume the non-degenerate reduction.
The monoidal bifunctor on $\mathrm{KK}_\mathrm{para}^\mathrm{fc}$ given by the external tensor product is well-defined since we can show the external tensor product of homotopic non-degenerate approximate Kasparov bimodules by dividing the cases of homotopy moves and degeneracy moves.
For the interchange law, there is already a well-defined interchanging unitary equivalence between the underlying Hilbert bimodules, so the problem is about Kasparov operators of Kasparov products.
Let $(E_{12},(F_{12l})_l)$ be a Kasparov product of $(E_1,(F_{1j})_j)$ and $(E_2,(F_{2k})_k)$.
Fix $*\in\{\mathrm{min},\mathrm{max}\}$.
We prove the interchange of whiskering.

Let $A_1$, $A_2$, $B_1$, and $B_2$ be para-unital C$^*$-algebras.
Let $\mathcal{I}$ be the directed set of finite subsets of $(K(\ell^2(\mathbb{N}))\otimes\widetilde A_1)\times(K(\ell^2(\mathbb{N}))\otimes\widetilde A_2)$, and for each $i\in\mathcal{I}$ we can define cofinal systems of separable C$^*$-subalgebras $A_{1i}\subset A_1$ and $A_{2i}\subset A_2$, so that $A_i:=C^*(A_{1i}\odot A_{2i})\subset A_1\otimes_*A_2$ also defines a cofinal system of separable C$^*$-subalgebras.
Let $(E_1,(F_{1i})_i)$ and $(E_2,(F_{2i})_i)$ be non-degenerate approximate Kasparov bimodules over $(A_1,B_1)$ and $(A_2,B_2)$ respectively, whose Kasparov operators are indexed by a common directed set $\mathcal{I}$.
We claim that there is a net $F_{12i}\in B(E_1\otimes_*E_2)$ of odd adjointable operators implementing a Kasparov product of $(E_1,(F_{1i})_i)\otimes_*A_2$ and $B_1\otimes_*(E_2,(F_{2i})_i)$, and simultaneously a Kasparov product of $A_1\otimes_*(E_2,(F_{2i})_i)$ and $(E_1,(F_{1i})_i)\otimes_*B_2$.
This claim is nothing but a modification of Theorem 2.14 (8) of \cite{MR918241}, a commutation for the cup-product.

First, choose separable fiberings $(E_{2i}',p_{A_2},A_{2i}')$ with respect to zero under $A_{2i}$ on the identity Hilbert bimodule $A_2$ such that $E_{2i}'$ and $A_{2i}'$ are non-decreasing with $E_{2i}'\subset A_{2i}'$ by induction.
Take a subnet $F_{2i}':=F_{2i'}$ of $F_{2i}$ with a monotone cofinal $\mathcal{I}\to\mathcal{I}:i\mapsto i'$ such that $\{F_{2i}\}_{i\ge i_0'}$ is Kasparov under $A_{2i_0}'$ for all $i_0\in\mathcal{I}$, as well as $F_{1i}'$.
Similarly choose $(E_{1i}',p_{A_1},A_{1i}')$, but before taking $F_{2i}'$, and then take $F_{1i}'$ with the same map $i\mapsto i'$ as $F_{2i}'$. 
Choose separable fiberings $(E_{1i},p_1,B_{1i})$ and $(E_{2i},p_2,B_{2i})$ with respect to all pairs $(F_{1i_0i}'(s),A_{1i_0}')$ and $(F_{2i_0i}'(s),A_{2i_0}')$ for all pairs $i_0\le i$ in $\mathcal{I}$ and all $s\in[0,1]$, respectively, where $F_{1i_0i}'$ and $F_{2i_0i}'$ are suitable operator homotopies that exist by the Kasparov family conditions for $\{F_{1i}'\}_{i\ge i_0}$ and $\{F_{2i}'\}_{i\ge i_0}$ under $A_{1i_0}'$ and $A_{2i_0}'$ respectively.
For dense sequences $a_{1n}\in A_{1i}\setminus\{0\}$ and $a_{2n}\in A_{2i}\setminus\{0\}$, define norm compact subsets
\begin{align*}
\Delta_1&:=\{(n\|a_{1n}\|)^{-1}a_{1n}\otimes1\}\cup\{0\}\cup\{F_{1i_0i}'(s)\otimes1:i_0\le i,\ s\in[0,1]\},\\
\Delta_2&:=\{1\otimes(n\|a_{2n}\|)^{-1}a_{2n}\}\cup\{0\}\cup\{1\otimes F_{2i_0i}'(s):i_0\le i,\ s\in[0,1]\},
\end{align*}
and separable C$^*$-subalgebras
\[J_1:=C^*(\Theta_{\xi_1,\eta_1}\otimes1:\xi_1,\eta_1\in E_{1i}),\qquad J_2:=C^*(1\otimes\Theta_{\xi_2,\eta_2}:\xi_2,\eta_2\in E_{2i})\]
in $B(E_1\otimes_*E_2)$.
Then, $[\Delta_1,J_1]\subset J_1$, $[\Delta_2,J_1]=0$, and $J_1J_2\subset K(E_1\otimes_*E_2)$.
Since $J_1(E_1\otimes_*A_2)=\overline{E_{1i}B_1^\sim}\otimes_*A_2$ is complemented in $E_1\otimes_*A_2$, by applying the technical theorem (Theorem \ref{Theorem: Technical theorem}) on the internal tensor product $E_1\otimes_*E_2=(E_1\otimes_*A_2)\otimes_{B_1\otimes_*A_2}(B_1\otimes_*E_2)$, we can construct an even positive contraction $M_i\in B(E_1\otimes_*E_2)$ satisfying $[M_i,\Delta_1\cup\Delta_2]$, $M_iJ_1$, and $(1-M_i)J_2$ are all belonging to $K(E_1\otimes_*E_2)$.
Define
\[F_{12i}:=M_i^{\frac12}(F_{1i}'\otimes1)+(1-M_i)^{\frac12}(1\otimes F_{2i}').\]
Then, $\{F_{12i}\}_{i\ge i_0}$ is Kasparov under $A_{i_0}$ since we can connect from $F_{12i_0}$ and $F_{12i}$ for $i\ge i_0$ by interpolating first $M_{i_0}$ to $M_i$ linearly and then moving $F_{1i_0}'$ and $F_{2i_0}'$ along the operator homotopies $F_{1i_0i}'$ and $F_{2i_0i}'$ taken beforehand.

Now we verify that $(E_1\otimes_*E_2,(F_{12i})_i)$ is a Kasparov product of $(E_1,(F_{1i})_i)\otimes_*A_2$ and $B_1\otimes_*(E_2,(F_{2i})_i)$.
Independently of $i_0$ in the definition of Kasparov product, with $F_{1i}'\otimes1$, let $E_{1i}:=\overline{E_{1i}\odot E_{2i}'}\subset E_1\otimes_*A_2$, $p_1:=p_1\otimes p_{A_2}$, and $B_i:=C^*(B_{1i}\odot A_{2i}')$, where $p_1$ is abusing notations.
Then, it is routine to check $(E_{1i},p_1,B_i)$ is a separable fibering with respect to $(F_{1i}'\otimes1,A_i)$ such that the four conditions of Kasparov product are satisfied.
By symmetry, $(E_1\otimes_*E_2,(F_{12i})_i)$ is also a Kasparov product of $A_1\otimes_*(E_2,(F_{2i})_i)$ and $(E_1,(F_{1i})_i)\otimes_*B_2$, hence the claim follows.

Let $\varphi_1$, $\psi_1$, $\varphi_2$, and $\psi_2$ be approximate Kasparov bimodules over para-unital C$^*$-algebras $(A_1,B_1)$, $(B_1,C_1)$, $(A_2,B_2)$, and $(B_2,C_2)$ respectively.
By the claim we can write and define $([\varphi_1]\otimes_*B_2)\circ(A_1\otimes_*[\varphi_2])=(B_1\otimes_*[\varphi_2])\circ([\varphi_1]\otimes_*A_2)=:[\varphi_1]\otimes_*[\varphi_2]$.
Then, together with the functoriality of whiskering the interchange law follows formally by
\begin{align*}
([\psi_1]\otimes_*[\psi_2])\circ([\varphi_1]\otimes_*[\varphi_2])
&=(C_1\otimes_*[\psi_2])\circ([\psi_1]\otimes_*B_2)\circ(B_1\otimes_*[\varphi_2])\circ([\varphi_1]\otimes_*A_2)\\
&=(C_1\otimes_*[\psi_2])\circ(C_1\otimes_*[\varphi_2])\circ([\psi_1]\otimes_*A_2)\circ([\varphi_1]\otimes_*A_2)\\
&=(C_1\otimes_*([\psi_2]\circ[\varphi_2]))\circ(([\psi_1]\circ[\varphi_1])\otimes_*A_2)\\
&=([\psi_1]\circ[\varphi_1])\otimes_*([\psi_2]\circ[\varphi_2]).
\end{align*}
This proves that the minimal and maximal tensor product establish bifunctors $\mathrm{KK}_\mathrm{para}^\mathrm{fc}\times\mathrm{KK}_\mathrm{para}^\mathrm{fc}\to\mathrm{KK}_\mathrm{para}^\mathrm{fc}$.
Now then with the monoidal unit $\mathbb{C}$, the associators, two unitors, and swapper are canonically defined as usual as in the category $\mathrm{C^*Alg}$, so the category $\mathrm{KK}_\mathrm{para}^\mathrm{fc}$ is symmetric monoidal.

To show the functor $\mathrm{C^*Alg}_\mathrm{para}\to\mathrm{KK}_\mathrm{para}^\mathrm{fc}$ is strict symmetric monoidal, since it is an identity on objects, we need to check the coherence condition for the naturality between two bifunctors $\mathrm{C^*Alg}_\mathrm{para}\times\mathrm{C^*Alg}_\mathrm{para}\to\mathrm{KK}_\mathrm{para}^\mathrm{fc}$ comparing the order of external tensor and the KK-theory functor.
Consider $*$-homomorphisms between C$^*$-algebras $\varphi_1:A_1\to B_1$ and $\varphi_2:A_2\to B_2$, and use the same notation for corresponding non-degenerate Kasparov bimodules $\varphi_1=(\varphi_1(A_1)B_1,0)$ and $\varphi_2=(\varphi_2(A_2)B_2,0)$.
Since the Kasparov bimodule $\varphi_1\otimes_*B_2$ is represented by the $*$-homomorphism $\varphi_1\otimes_*\mathrm{id}_{B_2}$, it is easy to see the $*$-homomorphism $\varphi_1\otimes_*\varphi_2$ is a Kasparov product of Kasparov bimodules $\varphi_1\otimes_*B_2$ and $A_1\otimes_*\varphi_2$ because ordinary compositions of $*$-homomorphisms give Kasparov products.
This implies the coherence condition $[\varphi_1\otimes_*\varphi_2]=[\varphi_1]\otimes_*[\varphi_2]$.
The associativity, unitality, and symmetry are clear.
\end{proof}

\begin{theorem}[Bott periodicity]
Let $A$ and $B$ be para-unital C$^*$-algebras.
There exists a canonical invertible element of $KK^\mathrm{fc}(A,C_0(\mathbb{R}^2,A))$, and there exists a canonical group isomorphism $KK^\mathrm{fc}(A,B)\cong KK^\mathrm{fc}(C_0(\mathbb{R},A),C_0(\mathbb{R},B))$.
\end{theorem}
\begin{proof}
It clearly follows by the monoidal structure on $\mathrm{KK}^\mathrm{fc}_\mathrm{para}$ and the Bott periodicity of the classical KK-theory.
\end{proof}

\subsection{Homotopy invariance, stability, split exactness}
For para-unital C$^*$-algebras, the homotopy invariance and complete stability immediately follow from a formal reasoning for monoidal functors.
However, in fact, it is possible to prove them for general C$^*$-algebras.
Here by the complete stability we mean the freely complementable KK-theory functor inverts the corner embedding $A\to K(H)\otimes A$ for all C$^*$-algebras $A$ and all generally uncountable dimensional Hilbert spaces $H$.

\begin{proposition}[Homotopy invariance]
For a C$^*$-algebra $A$, the constant embedding $A\to C([0,1],A)$ induces an isomorphism in $\mathrm{KK}^\mathrm{fc}$.
\end{proposition}
\begin{proof}
Since the $*$-homomorphism $C([0,1],A)\to C([0,1],A):a\mapsto a(0)$ is homotopic to the identity in the point-norm topology, they are also homotopic as Kasparov bimodules over $(C([0,1],A),C([0,1],A))$.
\end{proof}

\begin{proposition}[Complete stability]
For a C$^*$-algebra $A$ and a Hilbert space $H$, the corner embedding $A\to K(H)\otimes A$ induces an isomorphism in $\mathrm{KK}^\mathrm{fc}$.
\end{proposition}
\begin{proof}
Consider a Kasparov bimodule $(H\otimes\widetilde A,0)$ over $(K(H)\otimes A,A)$, where the left action is given such that $K(H)$ acts on $H$ and $A$ acts on $\widetilde A$.
For the corner embedding $A\to K(H)\otimes A$, let $p\in K(H)$ be the corresponding rank-one projection.
The pullback of $(H\otimes\widetilde A,0)$ along the corner embedding is $(H\otimes\widetilde A,0)\cong(\widetilde A,0)\oplus((1-p)H\otimes\widetilde A,0)$ over $(A,A)$, and here the latter factor $((1-p)H\otimes\widetilde A,0)$ has zero left action so that it is degenerate.

The difficult part is the pushforward by the corner embedding, $(H\otimes(K(H)\otimes A)^\sim,0)$ over $(K(H)\otimes A,K(H)\otimes A)$, where the left action is
\[\varphi(k\otimes a)=k\otimes p\otimes a\in K(H)\otimes K(H)\otimes A\subset K(H\otimes(K(H)\otimes A)^\sim),\qquad k\otimes a\in K(H)\otimes A.\]
Letting $w$ be the flipping self-adjoint unitary on $H\otimes H$, define $u_t:=e^{\pi it(1-w)/2}\in B(H\otimes H)$ a continuous path of unitaries parametrized by $t\in[0,1]$, which connects $u_0=1$ and $u_1=w$.
Then, by rotating with the adjoint of $u_t\otimes1$, the left action of the Kasparov product can be deformed to $\varphi'$ defined by
\[\varphi'(k\otimes a):=p\otimes k\otimes a,\qquad k\otimes a\in K(H)\otimes A.\]
Then, by cutting out the degenerate part $(1-p)H\otimes(K(H)\otimes A)^\sim$, the proof is complete.
\end{proof}

In the following theorem, the symbols $\iota$, $\pi$, $\sigma$, and $\rho$ in the following are chosen to mean `inclusion', `projection', `section', and `retract'.

\begin{theorem}[Split exactness for para-unital ideals]
Let
\[\begin{tikzcd}
0\rar&B\rar{\iota}&D\rar[swap]{\pi}&A\rar\lar[swap,bend right]{\sigma}&0
\end{tikzcd}\]
be a split exact sequence of C$^*$-algebras with $B$ para-unital.
Then, $D$ is KK$^\mathrm{fc}$-equivalent to $A\oplus B$.
\end{theorem}
\begin{proof}
Let $\rho:=(E_\rho,F_\rho)$ be a Kasparov bimodule over $(D,B)$ defined by
\[E_\rho:=B\boxplus B,\qquad\varphi_\rho(d):=d\boxplus\sigma(\pi(d)),\qquad F_\rho:=\begin{bmatrix}0&1\\1&0\end{bmatrix},\]
for all $d\in D$, where $\varphi_\rho$ denotes the left action on $E_\rho$.
It suffices to prove the three relations
\[[\rho]\circ[\iota]=[\mathrm{id}_B],\qquad[\rho]\circ[\sigma]=0,\qquad[\sigma]\circ[\pi]+[\iota]\circ[\rho]=[\mathrm{id}_D]\]
because
\[([\pi]+[\rho])\circ([\sigma]+[\iota])=[\pi]\circ[\sigma]+[\pi]\circ[\iota]+[\rho]\circ[\sigma]+[\rho]\circ[\iota]=[\mathrm{id}_A]+0+0+[\mathrm{id}_B]=[\mathrm{id}_{A\oplus B}]\]
and
\[([\sigma]+[\iota])\circ([\pi]+[\rho])=[\sigma]\circ[\pi]+[\iota]\circ[\rho]=[\mathrm{id}_D]\]
deduce that $[\pi]+[\rho]\in KK^\mathrm{fc}(D,A\oplus B)$ and $[\sigma]+[\iota]\in KK^\mathrm{fc}(A\oplus B,D)$ are mutually inverses of each other.

The first two relations can be shown as usual, so we only prove the third relation.
Let $(E_{\sigma\pi},F_{\sigma\pi})$ and $(E_{\iota\rho},F_{\iota\rho})$ be Kasparov bimodules over $(D,D)$ defined by
\[E_{\sigma\pi}:=\widetilde D\boxplus0,\qquad\varphi_{\sigma\pi}(d):=\sigma(\pi(d))\boxplus0,\qquad F_{\sigma\pi}:=0,\]
and
\[E_{\iota\rho}:=B\boxplus B,\qquad\varphi_{\iota\rho}(d):=d\boxplus\sigma(\pi(d)),\qquad F_{\iota\rho}:=\begin{bmatrix}0&1\\1&0\end{bmatrix},\]
for all $d\in D$, where $\varphi_{\sigma\pi}$ and $\varphi_{\iota\rho}$ denote the left actions on $E_{\sigma\pi}$ and $E_{\iota\rho}$ respectively.
Then, we have $[(E_{\sigma\pi},F_{\sigma\pi})]=[\sigma]\circ[\pi]$ and $[(E_{\iota\rho},F_{\iota\rho})]=[\iota]\circ[\rho]$ by construction.
By rotating the even part of the left action on the direct sum, we can define a Kasparov bimodule $(E,F)$ over $(D,C([0,1],D))$ by $E:=\mathrm{cyl}((\widetilde D\oplus B)\boxplus B)$,
\[\varphi(d)(t):=\begin{bmatrix}\begin{pmatrix}c_t^2\sigma(\pi(d))+s_t^2d&c_ts_t(\sigma(\pi(d))-d)\\s_tc_t(\sigma(\pi(d))-d)&s_t^2\sigma(\pi(d))+c_t^2d\end{pmatrix}&0\\0&\sigma(\pi(d))\end{bmatrix},\qquad F(t):=\begin{bmatrix}0&\begin{pmatrix}0\\1\end{pmatrix}\\\begin{pmatrix}0&1\end{pmatrix}&0\end{bmatrix},\]
for $d\in D$ and $t\in[0,1]$, where $c_t:=\cos\frac{\pi t}2$ and $s_t:=\sin\frac{\pi t}2$, and $\varphi$ denotes the left action on $E$.
Then, removing the degenerate part at $t=1$, we can verify $[\sigma]\circ[\pi]+[\iota]\circ[\rho]=[\mathrm{id}_D]$.
\end{proof}

\subsection{Homotopy reduction revisited}

\begin{proposition}[Operator homotopy reduction for para-unitals]\label{Proposition: Operator homotopy reduction}
Let $A$ and $B$ be C$^*$-algebras with $B$ para-unital.
Let $\varphi_0$ and $\varphi_1$ be approximate Kasparov bimodules over $(A,B)$ connected by a homotopy.
If $\varphi_0$ and $\varphi_1$ have zero fiber at infinity, then there is a degenerate Kasparov bimodule $\varphi'$ with zero fiber at infinity over $(A,B)$ such that $\varphi_0\oplus\varphi'$ and $\varphi_1\oplus\varphi'$ are operator homotopic, after unitary identification.
\end{proposition}

\begin{proof}
Let $\mathbb{T}:=\mathbb{R}/2\pi\mathbb{Z}$ with the canonical embedding $C(\mathbb{T})\subset C([0,2\pi])$.
Let $H,P\in B(L^2([0,2\pi]))$ be the periodic Hilbert transform and the Hardy projection, precisely given by
\[H\varepsilon_n:=\left\{\begin{alignedat}{2}&-i\mathrm{sgn}(n)\varepsilon_n&\quad&\text{ if }n\ne0,\\&0&&\text{ if }n=0,\end{alignedat}\right.\qquad P\varepsilon_n:=\left\{\begin{alignedat}{2}&\varepsilon_n&\quad&\text{ if }n\ge0,\\&0&&\text{ if }n<0,\end{alignedat}\right.\]
where $\{\varepsilon_n:=(t\mapsto e^{int}):n\in\mathbb{Z}\}$.
Then, we can write $H\equiv-iP+i(1-P)$ modulo $K(L^2([0,2\pi]))$, and $H$ is an essentially unitary that essentially commutes with $C(\mathbb{T})$ since $[P,a]\equiv0$ for all $a\in C(\mathbb{T})$.
Let $C\in C([0,2\pi])$ be any self adjoint contraction such that $C(0)=1$ and $C(2\pi)=-1$, and let $S:=(1-C^2)^{1/2}\in C(\mathbb{T})$.
If we let $U:=C+iS\in C([0,2\pi])$, and that $U^2\in C(\mathbb{T})$ is a unitary of winding number one.
Since we have modulo compacts
\[C+S^{1/2}HS^{1/2}\equiv C+HS\equiv U^*P+U(1-P)\equiv U(PU^{*2}P+(1-P)),\]
which is essentially unitary that essentially commutes with $C(\mathbb{T})$, and since $PU^{*2}P$ is exactly the Toeplitz operator of $(U^2)^*$, it is always of index one.
Consider a norm continuous path $s\mapsto F_s:=C_s+S_s^{1/2}HS_s^{1/2}\in B(L^2([0,2\pi]))$ of such essentially unitary operators first defined for $s\in[\frac13,\frac23]$ such that $C_{1/3}(t)=-1$ for $t\in[\frac{2\pi}3,2\pi]$ and $C_{2/3}(t)=1$ for $[0,\frac{4\pi}3]$.
Then, $F_{1/3}$ and $F_{2/3}$ are reduced by $L^2([0,\frac{2\pi}3])$ and $L^2([\frac{4\pi}3,2\pi])$ respectively.
Since the space of essentially unitary operators of a fixed index is path-connected in norm, we can prolong the path $s\mapsto F_s$ to have the domain $s\in[0,1]$ such that $F_0,F_1$ are coisometries of one-dimensional kernel, and that $F_s$ is an essentially unitary operator reduced by $L^2([0,\frac{2\pi}3])$ with $F_s=-1$ on $L^2([\frac{2\pi}3,2\pi])$ for $s\in[0,\frac13]$ and by $L^2([\frac{4\pi}3,2\pi])$ with $F_s=1$ on $L^2([0,\frac{4\pi}3])$ for $s\in[\frac23,1]$.
Then, we can define a Kasparov bimodule $(E,F)$ over $(C([0,1]),C([0,1]))$ by
\begin{gather*}
E=C([0,1],L^2([0,2\pi])^{\boxplus2}),\qquad
\varphi(a)\xi(t):=\begin{cases}
a(0)\xi(t)&\text{ if }\xi\in L^2([0,\frac{2\pi}3])^{\boxplus2},\\
a(\frac{3t}{2\pi}-1)\xi(t)&\text{ if }\xi\in L^2([\frac{2\pi}3,\frac{4\pi}3])^{\boxplus2},\\
a(1)\xi(t)&\text{ if }\xi\in L^2([\frac{4\pi}3,2\pi])^{\boxplus2},
\end{cases}\qquad
F(s):=\begin{bmatrix}0&F_s^*\\F_s&0\end{bmatrix}
\end{gather*}
for $s\in[0,1]$ and $t\in[0,2\pi]$, where $\varphi:C([0,1])\to L^\infty([0,2\pi])\subset B(L^2([0,2\pi])^{\boxplus2})$ denotes the left action on $E$ constant along $s$.
On $L^2([0,\frac{2\pi}3])^{\boxplus2}\oplus L^2([\frac{2\pi}3,2\pi])^{\boxplus2}$ we have
\[\varphi(a)=\delta_0(a)\oplus\varphi(a)|_{L^2([\frac{2\pi}3,2\pi])^{\boxplus2}},\qquad F(0)=\left.\begin{bmatrix}0&F_0^*\\F_0&0\end{bmatrix}\right|_{L^2([0,\frac{2\pi}3])^{\boxplus2}}\oplus\begin{bmatrix}0&-1\\-1&0\end{bmatrix},\]
and on $L^2([0,\frac{4\pi}3])^{\boxplus2}\oplus L^2([\frac{4\pi}3,2\pi])^{\boxplus2}$ we have
\[\varphi(a)=\varphi(a)|_{L^2([0,\frac{4\pi}3])^{\boxplus2}}\oplus\delta_1(a),\qquad F(1)=\begin{bmatrix}0&1\\1&0\end{bmatrix}\oplus\left.\begin{bmatrix}0&F_1^*\\F_1&0\end{bmatrix}\right|_{L^2([\frac{4\pi}3,2\pi])^{\boxplus2}},\]
where $\delta_0,\delta_1:C([0,1])\to\mathbb{C}$ are the point masses at zero at one.
Then, let $H_0:=\ker F_0$ and $H_1:=\ker F_1$ with the orthogonal complement $H_m$ of $H_0\oplus H_1$ in $L^2([0,2\pi])^{\boxplus2}$.

The above construction was originally from Section 6 of \cite{MR582160}, and one may see 18.5.1 of \cite{MR1656031} and 2.2.17 \cite{MR1124848}.
Extending this by whiskering, we obtain a Kasparov bimodule $\omega=(E_\omega,F_\omega)$ over $(C([0,1],B),C([0,1],B))$ satisfying the following.
The underlying freely complementable augmented Hilbert bimodule is
\[E_\omega:=C([0,1],H_\omega\otimes B),\qquad H_\omega:=H_0\oplus H_m\oplus H_1\]
with $H_0\cong H_1\cong\mathbb{C}$ and $H_m\cong\ell^2(\mathbb{N})\boxplus\ell^2(\mathbb{N})$ as graded Hilbert spaces.
The left action $\omega:C([0,1],B)\to B(E_\omega)$ is constant along $C([0,1])$ in the right action, and has the form
\[\omega(fb):=(\delta_0\oplus\omega_m\oplus\delta_1)(f)\otimes b,\qquad f\in C([0,1]),\ b\in B,\]
where $\omega_m:C([0,1])\to B(H_m)$ is a unital representation.
For $F_\omega\in C([0,1],B(H_\omega))\subset B(E_\omega)$, it is a norm continuous path, and the endpoints $F_\omega(0),F_\omega(1)\in B(H_\omega)\subset B(E_\omega(0))$ are given by
\[F_\omega(0)=0\oplus N_{\omega0},\qquad F_\omega(1)=N_{\omega1}\oplus0,\]
where $N_{\omega0}\in B(H_m\oplus H_1)$ and $N_{\omega1}\in B(H_0\oplus H_m)$ are odd self-adjoint unitary commuting with $C([0,1],B)$.

Let $(E,(F_i)_i)$ be an approximate Kasparov bimodule over $(A,C([0,1],B))$ with its restrictions $(E_0,(F_{0i})_i)$ and $(E_1,F_{1i})_i)$ at zero and one.
The internal tensor product with $\omega$ is decomposed into
\[E\otimes_{C([0,1],B)}E_\omega=C([0,1],E_0\oplus E_m\oplus E_1),\qquad E_m:=E\otimes_{C([0,1])\otimes B}(H_m\otimes B).\]
On this, the left action of $A$ is constant and reduced by each summand of $E_0\oplus E_m\oplus E_1$.
Define $\widehat F_i:=F_i\otimes1$ and a Grassmann connection $\widehat F_\omega$ of $F_\omega$.
These operators $\widehat F_i$ and $\widehat F_\omega$ can be regarded as constant paths and a norm continuous path in $B(E_0\oplus E_m\oplus E_1)$ respectively.
For $\widehat F_i$, it is constant and reduced like the left action, and we can write
 \[\widehat F_i=F_{0i}\oplus F_{mi}\oplus F_{1i}\in B(E_0\oplus E_m\oplus E_1)\]
for some $F_{mi}\in B(E_m)$.
For $\widehat F_\omega$, we have the same formula for its endpoints $\widehat F_\omega(0)=0\oplus\widehat N_{\omega0}$ and $\widehat F_\omega(1)=\widehat N_{\omega1}\oplus0$ with degenerate Kasparov operators $\widehat N_{\omega0}\in B(E_m\oplus E_1)$ and $\widehat N_{\omega1}\in B(E_0\oplus E_m)$ commuting with the whole C$^*$-algebra $A$.
By the construction of Grassmann connections, $\widehat N_{\omega0}$ and $\widehat N_{\omega1}$ graded-commute with operators in $B(E)\otimes1$, especially $\widehat F_{0i}$ and $\widehat F_{1i}$ respectively.

For fixed $i$, with respect to $F_i$ under $A_i$, choosing a suitable separable fibering, because $F_\omega(t)$ is Kasparov under any separable C$^*$-subalgebra of $C([0,1],B)$ for all $t\in[0,1]$, we can construct a mediator $M_i\in B(E_0\oplus E_m\oplus E_1)$ for $\widehat F_i$ and $\{\widehat F_\omega(t):t\in[0,1]\}$.
Then, we have a Kasparov product $F_{pi}\in C([0,1],B(E_0\oplus E_m\oplus E_1))$ by
\[F_{pi}(t):=M_i^{\frac12}\widehat F_i+(1-M_i)^{\frac12}\widehat F_\omega(t),\qquad t\in[0,1].\]
It is an operator homotopy connecting its endpoints.
Then, since $P_0:=1-\widehat F_\omega(0)^2$ is the projection onto $E_0$ with $P_0\widehat F_i+\widehat F_\omega(0)=F_{0i}\oplus\widehat N_{\omega0}$, the computation
\[a^*[F_{pi}(0),P_0\widehat F_i+\widehat F_\omega(0)]a\equiv a^*2M_i^{\frac12}P_0\widehat F_i^2a+a^*2(1-M_i)^{\frac12}(1-P_0)a\ge0,\qquad a\in A_i,\]
modulo $K(E_0\oplus E_m\oplus E_1)$ shows that $F_{pi}(0)$ is Connes-Skandalis operator homotopic to $F_{0i}\oplus\widehat N_{\omega0}$.
Similarly $F_{pi}(1)$ is operator homotopic to $N_{\omega1}\oplus F_{1i}$.
If we define degenerate Kasparov bimodules $\varphi_0':=(E_m\oplus E_1,\widehat N_{\omega0})$ and $\varphi_1':=(E_0\oplus E_m,\widehat N_{\omega1})$, then $(E_0,(F_{0i})_i)\oplus\varphi_0'$ and $(E_1,(F_{1i})_i)\oplus\varphi_1'$ are operator homotopic, so $\varphi':=(\varphi_0'\oplus\varphi_1')^{\oplus\infty}$ completes the proof.
\end{proof}

\subsection{Complete additivity}
The classical KK-theory has a property called the countable additivity for the first variable, originally proved by Rosenberg, Theorem 1.12 of \cite{MR894590}.
We prove a generalization of this for uncountable sums.
Other than a subtlety in the addition of Kasparov operators that cannot be with norm addressed in \cite{MR894590}, we have a more critical issue on the norm topology.
Precisely speaking, when $F_i\in C([0,1],B(E_i))$ are uniformly bounded norm continuous paths of adjointable operators on Hilbert modules $E_i$ over a common C$^*$-algebra, the orthogonal sum $[0,1]\to B(\bigoplus_iE_i):t\mapsto\bigoplus_iF_i(t)$ is not a norm continuous path in general.
Contrary to ordinary Kasparov bimodules, approximate Kasparov bimodules need to satisfy Kasparov family conditions that Kasparov operators are mutually connected by norm continuous paths.
To construct a norm continuous sum of paths, their moduli of continuity must be controlled uniformly, meaning that the connecting operator homotopies in a Kasparov family cannot be arbitrary.
The following lemma is for this.

In the following lemma and theorem, for $E=E_+\boxplus E_-$ a graded Hilbert module, denote by $\overline E:=E_-\boxplus E_+$ the opposite of $E$, and let $\widehat E:=E^{\oplus\infty}\oplus\overline E{}^{\oplus\infty}$.

\begin{lemma}[Universally Lipschitz Kasparov family]\label{Lemma: Universally Lipschitz Kasparov family}
Let $A$ and $B$ be C$^*$-algebras.
Let $(E_0,(F_{0i})_i)$ be an approximate Kasparov bimodule over $(A,B)$ with $E_0=E_0B$ and $F_{0i}$ self-adjoint contractions.
Then, there is a homotopic approximate Kasparov bimodule $(\widehat E_0,(\widehat F_{0i})_i)$ over $(A,B)$ such that for every separable C$^*$-subalgebra $A_0\subset A$ there is $i_0$ satisfying that every two members of the Kasparov family $\{\widehat F_{0i}\}_{i\ge i_0}$ can be connected by an operator homotopy of Lipschitz constant at most $7\pi/3$ under $A_0$.
The construction commutes with direct sums and evaluation of a homotopy.
\end{lemma}
\begin{proof}
As in Section 17.6 in \cite{MR1656031}, we can first reduce to the case that each Kasparov operator is a self-adjoint unitary, precisely as follows.
Consider an approximate Kasparov bimodule $(E_1,(F_{1i})_i)$ over $(A,B)$ such that
\[E_1:=E_0\oplus\overline E_0,\qquad\varphi_1(a):=\varphi_0(a)\oplus0,\qquad F_{1i}:=\begin{pmatrix}
F_{0i}&(1-F_{0i}^2)^{1/2}\\(1-F_{0i}^2)^{1/2}&-F_{0i}
\end{pmatrix},\]
for $a\in A$, where $\varphi_0$ and $\varphi_1$ denote the left actions of $(E_0,(F_{0i})_i)$ and $(E_1,(F_{1i})_i)$ respectively.
By adding to $(E_0,(F_{0i})_i)$ an approximate Kasparov bimodule $(\overline E_0,(-F_{0i})_i)$ with zero left action, and the linear interpolation of off-diagonal components justified by $E_0=E_0B$, we can see $(E_0,(F_{0i})_i)$ is homotopic to $(E_1,(F_{1i})_i)$ with $F_{1i}$ self-adjoint unitaries.
Since the linear path can also be defined to the operator homotopies between Kasparov operators which exist by the Kasparov family condition, so the Kasparov operators of $(E_1,(F_{1i})_i)$ can be connected by a self-adjoint unitary operator homotopy under each separable C$^*$-subalgebra of $A$.

Define an approximate Kasparov bimodule $(E_2,(F_{2i})_i)$ over $(A,B)$ by
\[E_2:=E_1\oplus\bigoplus_{n=1}^\infty(\overline E_1\oplus E_1),\qquad\varphi_2(a):=\varphi_1(a)^{\oplus\infty},\qquad F_{2i}:=F_{1i}\oplus J^{\oplus\infty},\qquad J:=\begin{pmatrix}0&1\\1&0\end{pmatrix}\in M_2(B(E_1)),\]
for $a\in A$, where $\varphi_2$ denotes the left action.
We have $E_2\cong\widehat E_0$.
It is homotopic to $(E_1,(F_{1i})_i)$, and hence to $(E_0,(F_{0i})_i)$.
To prove it satisfies the Lipschitz condition, fix a separable C$^*$-subalgebra $A_0\subset A$ and take a norm continuous path $F_{1ii'}$ of self-adjoint unitary Kasparov operators connecting $F_{1i}$ and $F_{1i'}$ under $A_0$.
Subdivide to choose a finite sequence $U_1,\cdots,U_n$ on this path such that $U_0=F_{1i}$, $U_n=F_{1i'}$, and $\|U_k-U_{k-1}\|<1$ for all $1\le k\le n$.
Then, we will connect $F_{2i}$ and $F_{2i'}$ by performing the following deformations on the first $2n+1$ summands of $E_2$ as
\[U_0\oplus J^{\oplus n}\sim U_0\oplus\bigoplus_{k=1}^n((-U_k)\oplus U_k)\sim\bigoplus_{k=0}^{n-1}(U_k\oplus(-U_k))\oplus U_n\sim J^{\oplus n}\oplus U_n\sim U_n\oplus J^{\oplus n},\]
where we use $\sim$ to mean that there is a connecting operator homotopy between the two Kasparov operators under $A_0$.
We will also see that the four paths have Lipschitz constants at most $\pi/2$, $\pi/3$, $\pi/2$, and $\pi$, respectively, and their sum is $7\pi/3$.

The first and third paths use the formula to connect a self-adjoint unitary $U\oplus(-U)$ to $J$ in $M_2(B(E_1))$ as
\[s\mapsto\begin{pmatrix}c_sU&s_s\\s_s&c_s(-U)\end{pmatrix},\qquad s\in[0,1],\]
where $c_s:=\cos\frac{\pi s}2$ and $s_s:=\sin\frac{\pi s}2$.
For the second path for each $k$, since $\|U_k-U_{k-1}\|<1$ so that $\|1-U_{k-1}U_k\|<1$, we can define $h_k:=(2i)^{-1}\log(U_{k-1}U_k)$ with $\|h_k\|<\pi/6$ by the principal logarithm defined on the positive real part, and this gives a path $s\mapsto e^{ish_k}U_ke^{-ish_k}$ for $s\in[0,1]$ that connects from $U_k$ to $U_{k-1}$.
For the fourth path $J^{\oplus n}\oplus U_n\sim U_n\oplus J^{\oplus n}$ in $M_{2n+1}(B(E_1))$, we can consider a flipping self-adjoint unitary $w\in M_{2n+1}(\mathbb{C})$ such that $w_{ij}=1$ if $i+j=2n+2$ and $w_{ij}=0$ otherwise for $1\le i,j\le 2n+1$, then the adjoint of $s\mapsto e^{ish}$ for $h:=(\pi/2)(1-w)$ on $s\in[0,1]$ gives desired path, with Lipschitz constant at most $\pi$ computed by its differentiation $[h,-]=-(\pi/2)[w,-]$.
\end{proof}

\begin{theorem}[Complete additivity]
Let $A_j$ be a family of para-unital C$^*$-algebras, and let $B$ be a para-unital C$^*$-algebra.
Then, there is a natural isomorphism
\[KK^\mathrm{fc}\Bigl(\bigoplus_jA_j,B\Bigr)\to\prod_jKK^\mathrm{fc}(A_j,B)\]
\end{theorem}
\begin{proof}
(Surjectivity)
Together with the non-degenerate reduction (Theorem \ref{Theorem: Non-degenerate reduction}), choose an approximate Kasparov bimodule over $(A_j,B)$ for each $j$ which satisfies the Lipschitz condition in the previous lemma (Lemma \ref{Lemma: Universally Lipschitz Kasparov family}).
Then, the direct sum of all chosen approximate Kasparov bimodules with setting the product directed set for the indexing set of Kasparov operators is an approximate Kasparov bimodule over $(\bigoplus_jA_j,B)$, because the direct sum of paths sharing a universal Lipschitz constant is still Lipschitz continuous to satisfy the Kasparov family condition.

(Injectivity)
Suppose a non-degenerate approximate Kasparov bimodule $(E_0,(F_{0i})_i)$ over $(\bigoplus_jA_j,B)$ has the pullback onto each $A_j$ that vanishes in $KK^\mathrm{fc}(A_j,B)$, and $F_{0i}$ is a self-adjoint contraction for each $i$.
Then, there is an orthogonal family $P_j\in B(E_0)$ of projections such that $A_jE_0=P_jE_0$ by the non-degeneracy of left action.
For every $j$, we may assume $F_{0i}=\bigoplus_jF_{0ji}$ with $F_{0ji}:=P_jF_{0i}P_j$ for all $i$ by deforming $F_{0i}$ linearly, and the indexing set $\mathcal{I}$ for $i$ is obtained by the finite down set reduction with a cofinal system of separable C$^*$-subalgebras $(\bigoplus_jA_{ji})_i$ of $\bigoplus_jA_j$ such that $\{F_{0ji}\}_{i\ge i_0}$ is Kasparov under $A_{ji_0}$.

For fixed $j$, we want to connect $(P_jE_0,(F_{0ji})_i)$ to zero with operator homotopies by adding a common degenerate Kasparov bimodule $(E_{0j}',F_{0j}')$ over $(A_j,B)$.
Since $B$ is para-unital, and since the non-degenerate reduction (Theorem \ref{Theorem: Non-degenerate reduction}) allows us to assume homotopies have zero fibers at infinity, we can use the operator homotopy reduction (Proposition \ref{Proposition: Operator homotopy reduction}) for each homotopy move.
We can also use the fact that a degenerate approximate Kasparov bimodule can be connected to zero by a single homotopy constructed by the cone on which we use again operator homotopy reduction (Proposition \ref{Proposition: Operator homotopy reduction}) for degeneracy moves.
Now observing that gluing operator homotopies is not difficult by doing the same thing as in the proof that operator homotopic approximate Kasparov bimodules are homotopic (Proposition \ref{Proposition: Operator homotopic homotopic}), we obtain a degenerate Kasparov bimodule $(E_{0j}',F_{0j}')$ over $(A_j,B)$ by taking the direct sum of the degenerate Kasparov bimodules constructed above, and a single homotopy $(E_j,(F_{jl})_{l\in\mathcal{L}_j})$ with zero fiber at infinity over $(A_j,C([0,1],B))$, whose Kasparov operators can be normalized to be self-adjoint contractions, together with a monotone cofinal map $\mathcal{L}_j\to\mathcal{I}:l\mapsto i_j(l)$ such that $F_{jl}(0)=F_{0ji_j(l)}\oplus F_{0j}'$ and $F_{jl}(1)=F_{0j}'$ for all $l$, connecting $(P_jE_0,(F_{0ji})_i)\oplus(E_{0j}',F_{0j}')$ and $(E_{0j}',F_{0j}')$.

Applying the previous lemma (Lemma \ref{Lemma: Universally Lipschitz Kasparov family}) to each homotopy $(E_j,(F_{jl})_{l\in\mathcal{L}_j})$ to get $(\widehat E_j,(\widehat F_{jl})_{l\in\mathcal{L}_j})$, construct their direct sum $(\widehat E,\widehat F_l)$ indexed by $\mathcal{L}:=\{(i,(l_j)_j):i_j(l_j)\ge i\text{ for all }j\}\subset\mathcal{I}\times\prod_j\mathcal{L}_j$, which is an approximate Kasparov bimodule over $(\bigoplus_jA_j,C([0,1],B))$ by the Lipschitz condition.
At one it is a degenerate cycle, say $(\widehat E_0',\widehat F_0')$, and at zero it is $(\widehat E_0,(\bigoplus_j\widehat F_{0ji_j(l_j)})_{l\in\mathcal{L}})\oplus(\widehat E_0',\widehat F_0')$, and it is operator homotopic to $(\widehat E_0,(\bigoplus_j\widehat F_{0ji})_i)\oplus(\widehat E_0',\widehat F_0')$ because $\widehat F_{0ji_j(l)}$ and $\widehat F_{0ji}$ are operator homotopic under $A_{ji}$ with a Lipschitz constant independent of $j$ by the condition $i_j(l_j)\ge i$.
\end{proof}

\bibliographystyle{alpha}
\bibliography{bib}

\end{document}